\documentclass[10pt]{article}

\usepackage{src/preamble} 
\usepackage{src/macros} 
\usepackage{authblk}
\usepackage[normalem]{ulem}

\bibstyle{sort&compress}
\PassOptionsToPackage{sort&compress}{natbib}
\hypersetup{colorlinks, linkcolor=black!50, citecolor=black!50, urlcolor=black!50}
\allowdisplaybreaks
\crefname{equation}{}{}
\crefname{figure}{Figure}{Figure}

\title{Minimax Optimality of Classical Scaling\\ Under General Noise Conditions}
\author[1]{Siddharth Vishwanath} 
\author[1,2]{Ery Arias-Castro}
\affil[1]{\small Department of Mathematics, University of California, San Diego} 
\affil[2]{\small Halıcıoğlu Data Science Institute, University of California, San Diego}
\date{}

\begin{document}
\maketitle
\vspace*{-3em}
\begingroup
\let\clearpage\relax


\begin{abstract}
We establish the consistency of classical scaling under a broad class of noise models, encompassing many commonly studied cases in literature. Our approach requires only finite fourth moments of the noise, significantly weakening standard assumptions. We derive convergence rates for classical scaling and establish matching minimax lower bounds, demonstrating that classical scaling achieves minimax optimality in recovering the true configuration even when the input dissimilarities are corrupted by noise.
\end{abstract}



\section{Introduction}
\label{sec:introduction}

Given a matrix of dissimilarities, $\Del = (\del_{ij}) \in \R^{n \times n}$, which corresponds to the collection of pairwise dissimilarities between $n$ items, the objective of \textit{multidimensional scaling} is to represent the constituent items as an embedding $\qty{\h x_1, \h x_2, \cdots, \h x_n} \in \R^p$ in Euclidean space {for a given $p$}, such that the embedding {reproduces} the pairwise dissimilarities as closely as possible \citep{de198213}.  Multidimensional scaling plays an important role in unsupervised machine learning, and has found extensive applications across a range of diverse fields; we refer the reader to \cite{borg2007modern} and the references therein for a comprehensive overview.

In the classical multidimensional scaling setting, the matrix of dissimilarities, $\Del$, is assumed to be a \textit{Euclidean dissimilarity matrix} \citep{dattorro2010convex}, i.e., there exists a configuration $\qty{x_1, x_2, \cdots, x_n}\in \R^p$ such that $\del_{ij} := \norm{x_i - x_j}^2$. A classical result due to \cite{schoenberg1935remarks} establishes that $\Del$ is a Euclidean dissimilarity matrix if and only if the \textit{double centering} transformation of $\Del$, given by
\begin{align}
    \Delc := -\frac{1}{2}H\Del H, \qq{where} H = I - \frac{1}{n}\onev\onev\tr,
\end{align}
is positive {semi-}definite. Using $\X \in \R^{n \times p}$ to denote the matrix of {latent configurations} and $\Del\equiv \Del(\X)$ its associated Euclidean dissimilarity matrix, it is easy to see that $\Delc = \X\X\tr$. Schoenberg's result, essentially guarantees that the converse is true as well. \cite{gower1982euclidean} later notes that ``\emph{[Schoenberg's result] was a surprisingly late date for such a classical result in Euclidean geometry}''. This insight underpins the \textit{classical scaling algorithm} due to \cite{young1938discussion} and formalized by \cite{torgerson1952multidimensional,torgerson1958theory} and \cite{gower1966some}, wherein the embedding $\hX \in \Rnp$ is obtained by optimizing the following objective in the $\ell_2$-operator norm:
\begin{align}
    \hX = \argmin_{\Y \in \Rnp}\opnorm\big{ \Y\Y\tr - \Delc }.\label{eq:cs-objective}
\end{align}
In practice, the procedure for solving $\hX$ amounts to performing a spectral decomposition of the double centered matrix $\Delc$, as outlined in Algorithm~\ref{alg:cs}. Here, one can only hope to recover the true configuration up to an arbitrary rigid transformation. Indeed, for any rigid transformation $g \in \euc{p}$, $\Del(\X) = \Del( g(\X) )$. In other words, for $g(x) = Qx + w$ where $Q \in \orth{p}$ is an orthogonal matrix and $w \in \R^p$, it follows that $\norm{g(x_i) - g(x_j)}^2 = \norm{x_i - x_j}^2$, and the configuration $\qty{g(x_1), g(x_2), \cdots, g(x_n)}$ admits the same dissimilarity matrix $\Del$. It is, therefore, instructive to consider a loss function of the form
\begin{align}
    \loss(\hX, \X) = \min_{g \in \euc{p}} \norm\big{\hX - g(\X)}_\dagger,\label{eq:reconstruction-error}
\end{align}
as a metric to assess the quality of recovery in classical scaling, where $\norm{\cdot}_\dagger$ is a suitable norm on the space of $n \times p$ configuration matrices. The loss in \cref{eq:reconstruction-error} is often referred to as the \textit{reconstruction error} of the embedding. Some common choices for $\smallnorm{\hX - \X}_\dagger$ include the Frobenius norm, $\smallnorm{\hX - \X}_F$ and the $\ell\ttinft$-operator norm $\smallnorm{\hX - \X}\ttinft$. Respectively, these lead to the \textit{root mean-squared error loss} $\loss\rmse$ and the \textit{uniform loss} $\loss\ttinft$, given by
\begin{align}
    \loss\rmse(\hX, \X) = \qty(\min_{g \in \euc{p}} {\frac{1}{n}\sum_{i=1}^n\smallnorm{\hx_i - g(X_i)}^2})^{1/2} \qq{and} \loss\ttinft(\hX, \X) = \min_{g \in \euc{p}} \max_{i \in [n]}\smallnorm{\hx_i - g(X_i)}.\label{eq:losses}
\end{align}

\begin{algorithm}[t]
    \begin{minipage}{0.96\linewidth}
    \SetAlgoLined
    \caption{Classical Scaling}\label{alg:cs}
    \textbf{Input:} Dissimilarity matrix $D$ and embedding dimension $p$\\
    \textbf{Output:} Configuration $\hX = \mathsf{CS}(D, p)$
    \begin{algorithmic}[1]
        \State \textbf{Double centering of $D$:}
        $$
            \Dc = -\frac{1}{2} H D H, \qq{where} H = I - \frac{1}{n} J.
        $$
        \State\label{algline:cs-spectral-decomposition} \textbf{Spectral decomposition of $\Dc$:}
        \begin{align}
            \Dc = \hU \hL \hU\tr + \h{V}_\perp \hL_\perp \h{V}_\perp\tr,\label{eq:cs-spectral-decomposition}
        \end{align}
        where $\hL = \diag({\h\lambda_1, \h\lambda_2, \cdots, \h\lambda_p})$ is the $p\times p$ diagonal matrix of the top-$p$ eigenvalues of $\Dc$ and $\hU = [\h{u}_1\,\h{u}_2\,\cdots\,\h{u}_p]\tr \in \R^{n \times p}$ are the corresponding eigenvectors.
        \State \textbf{Spectral Embedding:}
        \begin{align}
            \hX = \hU \hL^{1/2}.\label{eq:hX}
        \end{align}
    \end{algorithmic}
\end{minipage}
\end{algorithm}

{If the latent configuration $\X \in \Rnp$ is in \textit{general position}\footnote{i.e., no subset of $(p+1)$ points lies within a $(p-1)$-dimensional affine subspace and $\rank(\X)=p$.}} and the observed dissimilarities are the pairwise distances between $\X$, then it is well known that the output of \cref{alg:cs} recovers the true configuration exactly (up to an arbitrary rigid transformation). In most practical applications, however, the observed dissimilarities are subject to measurement errors and/or corrupted by noise, and the observed dissimilarity matrix is given by
$$
D = \Del + \Eps.
$$ 
Despite its long history and widespread use, relatively little was known about the performance of classical scaling under perturbations and noise until recently. In this work, we aim to provide a detailed statistical analysis of classical scaling in the presence of noise.

\subsection{Related Work}
\label{sec:related-work}

\textit{Perturbation bounds}.\quad A detailed sensitivity analysis of classical scaling was first studied in a series of works by \cite{sibson1978studies,sibson1979studies,sibson1981studies}. For a small local perturbation {$\tilde{\Del} \asymp \Del + t\Phi$}, they show that {$\smallnorm{\tilde\X - \X}_2 \asymp C t \smallnorm{\Phi}_2$} via local first-order approximations. \cite{de2004sparse} subsequently extended the generality of these first order perturbations. More recently, \cite{arias2020perturbation} provide refined perturbation bounds and improve on earlier results in two ways. First, they show that the output of \cref{alg:cs} is stable with respect to arbitrary perturbations in the Schatten $r$-norm, i.e., ${\min_{g \in \euc{p}}\smallnorm{\tilde\X - g(\X)}_{r} \lesssim \smallnorm{\Delc(\tilde\X) - \Delc(\X)}_r^2}$. Second, by characterizing the stability of the orthogonal Procrustes problem, their  perturbation bounds explicitly characterize the arbitrary rigid transformation $g \in \euc{p}$ that minimizes the reconstruction error. As we shall see later on, this will play a key role in our analyses. However, the perturbation bounds in these works operate in the deterministic setting, and do not account for the randomness arising from $\X$ or the noise $\Eps$.

\textit{Noisy dissimilarities}.\quad When $\Eps$ is random, \cite{javanmard2013localization} study the problem of localization from noisy measurements, i.e., where the observed dissimilarity $D$ is both corrupted by noise and incomplete measurements. Their method considers a semidefinite relaxation of the objective in \cref{eq:cs-objective}, and establish conditions under which their solution recovers the true configuration with the $\loss\rmse$ loss. A similar problem is also considered by \cite{chatterjee2015matrix} in the context of matrix completion. Both of these works focus on the case where the noise is bounded, and provide guarantees on the reconstruction error when this bound is known \textit{a priori}. \cite{zhang2016distance} consider the case where the dissimilarities are completely observed but corrupted by \iid{} sub-Gaussian noise; in this setting (and, although the authors don't explicitly make the connection to the work of  \citealp{javanmard2013localization}), they show that the output from the same semidefinite relaxation is consistent in the $\loss\rmse$ loss.

\textit{Classical scaling with noisy dissimilarities.}\quad To the best of our knowledge, detailed statistical analyses of classical scaling in the presence of noise have only been considered recently. \cite{peterfreund2021multidimensional} and \cite{little2023analysis} consider the spiked model where: $d_{ij} = \norm{\tilde{x}_i - \tilde{x}_j}^2$ for ${x}'_i = x_i + z_i$ where $z_i$ are some (possibly high dimensional) noise. In other words, { the dissimilarities are the Euclidean distances associated with a noisy (low-dimensional) configuration} $\tilde{\X} = \X + \mathrm{Z}$. 
In the high-dimensional regime, and when the eigenvalues used in classical scaling are appropriately thresholded, \cite{peterfreund2021multidimensional} establish consistency in the $\loss\rmse$ loss. In a similar vein, \cite{little2023analysis} shed light on the signal-to-noise thresholds for optimal clustering using the output of classical scaling in the high-dimensional regime by establishing the convergence rates in $\loss_{\max}$ ({{where the matrix $\norm{\cdot}_{\max}$-norm is used in the reconstruction error from \cref{eq:reconstruction-error}}}). Notably, in both of these works, the noise is assumed to manifest in the configuration which generates the dissimilarities and not in the observed dissimilarities themselves. To this end, and, perhaps most relevant to the present work, \cite{li2020central} provide the first detailed analysis of the consistency of classical scaling when $D = \Del + \Eps$; they consider three types of noise models---all of which, in some way, rely on the noise being \iid{} and sub-Gaussian. Under these assumptions and for an appropriate rigid transformation $g \in \euc{p}$, they show that in the $\loss\avg$ loss, the output of classical scaling is $\sqrt{n}$-consistent. This allows them to establish a central limit theorem for the output of classical scaling: the embedding produced by classical scaling asymptotically converges to a Gaussian distribution centered on $g(\X)$---the true configuration after a rigid transformation. These contributions form the main backdrop for the present work. 

\subsection{Contributions}
\label{sec:contributions}

While these results described in \cref{sec:related-work} establish consistency of classical scaling, it is unclear whether their conclusions only hold for the specific type of noise models they consider, or if they hold more generally. Moreover, it is not known if the output of classical scaling is optimal in recovering the true configuration in the presence of noise. Building upon the existing literature, our work makes the following key contributions:

\begin{enumerate}[label=(\roman*)]
    \item We consider a flexible noise framework where the noise terms $\eps_{ij}$ in the observed dissimilarity matrix $D$ can be expressed as $\eps_{ij} = \Psi(\del_{ij}, \xi_{ij})$ for some function $\Psi$, with $\xi_{ij}$ random and independent of $\del_{ij}$ (\cref{noisy-setting}). This allows for cases where $\eps_{ij}$ are neither independent nor identical, and encompasses noise models previously considered in literature (see \cref{tab:noise-models}) and others. For instance, this includes multiplicative noise models of the form $\eps_{ij} = (1+\del_{ij}) \xi_{ij}$ for some $\xi_{ij}$ independent of $\del_{ij}$. Importantly, $\eps_{ij}$ is allowed to be unbounded, heavy-tailed and depend on $\del_{ij}$, requiring only finite moments up to order $q > 4$.
    \item Under the general noise framework, we establish consistency for the output of classical scaling from the noisy dissimilarity matrices, and establish their convergence rates in the $\ell_2$-operator norm (\cref{thm:X-opnorm}), the root mean-squared loss $\loss\rmse$ (\cref{cor:frobenius-norm}), and uniform convergence under the $\loss\ttinft$ loss (\cref{thm:X-ttinf}). The results hold for both fixed and random configurations $\X$---where the rows of $\X$ are sampled \iid{} from a sufficiently regular distribution (\cref{cor:iid}). In relation to previous work, the results make use of probabilistic symmetrization along with truncation to handle the heavy-tailed noise; see \cref{sec:consistency} for further discussion.
    \item We establish matching lower bounds on the minimax risk of estimating the configuration~$\X$ up to rigid transformation from noisy dissimilarities under the $\loss\rmse$ loss (\cref{thm:minimax-frobenius}) and under the $\loss\ttinft$ loss (\cref{thm:minimax-ttinf}). This shows that the output of classical scaling is minimax optimal in recovering the true configuration even when the input dissimilarities are subject to noise.
\end{enumerate}

\noindent With this background, the following result is a corollary
of more general results that are established in the paper. 

\begingroup
\renewcommand*{\thecorollary}{\!\!\!\:}
\begin{corollary}
    For fixed $p$, suppose the rows of $\X \in \Rnp$ are sampled \iid{} from a probability distribution $F$ with $\diam(\supp(F)) \le \Rx$ and $\frac{1}{\kappa^2}I_p \preccurlyeq \cov(F) \preccurlyeq \kappa^2 I_p$ where $\kappa, \Rx > 0$ and let $\Del$ be the pairwise Euclidean dissimilarity matrix. Let $D = \Del + \Eps$ be the observed dissimilarity matrix, where the noise $\eps_{ij}$ satisfies $\E[|\eps_{ij}|^q \mid \del_{ij}] \le \sigma^q$ for $q > 4$ in addition to some mild regularity conditions described in \ref{assumption:independence}--\ref{assumption:moments}. Let $\hX = \cs{D, p}$ be the configuration obtained using \cref{alg:cs}. Then, there exists $\enn(q, \sigma, \xpar)$ such that for all $n > \enn(q, \sigma, \xpar)$ and with high probability,
    \begin{align}
        \loss\rmse(\hX, \X) \asymp \frac{\sigma\kappa}{\sqrt{n}} \qq{and} \loss\ttinft(\hX, \X) \asymp \boldsymbol{c}(\kappa, \Rx) \sigma \sqrt{\frac{\log{n}}{{n}}},\label{eq:main-result}
    \end{align}
    where $\boldsymbol{c}(\kappa, \Rx)$ is a constant depending only on $\kappa, \Rx$.
\end{corollary}
\endgroup

The remainder of the paper is organized as follows. In \cref{sec:background}, we describe the setting we consider in this work and introduce the main assumptions. In \cref{sec:consistency}, we derive the convergence rates for classical scaling under noise. \cref{sec:lower-bound} establishes the minimax lower bounds for the problem of recovering the true configuration from noisy dissimilarities. In the interest of clarity, we defer the proofs of the main results to \cref{sec:proofs}. Auxiliary results and proofs to more technical lemmas are collected in \cref{sec:auxiliary,sec:toolkit}.

\textbf{Notation.}\quad We use $\onev_k, \zerov_k, e^k_i \in \R^k$ to denote the vector of ones, the vector of zeros, and the $i$th canonical basis vector, respectively; and $I_k, J_k, \O_k$ to denote the $k \times k$ identity matrix, matrix of all ones, and matrix of all zeros, respectively. In most cases, we suppress the subscript $k$ when the dimension is clear from the context. For a vector $x \in \Rp$, $\norm{x}$ denotes the $\ell_2$-norm of $x$, and for a matrix $A \in \Rnp$, 
$$
\opnorm{A}\!\!:= \max_{\norm{x}=1}\norm{Ax}
,\quad 
\ttinf{A}\!\!:= \max_{\norm{x}=1}\norm\big{Ax}_\infty
,\quad
\norm{A}_\infty\!\!:= \max_{\norm{x}_\infty=1}\norm\big{Ax}_\infty
\qc{\text{and }} 
\norm{A}_F\!\!:= \sqrt{\trace({A\tr A})},
$$ 
denote the $\ell_2$-operator norm, the $\ell_{2\to\infty}$-operator norm, the $\ell_{\infty}$-operator norm, and the Frobenius norm of $A$, respectively. $H := I - J/n$ denotes the centering matrix, and $A\pinv$ denotes the Moore-Penrose pseudoinverse of a full rank matrix $A$. For two matrices $A, B \in \R^{n \times p}$, $A \circ B$ denotes the Hadamard product of $A$ and $B$. For $A \in \R^{m \times k}$, $s_1(A) \ge \cdots \ge s_{{m \wedge k}}(A)$ are the singular values of $A$ and for a square matrix $B \in \R^{k \times k}$, $\lambda_1(B) \ge \cdots \ge \lambda_k(B)$ are its eigenvalues. Lastly, $\orth{k} := \qty{Q \in \R^{k \times k}: QQ\tr = Q\tr Q = I_k}$ denotes the group of orthogonal matrices, and $\euc{k} := \qty{g: g(x) = Qx + a,\, Q \in \orth{k},\, a \in \R^k}$ denotes the group of rigid transformations in $\R^k$.

In what follows, we use standard Bachmann-Landau notation; we write {$a_n = O(b_n)$ and $a_n \lesssim b_n$ (equivalently, $b_n = \Omega(a_n)$ and $b_n \gtrsim a_n$)}, if there exists a constant $C > 0$ such that $\abs{a_n} \le C\abs{b_n}$ for all $n > N_C$, and $b_n = o(a_n)$ (equivalently, $b_n \ll a_n$) if $\lim_{n} |b_n/a_n| = 0$. \cref{tab:notation} in \cref{sec:toolkit} provides a summary of the notation used in this work.



\section{{Setting}}\label{sec:background}

Given an input dissimilarity matrix, the objective of classical multidimensional scaling is to recover a configuration in $\R^p$ which faithfully reproduces the pairwise dissimilarities in the original input. As noted in \Cref{sec:introduction}, the general setting assumes that the input dissimilarities arise from a realizable configuration.
\begin{definition}[Realizable setting]
    The dissimilarity matrix $\Del = (\del_{ij}) \in \R^{n \times n}$ is realizable if there exists a latent configuration $\X \in \Rnp$, such that $\del_{ij} = \norm{X_i - X_j}^2$ for all $i, j \in [n]$, i.e., 
    \begin{align}
    \Del(\X) = \diag(\X\X\tr)\onev\tr + \onev \diag(\X\X\tr)\tr - 2\X\X\tr.\label{eq:Del-matrix-form}
\end{align}
\end{definition}
In the realizable setting, it is well known that the classical scaling algorithm exactly recovers the true configuration up to an arbitrary rigid transformation, i.e., for $\hX = \cs{\Del, p}$ from \cref{alg:cs}, there exists $Q \in \orth{p}$ such that $\hX = (H\X) Q$. Throughout, we consider the setting where the dimension $p$ is fixed and known. In practice, however, the choice of the latent dimension, $\h{p}$, can be chosen reliably using various methods, e.g., a screeplot of the eigenvalues of $\Dc$, or other thresholding methods. Moreover, the assumption that $p < n$ is encodes the hypothesis that the true ``signal'' we hope to recover is low-dimensional, and stays true to the spirit of classical multidimensional scaling. We refer the reader to \cite{peterfreund2021multidimensional} for results on classical scaling in the high-dimensional regime. In the noisy realizable setting, we assume that the observed dissimilarities are subject to noise.

\begin{definition}[Noisy realizable setting]\label{noisy-setting}
    The observed dissimilarities are of the form 
    \begin{align}
        d_{ij} = \del_{ij} + \eps_{ij}, \qq{for } \eps_{ij} = \Psi(\del_{ij}, \xi_{ij}),
    \end{align}
    where $\xi_{ij}$ is random and $\Psi: \R \times \R \to \R$ is a noise model satisfying\footnote{The condition $\Psi(0, 0) = 0$ ensures that the resulting $\Eps = (\eps_{ij})$ is a hollow matrix.} $\Psi(0, 0) = 0$. We let $D = (d_{ij})$, $\Del = (\delta_{ij})$, $\Eps = (\eps_{ij})$, and $\Xi = (\xi_{ij})$. In particular, $D = \Del + \Eps$.
\end{definition}

It is easy to see that when $\Xi = (\xi_{ij})$ is an $n \times n$ symmetric hollow matrix, then the observed dissimilarity matrix, $D$, is also symmetric and hollow. The generality of $\Psi$ allows us to consider a wide range of noise models. In the noisy realizable setting, we are interested in understanding how well $\hX = \cs{D, p}$ recovers the configuration $\X$ {(up to a rigid transformation)} in the presence of noise. 

\textbf{Assumptions.}\quad In our analyses, we make the following assumptions on the latent configuration $\X$, the noise $\Eps$, and the noise model $\Psi$. 

\begin{enumerate}[label=(\textbf{A}$_{\arabic*}$), ref=\textup{(\textbf{A}$_{\arabic*}$)}]
    \item\label{assumption:compact} For $\Rx > 0$ and $\kappa > 1$, the configuration matrix $\X \in \bbX(\kappa, \Rx) \subset \Rnp$ where
    \begin{align}
    \bbX(\kappa, \Rx) := \Bigg\{\X \in \Rnp:
        \frac{1}{\kappa} \le s_p\qty(\frac{HX}{\sqrt{n}}) \le s_1\qty(\frac{HX}{\sqrt{n}}) \le \kappa \qq{and} \ttinf{H\X} \le \Rx\Bigg\},\label{eq:bbX}
    \end{align}
    \item\label{assumption:independence} $\Xi = (\xi_{ij})$ is a random, symmetric hollow matrix with \iid{} entries and $\gamma := \E(\xi_{ij})$.
    \item\label{assumption:moments} There exists $\sigma > 0$ and $q > 4$, such that 
    $
    \qty\big[\E(\abs{\eps_{ij}}^q | \del_{ij})]^{1/q} \le \sigma \,\text{ for all }\, i, j \in [n].
    $
    \item\label{assumption:expectation} The noise model $\Eps = \Psi(\Del, \Xi)$ and the configuration $\X \in \Rnp$ are such that 
    $$
    \norm{H\,\E(\Eps)\,H}_{\infty} \lesssim \sqrt{n}.
    $$
\end{enumerate}

\begin{table}[t!]
    \centering\small
    \caption{Noise models for the observed dissimilarity matrix $D$. The last two columns describe the sufficient conditions on $\Xi = (\xi_{ij})$ in order for \ref{assumption:expectation} to hold when $\X$ is fixed or random with \iid{} rows.}
        \begin{tabular}{llc|cc}
        \toprule
        \textbf{Setting} & \textbf{Noise Model} & $\Eps=\Psi(\Del, \Xi)$ & Fixed $\X$ & Random $\X$ \\
        \toprule
        {1}. Additive & $\phantom{\sqrt{}}d_{ij} = \del_{ij} + \xi_{ij}$ & $\Xi$  & $\E(\xi_{ij})=\gamma$ & $\E(\xi_{ij})=\gamma$ \\[5pt]
        {2}. Multiplicative & $\phantom{\sqrt{}}d_{ij} = \del_{ij}(1 + \xi_{ij})$ & $\Del \circ \Xi$ & $\E(\xi_{ij})=0$ & $\E(\xi_{ij})=\gamma$ \\[5pt]
        {3}. Absolute Additive & $\sqrt{d}_{ij} = \abs\big{\sqrt \del_{ij} + \xi_{ij}}$  & $\Xi\sq + 2 \sqrt\Del \circ \Xi$  & $\E(\xi_{ij})=0$ & $\E(\xi_{ij})=\gamma$ \\[5pt]
        {4}. Absolute Multiplicative & $\sqrt{d}_{ij} = \abs\big{\sqrt \del_{ij} (1 + \xi_{ij})}$  & $\Del \circ (2\Xi + \Xi\sq)$ & $\cross$ & $\E(\xi_{ij})=\gamma$ \\[5pt]
        {5}. Thresholded Additive & $\phantom{\sqrt{}}d_{ij} = \max\qty{0, \del_{ij} + \xi_{ij}}$ & $\max\qty{-\Del, \Xi}$ & $\cross$ & $\E(\xi_{ij})=\gamma$ \\[5pt]
        {6}. Thresholded Multiplicative & $\phantom{\sqrt{}}d_{ij} = \max\qty{0, \del_{ij}(1 + \eps_{ij})}$ & $\Del \circ \max\qty{-J, \Xi}$ & $\cross$ & $\E(\xi_{ij})=\gamma$ \\
        \bottomrule
        \end{tabular}
    \label{tab:noise-models}
\end{table}

{
We note that the conditions on $\X$ and $\Psi$ here are rather mild. For starters, assumption~\ref{assumption:compact} is standard in literature, and requires the rows of $\X$ to be bounded, and for the sample covariance matrix of $\X$ to have bounded eigenvalues. In particular, from standard matrix concentration bounds \citep[e.g.,][]{vershynin2018high}, assumption~\ref{assumption:compact} is satisfied with high probability if the rows of $\X$ are sampled \iid{} from a sufficiently regular probability distribution on $\Rp$, as described in the following lemma. The proof can be found in \cref{proof:lemma:iid}.

\begin{lemma}\label{lemma:iid}
    Let $\F(\xpar)$ be a collection of probability distributions on $\Rp$ given by
    \begin{align}
        \F(\xpar) := \qty{ F : \diam(\supp(F)) \le \Rx \qq{and} \frac{1}{\kappa^2}I_p \preccurlyeq \cov(F) \preccurlyeq \kappa^2 I_p },\label{eq:F}
    \end{align}
    and let $\alpha_n := \kappa^2\Rx^2\sqrt{\log{n}/n}$. If the rows of $\X \in \Rnp$ are observed \iid{} from $F \in \F(\kappa, \Rx)$, then there exists $\enn_0(\kappa, \Rx)$ such that for all $n > \enn_0(\kappa, \Rx)$,
    $$
    \X \in \bbX\qty\Big( \kappa\qty\big(1 + O(\alpha_n) ), \Rx)
    $$
    with probability greater than $1 - O(n^{-2})$.
\end{lemma}

Assumption \ref{assumption:independence}, on the other hand, imposes restrictions on the dependence structure between the noise $\Eps$ and the latent configuration $\X$. In particular, conditional on $\Del$, the elements of $\Eps$ are independent, and rules out pathological cases, e.g., $\Eps = -\Del$ where the observed dissimilarity matrix becomes $\O$. In contrast to \cite{li2020central} and \cite{little2023analysis}---where the errors $\eps_{ij}$ are assumed to be sub-Gaussian---assumption~{\ref{assumption:moments}} substantially relaxes the assumptions on $\eps_{ij}$, and only require finite moments up to order $q > 4$. Similar relaxations to noise models have appeared in recent literature on supervised learning (\citealp[e.g.,][]{kuchibhotla2022least,kuchibhotla2022moving,bakhshizadeh2023sharp}). Unlike these works, however, we require at least $q > 4$ moments, which is necessary for controlling the norm of the random matrices (\citealp{bai1988note}, \citealp{silverstein1989weak}). In particular, \cite{bai1988note} show that if $\Eps = (\eps_{ij})$ is a random matrix with mean-zero \iid{} entries and $\E(\eps_{ij}^4) = \infty$, then $\opnorm{\Eps} = \infty$ a.s. 

Lastly, assumption~\ref{assumption:expectation} imposes a restriction on the noise model $\Psi$ depending on the nature of the configuration $\X$. The following lemma highlights two key settings where \ref{assumption:expectation} holds.

\begin{lemma}\label{lemma:noise-models}
    For $D = \Del(\X) + \Eps$ in the noisy realizable setting of \cref{noisy-setting}, assumption \ref{assumption:expectation} is met if any of the following conditions hold:
    \begin{enumerate}[label=\textup{(\roman*)}]
        \item\label{lem:noise-1} The rows $X_1, X_2, \dots, X_n \in \Rp$ of $\X \in \Rnp$ are observed \iid{} from $F \in \F(\xpar)$.
        \item\label{lem:noise-2} $\Xi = (\xi_{ij})$ are independent, symmetric, mean-zero random variables, and for any $K, L > 0$ there exist a collection bounded functions $g_1, \dots, g_{K}$ is such that
        $$
        \Psi(\del_{ij}, \xi_{ij}) = \sum_{k=1}^K g_k(\del_{ij})\xi_{ij}^{2k+1} + \sum_{\ell=1}^{L} \xi_{ij}^{\ell}.
        $$
    \end{enumerate}
\end{lemma}
The proof of \cref{lemma:noise-models} is provided in \cref{proof:lemma:noise-models}. Some examples of noise models satisfying these conditions are listed in \Cref{tab:noise-models}. The additive and absolute additive models in Settings~1~and~3 correspond to the noise models in \citet[Sections~2.1~\textit{\&}~2.2]{li2020central}. In \cref{sec:consistency}, we analyze the performance of classical scaling under these assumptions.
}



\section{Performance Analysis of Classical Scaling}
\label{sec:consistency}

In the \textit{realizable setting}, given a pairwise dissimilarity matrix $\Del$ generated from $\X$, the classical scaling algorithm begins by double centering $\Del$. Schoenberg's result establishes that
\begin{align}
    \Delc  = -\frac{1}{2}H\Del H = (H\X)(H\X)\tr
\end{align}
is positive semi-definite. This can be easily seen from \cref{eq:Del-matrix-form} by noting that  $\onev$ is in the nullspace of $H$, and that the Gram matrix $(H\X)(H\X)\tr$ is always positive semi-definite. Moreover, $\Delc$ is a rank-$p$ matrix, and shares the same eigenvalues as the sample covariance matrix $(H\X)\tr(H\X)$. Therefore, $\Lambda_\perp = \O$, and, up to an arbitrary orthogonal transformation $Q \in \orth{p}$, 
\begin{align}
    H\X Q\tr = U\Lambda^{1/2},\label{eq:cs-exact-recovery}
\end{align}
where $U\Lambda U\tr$ is rank-$p$ component from the spectral decomposition of $\Delc$; see \cref{eq:cs-spectral-decomposition} from \cref{algline:cs-spectral-decomposition} of \cref{alg:cs}. By noting that $\cs{\Del, p}$ is exactly the term on the r.h.s., it is easy to see that for $\bx = \X\tr\onev_n/n$ and the rigid transformation $g(v) = Q(v - \bx)$, the classical scaling algorithm outputs $\hX = g(\X)$. When $D = \Del + \Eps$ in the \textit{noisy realizable setting}, the double centered matrix $\Dc$ can be expressed as
\begin{align}\label{eq:Dc}
    \Dc = \Delc - \half H\Eps H.
\end{align}
By letting $\hL$ denote the diagonal matrix of the top-$p$ nonnegative eigenvalues of $\Dc$, and $\hU$ to denote the corresponding eigenvectors, we have $\Dc = \hU \hL \hU\tr + \h V_\perp \hL_\perp \h V_\perp$ as per \cref{eq:cs-spectral-decomposition} of \cref{alg:cs}, and $\hX = \cs{D, p} = \hU \hL^{1/2}$. In order to prove the consistency of classical scaling in the noisy realizable setting, we need a handle on the deviation of $\Dc$ from $\Delc$ in the $\ell_2$-operator norm, which is described in the following result.

\begin{theorem}\label{lem:D-opnorm}
    Suppose $\X \in \bbX(\xpar)$ and $\D = \Del + \Eps$ in the noisy realizable setting of \cref{noisy-setting}. Then, under assumptions~\ref{assumption:compact}--\ref{assumption:expectation}, 
    \begin{align}
        \E\qty\Big[ \opnorm{\Dc - \Delc} ] \lesssim \sigma\sqrt{n}.\label{eq:D-opnorm-2}
    \end{align}
    Moreover, for all $0 < r < (q-4)/2$ there exists $\enn_1 \equiv \enn_1(r, q)$ such that for all $n > \enn_1(r, q)$,
    \begin{align}
        \opnorm{ \Dc - \Delc } &\lesssim \sigma\sqrt{n}\label{eq:D-opnorm-1}
    \end{align}
    with probability greater than $1 - O(n^{-2} + n^{-r})$.
\end{theorem}

The proof of \cref{lem:D-opnorm} is deferred to \cref{proof:lem:D-opnorm}; the bound on the expected value in \cref{eq:D-opnorm-2} follows from a straightforward use of symmetrization followed by an application of \citet[Theorem~2]{latala2005some}. On the other hand, the bound on the tail probability in \cref{eq:D-opnorm-1} requires a symmetrization argument for the tail probability associated with the $\ell_2$-operator norm (\cref{lem:opnorm-symmetrization}) followed by an application of a quantitative analogue of \citet[Theorem~2]{latala2005some} which is provided in \cref{prop:latala-quantitative}.

In other words, \cref{lem:D-opnorm} guarantees that typical fluctuations of $\Dc$ from $\Delc$ are of order $O(\sqrt{n})$. As noted in \cref{sec:background}, the existence of at least the fourth moment is necessary for controlling $\opnorm{\Dc - \Delc}$. In particular (and, in the spirit of examining Bernstein-type deviations) if one were to carefully account for the higher moments of the noise; then, from \cref{eq:moment-bounds} and \cref{eq:circled-3}, the precise bound for \cref{eq:D-opnorm-1} can be written as
\begin{align}
    \opnorm{ \Dc - \Delc } \le c_1 M_2\sqrt{n} + c_2 M_4\sqrt{n} + c_3 M_q n^{\frac{(r+2)}{q}}\log{n},\label{eq:D-opnorm-full}
\end{align}
where $\max_{ij}\E[ \abs{\eps_{ij}}^k \mid \del_{ij} ] \le M_k$ is a bound on the $k$-th conditional moment of the noise $\eps_{ij}$ and $c_1, c_2, c_3 > 0$ are absolute constants. Even in the simplest setting of the additive noise model in \cref{tab:noise-models} where $\Eps = \Xi$, \citet[Theorem~1.6]{tao2010random} show that fluctuations of order $O(\sqrt{n})$ for $\E\opnorm{\Eps}$ depend on the fourth moment $M_4$ appearing in \cref{eq:D-opnorm-full}.

Although \cref{eq:D-opnorm-1} by itself doesn't provide any guarantees for the embedded configuration, it will be the main workhorse in establishing the consistency results in this section. Building on this, we now proceed to study the reconstruction error of $\hX$ in the $\ell_2$-operator norm.

\begin{theorem}\label{thm:X-opnorm}
    Suppose $\X \in \bbX(\xpar)$ and $D = \Del + \Eps$ in the noisy realizable setting of \Cref{noisy-setting} satisfying assumptions~\ref{assumption:compact}--\ref{assumption:expectation}, and let $\hX = \cs{D, p}$. Then, for all $0 < r < (q-4)/2$ there exists $\enn_2(r, q, \sigma, \kappa)$ such that for all $n > \enn_2(r, q, \sigma, \kappa)$,
    \begin{align}
        \min_{g \in \euc{p}}\opnorm\big{\hX - g(\X)} \lesssim \sigma\kappa
    \end{align}
    with probability greater than $1 - O(n^{-2} + n^{-r})$.
\end{theorem}

The proof of \cref{thm:X-opnorm} is provided in \cref{proof:thm:X-opnorm}. The bound in \cref{thm:X-opnorm} relies on the perturbation result in \citet[Theorem~1]{arias2020perturbation} which bounds the reconstruction error  $\min_g\smallnorm{\hX - g(\X)}_2$ in terms of $\opnorm{\Dc - \Delc}$ and $\smallnorm{\X\pinv}_2$. The bound on the fluctuations of $\opnorm{\Dc - \Delc}$ in tail probability then leads to a control on the reconstruction error of $\hX$.

The rigid transformation $g \in \euc{p}$ which provides the guarantee in \cref{thm:X-opnorm} is obtained from the solution to the orthogonal Procrustes problem. Specifically, the orthogonal Procrustes problem finds the optimal $Q\in \orth{p}$ which aligns the subspaces of $\hU$ and $U$ by solving:
\begin{align}
    Q_* := \argmin_{Q \in \orth{p}}\norm\big{\h{U} - UQ}^2_F\label{eq:procrustes}
\end{align}
where, for the singular value decomposition $U\tr \hU = W_1 S W_2\tr$, the solution is given by $Q_* = W_1W_2\tr$. Then, the bound in \cref{thm:X-opnorm} holds for the rigid transformation given by
\begin{align}
    g_*(v) = Q_*\tr Q(v - \bx),\label{eq:g-star}
\end{align}
where $Q \in \orth{p}$ is the unidentifiable orthogonal transformation in \cref{eq:cs-exact-recovery}. In other words, $Q_*$ is the optimal rotation matrix which aligns the subspaces spanned by $\hX = \cs{\D, p}$ and $H\X Q\tr = \cs{\Del, p}$. Since $\hX$, and therefore $\hU$, are sample estimators based on the randomness underlying $D$, the optimal $Q_* \in \orth{p}$ and $g_* \in \euc{p}$ are also random and unidentifiable without prior information on $\X$.


To better understand how classical scaling performs on a more granular level, it is insightful to examine specific forms of reconstruction errors that are directly relevant in practical applications. A straightforward application of \cref{thm:X-opnorm} yields the following corollary, which provides a bound on the root-mean-squared reconstruction error, $\loss\rmse$, described in \cref{eq:losses}.

\begin{corollary}\label{cor:frobenius-norm}
    Under the same assumptions as in \cref{thm:X-opnorm}, for all $0 < r < (q-4)/2$ and $n > \enn_2(r, q, \sigma, \kappa)$,
    \begin{align}
    \loss\rmse(\hX, \X) \lesssim \frac{\sigma\kappa}{\sqrt{n}}.\label{eq:frobenius-norm-convergence}
    \end{align}
    with probability greater than $1 - O(n^{-2} + n^{-r})$,
\end{corollary}

The result in \cref{cor:frobenius-norm} shows that, up to a rigid transformation, the embedded configuration $\hX$ obtained using classical scaling essentially guarantees parametric rates of convergence. It is instructive to compare the result in \cref{cor:frobenius-norm} with existing results. \citet[Corollary~1]{zhang2016distance} recovers the same result as in \cref{eq:frobenius-norm-convergence} under the assumption that the $\eps_{ij}$ are \iid{} and that the noise is additive. In other words, classical scaling treats all noise models {falling within the confines of Section~\ref{sec:background}} as if they were simple additive \iid{} errors.  Similarly, for the average reconstruction error, 
$$
\loss\avg := \min_{g \in \euc{p}} \frac{1}{n} \sum_{i \in [n]} \smallnorm{\hx_i - g(X_i)},
$$ 
the bound $\loss\avg \le \loss\rmse$ follows directly from the Cauchy-Schwarz inequality; therefore $\loss\avg$ also satisfies the same parametric convergence rate as in \cref{eq:frobenius-norm-convergence}.

In many scenarios, it is crucial to have a uniform bound that controls the worst-case deviation of the estimated configuration from the true configuration. To this end, we consider the $\ell_{2\to\infty}$-operator norm, which provides a finer control on the reconstruction error of each row of $\hX$. Specifically, the $\loss\ttinft$ given by
\begin{align}
    \loss\ttinft(\hX, \X) = \min_{g \in \euc{p}}\ttinf\big{\hX - g(\X)} = \min_{g \in \euc{p}} \max_{i \in [n]}\norm\big{ \hx_{i} - g(x_{i}) }
\end{align}
provides a uniform bound on how well each point $\hx_i$ in the estimated configuration approximates $X_i$ in the latent configuration after a rigid transformation $g \in \euc{p}$. We now state the main result which  establishes the uniform convergence rate of reconstruction error of $\hX$ in the $\loss\ttinft$ metric.

\begin{theorem}\label{thm:X-ttinf}
    Suppose $\X \in \bbX(\xpar)$ and $D = \Del + \Eps$ in the noisy realizable setting of \Cref{noisy-setting} satisfying assumptions~\ref{assumption:compact}--\ref{assumption:expectation}, and let $\hX = \cs{D, p}$. Then, for all $0 < r < (q-4)/2$ there exists $\enn_3(r, q, \sigma, \kappa, \Rx)$ such that for all $n > \enn_3(r, q, \sigma,  \kappa, \Rx)$,
    \begin{align}
    \loss\ttinft(\hX, \X) \lesssim \mathbf{\overline{c}}(\xpar) \sigma\sqrt{\frac{ \log{n}}{n}},\label{eq:ttinf-concentration-bound}
    \end{align}
    with probability greater than $1 - O(n^{-2} + n^{-r})$, where $\mathbf{\overline{c}}(\xpar) = \kappa^2( \kappa + \Rx)$.
\end{theorem}

In other words, \cref{thm:X-ttinf} guarantees that, with high probability and modulo rigid transformations, every embedded point $\hx_i$ is within a $\sqrt{\log{n}/n}$ neighborhood of the a point $X_i$ in the latent configuration, and differs from \cref{cor:frobenius-norm} by a $\sqrt{\log\,{n}}$ factor. The proof of \cref{thm:X-ttinf} is provided in \cref{proof:thm:X-ttinf}. The key ingredients in the proof are a decomposition of the Frobenius-optimal Procrustes alignment (\cref{lem:decomposition}), followed by a matrix concentration inequality for controlling the $\ell_{2\to\infty}$ norm of a multiplier  process (\cref{prop:ttinf-concentration}). We also note that a direct application of the subordinate properties of the $\ell_{2\to\infty}$ together with \cref{thm:X-opnorm} don't provide the required control on the reconstruction error of $\hX$ as in \cref{thm:X-ttinf}. Specifically, a direct application of \citet[Proposition~6.3]{cape2019two} to the decomposition from \cref{lem:decomposition} or to the result of \cref{thm:X-opnorm} doesn't converge to zero at the rate in \cref{eq:ttinf-concentration-bound}. In essence, the key to obtaining the near parametric rate is to apply matrix concentration via \cref{prop:ttinf-concentration} soon after accounting for  the smaller residual terms arising from the Procrustes decomposition.

\begin{remark}
    We refer the reader to the proof of \cref{thm:minimax-frobenius} and \cref{thm:minimax-ttinf} in \cref{sec:proofs} where the precise constants in the tail probability bound, and the minimum sample size $\enn(r, q, \sigma, \kappa, \Rx)$ leading to the convergence rate are provided. We also note that, as such, the convergence rates in \cref{eq:frobenius-norm-convergence} and \cref{eq:ttinf-concentration-bound} hide the dependence on the dimension $p$, which is assumed to be fixed. A more refined bound, albeit in a more cumbersome form, is obtained in \cref{eq:precise-frobenius-norm} and \cref{eq:precise-ttinf-bound}, respectively.
\end{remark}
    
    In light of \cref{lemma:iid} and \cref{lemma:noise-models}, we have the following corollary which establishes the consistency of classical scaling when the rows of the latent configuration are observed \iid{} at random.

\begin{corollary}\label{cor:iid}
    Suppose the rows $X_1, X_2, \dots, X_n \in \R^p$ of $\X \in \Rnp$ are observed \iid{} from a distribution $F \in \F(\xpar)$ given in \cref{eq:F} and $D = \Del + \Eps$ in the noisy realizable setting of \cref{noisy-setting} satisfying assumptions~\ref{assumption:independence}~\&~\ref{assumption:moments}. Let $\hX = \cs{\D, p}$. Then, for all $0 < r < (q-4)/2$ there exists $\enn(r, q, \sigma, \kappa, \Rx)$ such that for all $n > \enn(r, q, \sigma, \kappa, \Rx)$,
    \begin{align}
        \loss\rmse(\hX, \X) &\lesssim \frac{\sigma\kappa}{\sqrt{n}} \qty\Big(1 + o(1)) \;\qq{and}\; \loss\ttinft(\hX, \X) \lesssim {\mathbf{\overline{c}}(\xpar)} \sigma\sqrt{\frac{\log{n}}{n}} \qty\Big( 1 + o(1) )\label{eq:iid-convergence}
    \end{align}
    with probability greater than $1 - O(n^{-2} + n^{-r})$.
\end{corollary}

Collectively, the results in \cref{thm:X-opnorm} and \cref{thm:X-ttinf} show that, given a noisy dissimilarity matrix, classical scaling is statistically consistent, with essentially parametric rates of convergence even under minimal assumptions on the noise. This might seem surprising at first, especially considering that as the sample size $n$ increases, the number of unknown parameters also increases. One possible way to interpret this is to view the classical scaling problem a as sort of ``full-rank regression'' problem: $d_{ij} = \del_{ij}(\X) + \eps_{ij}$ where we have ${n \choose 2}$ observations of the noisy dissimilarities (as responses), and the goal is to recover the true configuration $\X$ (the regression coefficients) which consists of $n$ unknown points. The fixed design for this ``regression setup''---so to speak---is the distance transformation operator in \cref{eq:Del-matrix-form}. With this analogy in mind, one would expect the root-mean-squared error risk for ordinary least squares in this setup to be of order 
$$
O\qty(\sigma \sqrt{\frac{n}{{n \choose 2}}}) = O\qty( {\frac{\sigma}{\sqrt{n}}}),
$$ 
which coincides with the rate in \cref{cor:frobenius-norm}. Similarly, the $\ell_\infty$ risk should be of order 
$$
O\qty(\sigma \sqrt{ \frac{{n}\log{n}}{{n \choose 2}} }) = O\qty(\sigma  \sqrt{\frac{\log{n}}{n}})
$$ 
which is consistent with the rate in \cref{eq:ttinf-concentration-bound}, modulo some constant factors depending on $\xpar$.



\section{Lower Bounds on the Reconstruction Error}\label{sec:lower-bound}

The results in \cref{sec:consistency} demonstrate that classical scaling applied to a noisy dissimilarity matrix consistently recovers the information underlying the configuration $\X$ under minimal assumptions on the noise model. To understand the fundamental limits of classical scaling and assess whether the derived rates of convergence are optimal, we adopt a minimax framework to establish lower bounds on the reconstruction error for latent position estimation in the presence of noise.

To formalize our analysis, consider the set of configurations $\bbX = \bbX(\xpar)$ as defined in \cref{eq:bbX}. We define $\Theta(\xpar, \sigma)$ to encompass all possible configurations and noise models under consideration:
\begin{align}
    \Theta(\xpar, \sigma) :=
    \left\{
    (\X, \Psi, \pi) 
    \quad\middle\lvert\:
    \parbox{0.55\textwidth}{\centering
    $\X \in \bbX(\xpar)$ and $\Eps = (\eps_{ij})$ satisfies \ref{assumption:independence}--\ref{assumption:expectation}\\[3pt] for $\eps_{ij} = \Psi(\del_{ij}(\X), \xi_{ij})$ and $\xi_{ij} \sim \pi$ for all $i < j$
    }
    \right\}.\label{eq:Theta}
\end{align}
In other words, each $\theta = (\X, \Psi, \pi) \in \Theta(\xpar,\sigma)$ defines a joint distribution $\pr_{(\X, \Psi, \pi)}$ of the observed $n \times n$ dissimilarity matrix
\begin{align}
    D = \Del(\X) + \Eps \qq{with} \eps_{ij} = \Psi(\del_{ij}(\X), \xi_{ij})\;\text{ and }\; \xi_{ij} \sim \pi,
\end{align}
in the noisy realizable setting of \cref{noisy-setting} satisfying assumptions~\ref{assumption:compact}-\ref{assumption:expectation}.

Given a reconstruction loss $\loss(\Y, \X)$ and an estimator $\hX: \Rnn \to \Rnp$, which takes an input dissimilarity matrix $D \sim \pr_{(\X, \Psi, \pi)}$ and outputs an $n\times p$ configuration, the sequence $r_n$ is a minimax lower bound for the reconstruction loss if
\begin{align}
    \inf_{\hX}\sup_{(\X, \Psi, \pi) \in \Theta(\xpar, \sigma)} \pr_{(\X, \Psi, \pi)}\qty{ \loss\qty\big(\hX(D), \X) \gtrsim r_n } \ge c,\label{eq:minimax-loss}
\end{align}
where $0 < c < 1$ is a universal constant. The minimax lower bound $r_n$ captures the best possible performance that any estimator can achieve uniformly over the parameter space $\Theta(\xpar, \sigma)$. The following result provides a minimax lower bound of the reconstruction error under the $\loss\rmse$ loss.

\begin{theorem}\label{thm:minimax-frobenius}
    Let $\Theta(\xpar, \sigma)$ be as defined in \cref{eq:Theta}. Then, there exists $\enn_4(\sigma, \kappa)$ such that for all $n > \enn_4(\sigma, \kappa)$,
    \begin{align}
        \inf_{\hX}\sup_{(\X, \Psi, \pi) \in \Theta(\xpar, \sigma)} \pr_{(\X, \Psi, \pi)}\qty{ \loss\rmse\qty\big(\hX(D), \X) \gtrsim \frac{\sigma\kappa}{\sqrt{n}} } \ge \half \qty\big( 1 - o(1) ).
    \end{align}
\end{theorem}

In view of \cref{cor:frobenius-norm}, the lower bound in \cref{thm:minimax-frobenius} matches the upper bound up to constant factors. Importantly, the result shows that the rate of convergence for classical scaling in \cref{cor:frobenius-norm} is minimax optimal. Moreover, the minimum sample sizes $\enn_2(\sigma, \kappa) \asymp \enn_4(\sigma, \kappa) \asymp \sigma^2\kappa^4$ also match up to absolute constants. The proof of \cref{thm:minimax-frobenius} is based on an application of Fano's method and is given in \cref{proof:thm:minimax-frobenius}.

We now turn our attention to the minimax risk under the $\loss\ttinft$ loss. As a preliminary attempt, note that using the identity $\ttinf{A} \ge \norm{A}_F/\sqrt{n} $ \citep[Proposition~6.3]{cape2019two} along with \cref{thm:minimax-frobenius} would suggest that $\loss\ttinft(\hX(D), \X) \gtrsim 1/n$; however, this is suboptimal compared to the rates derived in \cref{thm:X-ttinf}. The next result provides a tight minimax lower bound of the reconstruction error under the $\loss\ttinft$ loss.

\begin{theorem}\label{thm:minimax-ttinf}
    Let $\Theta(\xpar, \sigma)$ be as defined in \cref{eq:Theta}. Then, there exists $\enn_5(\sigma, \kappa, \Rx)$ such that
    \begin{align}
        \inf_{\hX}\sup_{(\X, \Psi, \pi) \in \Theta(\xpar, \sigma)} \pr_{(\X, \Psi, \pi)}\qty{ \loss\ttinft\qty\Big(\hX(D), \X) \gtrsim \textbf{\underline{c}}(\xpar) \sigma\sqrt{\frac{\log{n}}{n}} } \ge \half \qty\big(1 - o(1)),\label{eq:ttinf-minimax}
    \end{align}
    for all ${n > \enn_5(\sigma, \kappa, \Rx)}$ and $\textbf{\underline{c}}(\xpar) = {\kappa}/{1 + \kappa\Rx}$.
\end{theorem}

The proof of \cref{thm:minimax-ttinf} is given in \cref{proof:thm:minimax-ttinf}. Unlike the proof of \cref{thm:minimax-frobenius}, the proof here uses Le Cam's convex hull method \citep{yu1997assouad}. The main challenge lies in constructing a set of configurations that are well-separated in the $\loss\ttinft$ pseudometric, while ensuring their associated joint distributions over the observed dissimilarity matrix $D$ are close in the total variation metric. Similar to the  $\loss\rmse$ loss, \cref{thm:minimax-ttinf} confirms that the upper bound for classical scaling under the $\loss\ttinft$ loss is minimax optimal. Specifically, the dependence of the lower bound in \cref{eq:ttinf-minimax} on the noise parameter $\sigma$ and the sample size $n$ exactly matches the upper bound in \cref{eq:ttinf-concentration-bound}; notably, the additional $\sqrt{\log{n}}$ factor is essential and cannot be improved. The dependence on the parameters $\kappa$ and $\Rx$, however, exhibits a gap of constant order since $\textbf{\underline{c}}(\xpar) < \mathbf{\overline{c}}(\xpar)$. This gap is likely an artifact of the proof technique---owing to the challenging  geometry of $\Rnp$ under the $\ell\ttinft$-norm, and may not be inherent to classical scaling. Specifically, the upper bound in \cref{thm:X-ttinf} is derived using the Frobenius-optimal Procrustes alignment in \cref{eq:procrustes}. It is conceivable that an $\ell\ttinft$-optimal alignment, i.e., $Q' = \argmin_{Q \in \orth{p}}\smallnorm{\hU - UQ}\ttinft$, instead could close this gap and yield sharp bounds. Similarly, for the lower bound in \cref{thm:minimax-ttinf}, the construction of the local packing set in \cref{eq:bbX-1} which are maximally separated $\loss\ttinft$ pseudometric could yield tighter constants in \cref{lemma:X1-properties} if the perturbation $v(k)$ in \cref{eq:bbX-1} could be chosen to be in the orthogonal subspace to each $X_k$. We leave these refinements for future work.

\section{Numerical Experiments}
\label{sec:experiments}

We conduct numerical experiments to empirically validate the theoretical results in \cref{sec:consistency} and examine the effects of noise and the configuration on reconstruction errors. In all cases, points $\X \in \Rnp$ are sampled from a uniform distribution on the unit ball $B_p(\zerov, 1) \subset \Rp$ for $p=3$.

\begin{enumerate}[label=\textbf{Experiment \arabic*.}, ref=Experiment \arabic*, leftmargin=*, itemindent=8em]
    \item\label{exp:1} 
    To examine the influence of the configuration matrix, we transform the points $\X$ using a diagonal matrix $S$ with entries ranging from $1/\kappa$ to $\kappa$. The resulting configuration $\X_\kappa = \X S$ has its points uniformly distributed over an ellipsoid and the singular values of $\X_\kappa/\sqrt{n}$ are bounded in the interval $[1/\kappa, \kappa]$. The noisy dissimilarity matrix is then generated under the additive noise model:
    $$
    D = \Del(\X_\kappa) + \Xi, \quad \text{where} \quad \xi_{ij} \sim t_q, \; q \in \{3, 5, 7\}.
    $$
    As discussed in \cref{sec:background}, the noise satisfies assumptions \ref{assumption:compact}--\ref{assumption:expectation} only for $q=5, 7$, while for $q=3$, the fourth moment does not exist.

    \qquad For the embedding $\hX = \cs{D, p}$ obtained using classical scaling, we compute the reconstruction errors $\loss\rmse(\hX, \X)$ and $\loss\ttinft(\hX, \X)$ after a Procrustes alignment. Figures~\ref{fig:p11}~\textit{\&}~\ref{fig:p12} display the results on a $\log$-$\log$ scale over $10$ trials, with shaded regions indicating empirical 80\% intervals. The dashed lines represent the theoretical rates. The results illustrate the effect of both the singular value bound $\kappa \in \{1, 1.25, \dots, 2\}$  and the noise level $\sigma \in \{0.1, 0.25, 0.5\}$ on reconstruction performance.

    \qquad In \cref{fig:p11}, for $q > 4$, the reconstruction error $\loss\rmse$ follows the expected bound $O(\sigma\kappa\sqrt{1/n})$ from \cref{cor:frobenius-norm}. However, for $q = 3$ (first column), the error does not exhibit the same behavior. Similarly, in \cref{fig:p12}, the error $\loss\ttinft$ scales correctly with $\sigma$ and exhibits the expected $O(\sqrt{\log{n}/n})$ dependence on sample size $n$. Empirically, the dependence on $\kappa$ appears to scale as $O(\kappa^2)$.

    \item\label{exp:2} 
    To examine the sensitivity of classical scaling to choice of the noise model, we generate the dissimilarity matrix
    \[
    D = \Del(\X) + \Eps, \quad \text{where} \quad \Eps = \Psi(\Del, \Xi),
    \]
    for the noise models in \cref{tab:noise-models}. The random component $\Xi = (\xi_{ij})$ is sampled from a Student's~$t$-distribution, i.e., $\xi_{ij} \sim t_q$ for $q \in \{3, 5, 7\}$. 

    \qquad For $\hX = \cs{D, p}$, Figures~\ref{fig:p21}~\textit{\&}~\ref{fig:p22} show the reconstruction errors $\loss\rmse(\hX, \X)$ and $\loss\ttinft(\hX, \X)$, respectively, after a Procrustes alignment averaged across $10$ trials. The shaded regions indicate the $80\%$ empirical intervals and the black dashed lines represent the theoretical rates. The results confirm that classical scaling provides consistent recovery across all noise models. The reconstruction error follows the predicted rates—$O(\sigma\sqrt{1/n})$ in \cref{fig:p21} and $O(\sigma\sqrt{\log{n}/n})$ in \cref{fig:p22}. Notably, for $q=3$, where the noise lacks a finite fourth moment, the reconstruction error exhibits uncontrolled fluctuations, confirming the necessity of the moment condition.
\end{enumerate}

\begin{figure}
    \centering
    \includegraphics[width=1.0\linewidth]{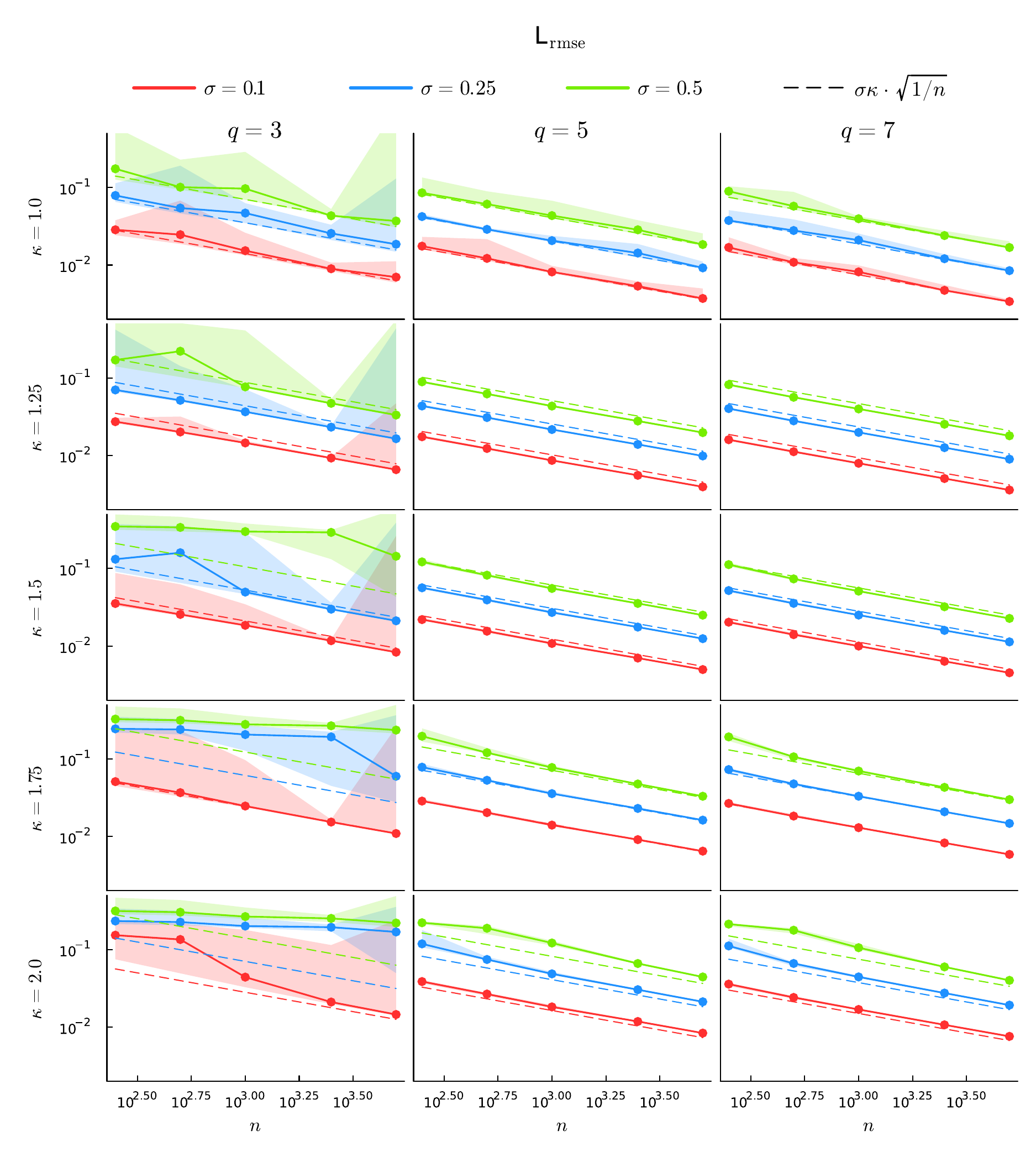}
    \caption{$\loss\rmse(\hX, \X)$ reconstruction error vs. $n$ for the setup in \ref{exp:1} on a $\log$-$\log$ scale.}
    \label{fig:p11}
\end{figure}

\begin{figure}
    \centering
    \includegraphics[width=1.0\linewidth]{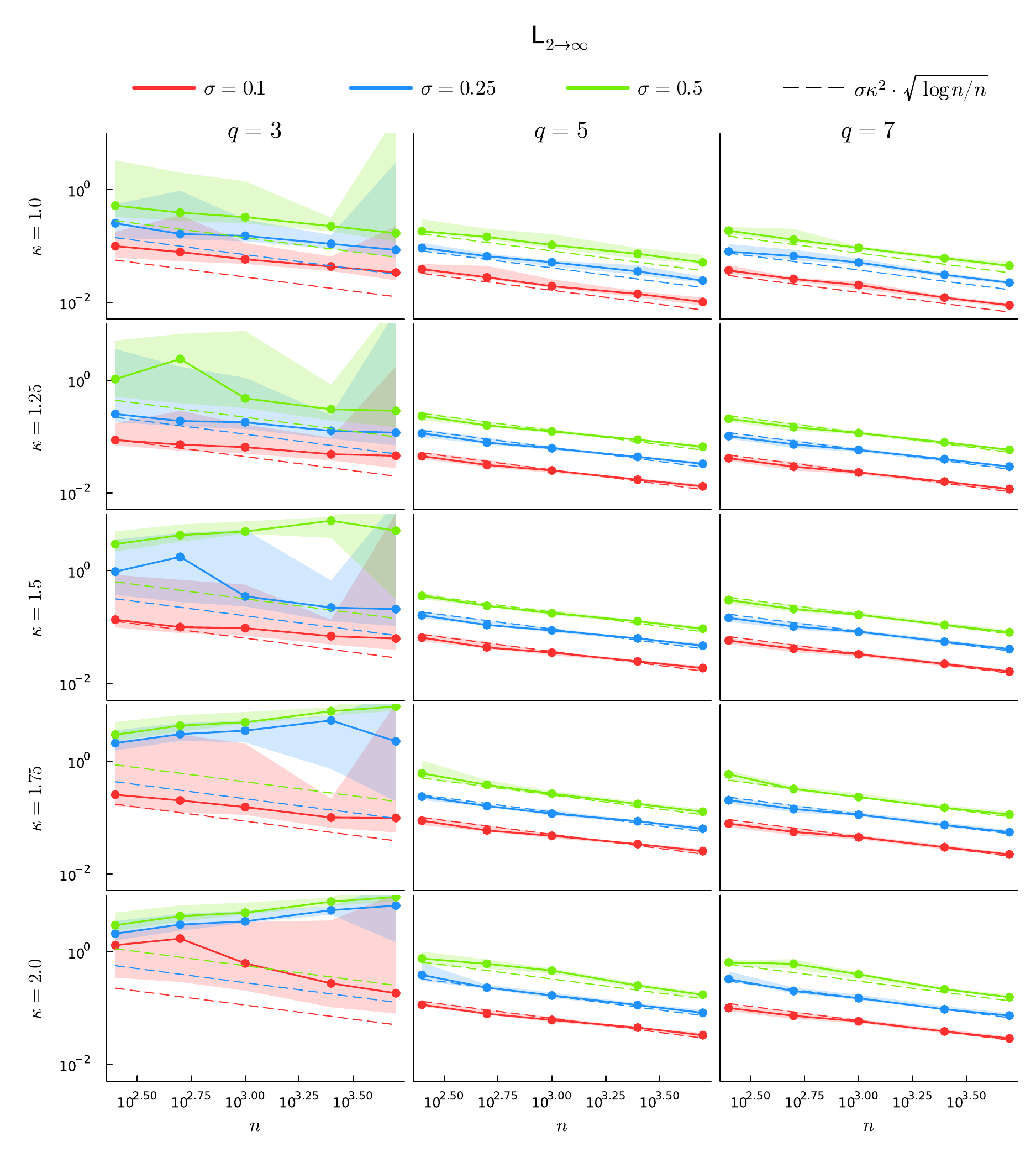}
    \caption{$\loss\ttinft(\hX, \X)$ reconstruction error vs. $n$ for the setup in \ref{exp:1} on a $\log$-$\log$ scale.}
    \label{fig:p12}
\end{figure}

\begin{figure}
    \centering
    \includegraphics[width=1.0\linewidth]{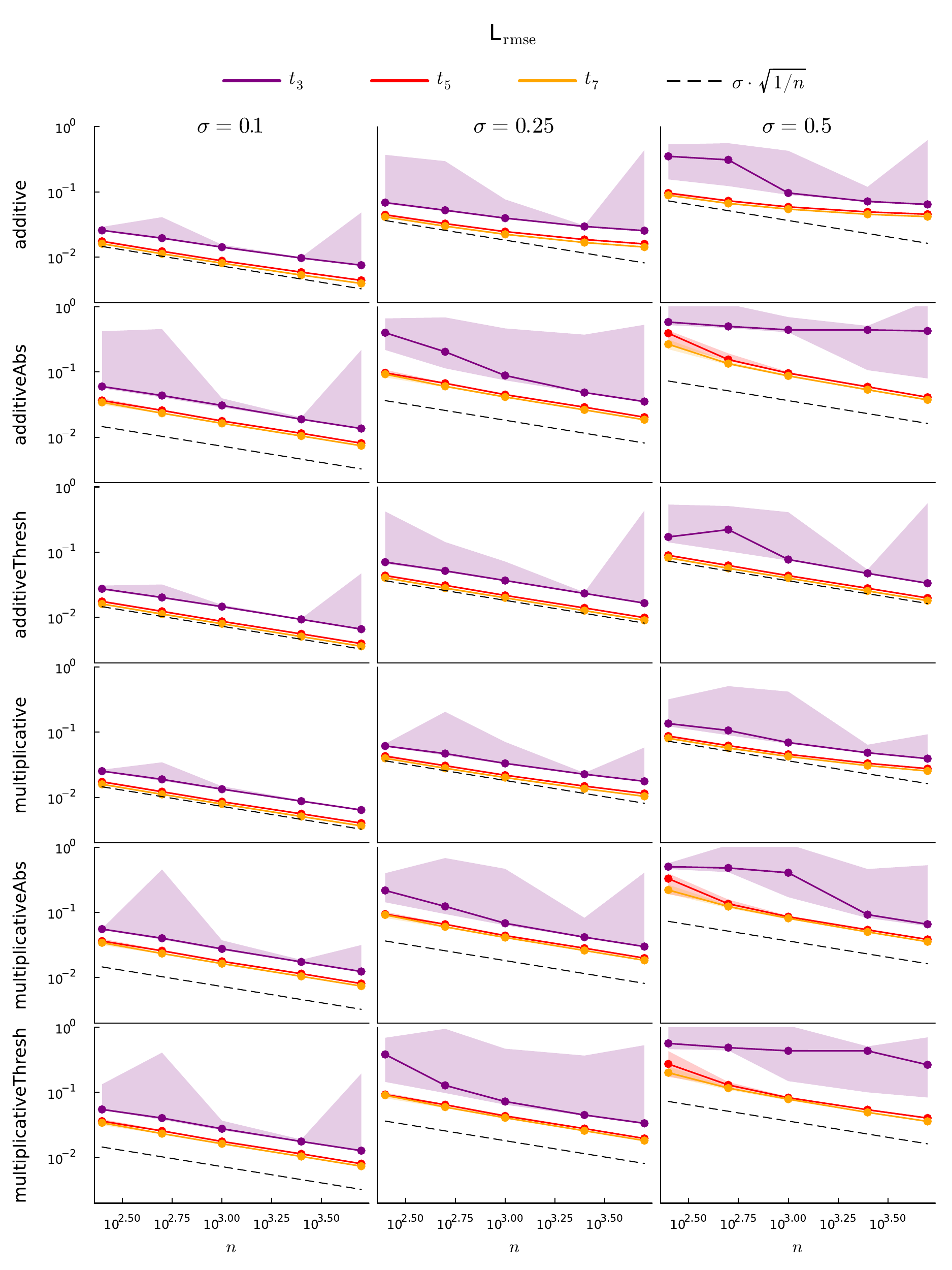}
    \caption{$\loss\rmse(\hX, \X)$ reconstruction error vs. $n$ for the setup in \ref{exp:2} on a $\log$-$\log$ scale.}
    \label{fig:p21}
\end{figure}

\begin{figure}
    \centering
    \includegraphics[width=1.0\linewidth]{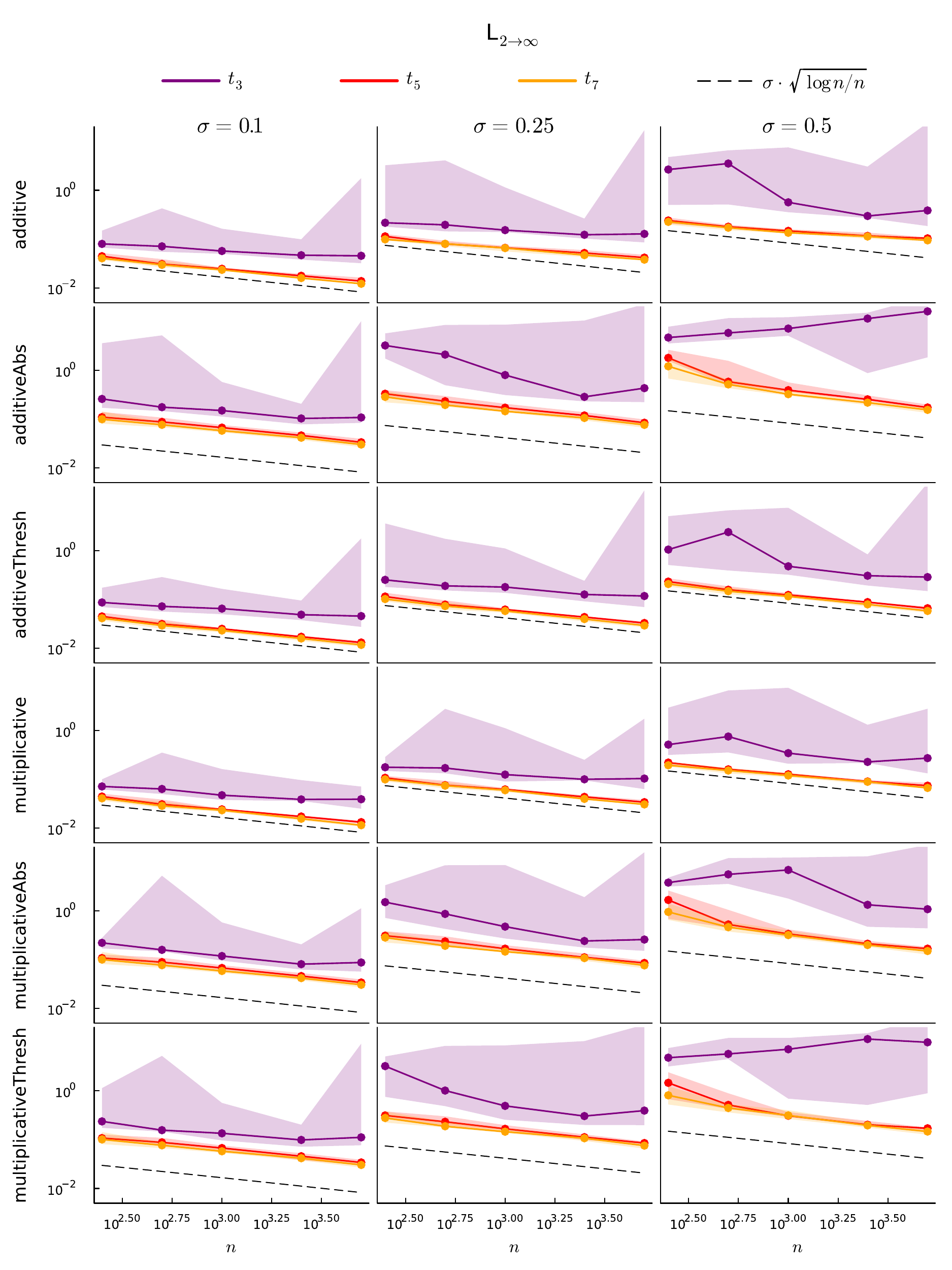}
    \caption{$\loss\ttinft(\hX, \X)$ reconstruction error vs. $n$ for the setup in \ref{exp:2} on a $\log$-$\log$ scale.}
    \label{fig:p22}
\end{figure}



\section{Discussion}
\label{sec:discussion}

Classical scaling is a simple yet powerful algorithm that has been widely used in a variety of applications across numerous disciplines. Given a noisy dissimilarity matrix, the results in this work establish that classical scaling guarantees consistent and minimax-optimal recovery of the latent configuration under a broad class of noise models and under minimal assumptions. This confirms that classical scaling is not only a practical choice but also a theoretically sound method that achieves the fundamental limits of recovery in noisy settings.

The primary limitations of classical scaling in practice stem not from the algorithm itself, but rather from the assumptions on the underlying data. The theoretical guarantees rely on the assumption that the true dissimilarities are realizable in Euclidean space. If the original dissimilarities---even before the addition of noise---are non-Euclidean, then the same guarantees no longer hold. Several alternative methods, e.g., ISOMAP \citep{tenenbaum2000global} and Maximum Variance Unfolding (MVU) \citep{weinberger2006graph}, have been developed to handle non-Euclidean data and are known to be consistent under various assumptions. However, their performance in the presence of noise is less well understood.

Another key limitation concerns the assumptions on the noise distribution. As demonstrated in \cref{sec:experiments}, the performance of classical scaling deteriorates significantly when the noise lacks a finite fourth moment. While some prior work has touched on robustness in this context \citep[e.g.,][]{cayton2006robust,mandanas2016robust}, formal statistical analyses remain relatively under-explored, and is an important direction for future research. Additionally, the dependence on the parameters $\kappa$ and $\Rx$ in \cref{thm:X-ttinf} and \cref{thm:minimax-ttinf} leaves a gap of constant order: $\mathbf{\underline{c}}(\xpar) < \mathbf{\overline{c}}(\xpar)$. Closing this gap will likely require a more refined analysis of the geometry of $\Rnp$ under the $\ell\ttinft$-norm, particularly in relation to rigid transformation alignment and packing arguments.

Beyond the finite-dimensional setting, recent work has explored generalizations of classical scaling to metric measure spaces \citep{adams2020multidimensional,kroshnin2022infinite,lim2024classical}, and extends the framework to an infinite-dimensional analogue distinct from the high-dimensional regime considered in \cite{little2023analysis,peterfreund2021multidimensional}. Developing statistical performance guarantees and minimax bounds for these settings, particularly in the Gromov-Wasserstein metric, is an interesting avenue for future work.



\section{Proofs}
\label{sec:proofs}

Throughout the proofs we use {$a_n \lesssim b_n$ to denote that there exists a constant $C > 0$ such that $a \le Cb$}, which may depend on the fixed, deterministic parameters of the model, and may change from line to line, but \textit{does not} depend on the sample size $n$.


\subsection{Proof of \cref{lemma:iid}}
\label{proof:lemma:iid}

Let $X_1, X_2, \dots, X_n \simiid{} F \in \F(\xpar)$ be the rows of $\X \in \Rnp$, and for $\bx = X\tr\onev$ let $\tx_i = X_i - \bx$ denote the rows of $H\X$. Since ${\diam(\supp(F)) \le \Rx}$, it follows that $\smallnorm{X_i - \bx} \le \Rx$ a.s. for all $i \in [n]$, and, therefore, 
\begin{align}
    \ttinf{H\X} = \max_{i \in [n]}\norm{X_i - \bx} \le \Rx \qq{a.s.}\label{eq:row-bound}
\end{align}
Moreover, since each $\smallnorm{\tx_i} \le \Rx$, from \citet[Example~2.5.8]{vershynin2018high}, $\tx_i$ is sub-Gaussian with parameter $\Rx$ and for any $A \in \R^{p \times p}$ the Orlicz norm $\smallnorm{A\tx_i}_{\psi_2} \le \opnorm{A}\Rx$. Let $\Sigma = \cov(F)$ and $\hat\Sigma = \frac{1}{n}(H\X)\tr(H\X)$ denote the sample covariance matrix of $\X$, and note that $\opnorm\big{\hat\Sigma} = s_1(H\X/\sqrt{n})^2$ and $\opnorm\big{\hat\Sigma\inv}\inv = s_p(H\X/\sqrt{n})^2$. 

From \citet[Theorem~4.7.1 and Example~4.7.3]{vershynin2018high}, for all $n > 1$ with probability greater than $1 - 2n^{-2}$,
\begin{align}
    \opnorm\big{\hat\Sigma - \Sigma} \lesssim \opnorm{\Sigma}\Rx^2 \qty( \sqrt{\frac{p + \log{n}}{n}} ). \label{eq:covariance}
\end{align}
For the inverse covariance matrix, let $\enn_0 \equiv \enn_0(\xpar)$ be such that $n \gtrsim (p+ \log{n}) (\kappa\Rx)^4$ for all $n > \enn_0$. Then, from \citet[Corollary~11]{kereta2021estimating} it follows that for $n > \enn_0$ and with probability greater than $1 - 2n^{-2}$,
\begin{align}
    \opnorm\big{\hat\Sigma\inv - \Sigma\inv} \lesssim \opnorm{\Sigma\inv}^2\Rx^2 \qty( \sqrt{\frac{p + \log{n}}{n}} ).\label{eq:precision}
\end{align}
Let $\alpha_n = \kappa^2\Rx^2\sqrt{p+\log{n}/ n}$. On the event that \cref{eq:covariance} holds, using the triangle inequality and by noting that $\kappa > 1$ from assumption~\ref{assumption:compact},
\begin{align}
    \opnorm\big{\hat\Sigma} \le \opnorm{\Sigma} + \opnorm\big{\hat\Sigma - \Sigma} \lesssim \opnorm{\Sigma} (1 + \alpha_n).\label{eq:opnorm}
\end{align}
Using the fact that $\opnorm{\Sigma} \le \kappa^2$ and $\sqrt{1+z}\le1+z$ for all $z\ge 0$, we have that
\begin{align}
    s_1\qty(\frac{H\X}{\sqrt{n}}) = \opnorm\big{\hat\Sigma}^{1/2} \lesssim \qty\Big(\opnorm{\Sigma} (1+\alpha_n))^{1/2} \le \kappa \qty(1 + \alpha_n)\label{eq:singular-1}
\end{align}
with probability greater than $1 - 2n^{-2}$. Similarly, on the event that \cref{eq:precision} holds, since ${\smallnorm{\Sigma\inv}_2 \le \kappa^2}$ we have
\begin{align}
    \opnorm\big{\hat\Sigma\inv} \le \opnorm\big{\Sigma\inv} + \opnorm\big{  \hat\Sigma\inv - \Sigma\inv} \lesssim \opnorm{\Sigma\inv} (1 + \alpha_n), \label{eq:opnorm-inv}
\end{align}
and it follows that for all $n > \enn_0$,
\begin{align}
    s_p\qty(\frac{H\X}{\sqrt{n}}) = \opnorm\big{\hat\Sigma\inv}^{-\half} \gtrsim \qty( \frac{\opnorm{\Sigma\inv}\inv}{1+\alpha_n} )^{\half} \gtrsim \frac{1}{\kappa\sqrt{1+\alpha_n}} \ge \frac{1}{\kappa(1+\alpha_n)}.\label{eq:singular-2}
\end{align}
with probability greater than $1 - 2n^{-2}$. Combining \cref{eq:singular-1,eq:singular-2}, it follows that\newline ${\X \in \bbX\qty\big(\kappa(1+\alpha_n), \Rx)}$ with probability greater than $1 - 4n^{-2}$. \QED


\subsection{Proof of \cref{lemma:noise-models}}
\label{proof:lemma:noise-models}

\textit{Claim~\ref{lem:noise-1}.} Let $W = \E(\Eps) = \E_{\Del, \Xi}( \Psi(\Del, \Xi) )$. By independence of $\Del$ and $\Xi$, and since the rows $\X$ are \iid{} from $F$ and the entries of $\Xi$ are \iid{} with $\E(\xi_{ij}) = \gamma$, it follows that there exists some fixed $w \in \R$ such that
\begin{align}\label{eq:Expectation-Eps}
    W_{ij} = \E\qty(\Psi(\del_{ij}, \xi_{ij})) = \begin{cases}
        w, & i \neq j\\
        0, & i = j
    \end{cases}
    \qq{}\qq{i.e.,} W = w(J-I).
\end{align}
Since $H = I - J/n$ is a projection matrix, $H(J-I) = -H$ and $H(J-I)H = -H^2 = -H$, and it follows that
\begin{align}
    \norm{H \,\E(\Eps) \, H}_\infty = \norm{-w H}_\infty = \max_{i \in [n]}\sum_{j \in [n]} \abs{w H_{ij}} = \abs{w} \qty(1 + \sum_{j \neq i}\abs{1 - \frac{1}{n}}) \le 2\abs{w} \lesssim \sqrt{n}.\FINEQ
\end{align}

\noindent\textit{Claim~\ref{lem:noise-2}.} When $\X$ is fixed, let $W' = \E(\Eps) = \E_{\Xi}(\Psi(\Del, \Xi))$. If $\xi_{ij}$ are \iid{} symmetric, zero mean random variables, then $\E(\xi^{2k+1}) = 0$ for all $k$, and
\begin{align}
    \E(\Psi(\del_{ij}, \xi_{ij})) = \sum_{k=1}^K g_k(\del_{ij})\E(\xi_{ij}^{2k+1}) + \sum_{\ell=1}^{L} \E(\xi_{ij}^{\ell}) = \sum_{\ell=1}^{L} \E(\xi_{ij}^{\ell}) =: w'.
\end{align}
Then, $W' = \E(\Eps) = w'(J-I)$ and, similar to claim~\ref{lem:noise-1}, it follows that $\norm{H \,\E(\Eps) \, H}_\infty \le 2\abs{w'} \lesssim \sqrt{n}$. Moreover, if $\xi_{ij}$ admits finite moments up to order $q' = q \cdot \max\qty{L, 2K+1}$, then $\eps_{ij} = \Psi(\del_{ij}, \xi_{ij})$ admits finite moments up to order $q$. Lastly, for the additive noise model in \cref{tab:noise-models}, since $\Psi(\del_{ij}, \xi_{ij}) = \xi_{ij}$, the zero mean assumption on $\xi_{ij}$ can be relaxed to $\E(\xi_{ij}) = \gamma$. In this case, $\E(\Eps) = \gamma(J-I)$, and the result follows from an identical argument. \qed

\subsection{Proof of \cref{lem:D-opnorm}}
\label{proof:lem:D-opnorm}

From \cref{eq:Dc}, we have $\Dc - \Delc = -\half H\Eps H$, and from the triangle inequality,
\begin{align}
    \opnorm{\Dc - \Delc} 
    = \half\opnorm{H \Eps H} &\le \half \qty( \opnorm\big{H \,\E(\Eps) \, H} + \opnorm\big{H\: \qty\big(\Eps - \E(\Eps)) \: H})\\
    &\lesssim \underbrace{\opnorm\big{H \,\E(\Eps) \, H}}_{=:\Circled{1}} + \underbrace{\opnorm\big{\Eps - \E(\Eps)}}_{=:\Circled{2}}.\label{eq:delc-opnorm}
\end{align}
where the final inequality follows from the fact that $\opnorm{HAH} \le \opnorm{A}$ since $H$ is a projection matrix with $\opnorm{H} = 1$. Note that $\Circled{1}$ is deterministic and from \citet[Corollary~6.1.5]{horn2012matrix} and assumption~\ref{assumption:expectation},
\begin{align}
    \Circled{1} = \opnorm{H \,\E(\Eps) \, H} \le \norm{H \,\E(\Eps) \, H}_{\infty} \lesssim \sqrt{n}.
\end{align}
We bound $\Circled{2}$ from above in expected value and in tail probability via symmetrization \citep[Chapter~2.3~of][]{van2023weak}. To this end, let $\Eps'$ be an \iid{} copy of $\Eps$. Let $\Eps' \indep \Eps$ be an \iid{} copy of $\Eps$. Then, by an application of Jensen's inequality,
\begin{align}
    \E\opnorm{\Eps - \E(\Eps)} = \E_{\Eps}\opnorm{\Eps - \E_{\Eps'}(\Eps')} = \opnorm{\E_{\Eps'}\qty[\Eps - \Eps']} \le \E_{\Eps, \Eps'}\qty[ \opnorm{\Eps - \Eps'} ],
\end{align}
where $\E_{\Eps,\Eps'}$ is the expected value jointly over $\Eps,\Eps'$. By symmetrization, for each $i, j \in [n]$,
\begin{align}
    (\eps_{ij} - \eps_{ij}') \stackrel{d}{=} (\eps_{ij}' - \eps_{ij}) \stackrel{d}{=} g_{ij}\cdot (\eps_{ij} - \eps_{ij}'),
\end{align}
where $g_{ij} \in \qty{-1, +1}$. Taking $G := (g_{ij})$ to be the symmetric matrix with \iid{} Rademacher random variables with $G \indep \qty{\Eps, \Eps'}$, and using the triangle inequality,
\begin{align}
    \E_{\Eps,\Eps'}\qty\big[ \opnorm{\Eps - \Eps'} ] = \E_{G,\Eps,\Eps'}\qty\big[ \opnorm{G \circ (\Eps - \Eps')}] \le 2\E_{\Eps,G}\qty\big[\opnorm{G \circ \Eps}].
\end{align}

For $\Eps = \Psi(\Del, \Xi)$, using the law of total expectation, we have
\begin{align}
    \E\opnorm{\Eps - \E(\Eps)} \lesssim \E_{\Eps,G}\qty\big[\opnorm{G \circ \Eps}] = \E_{\Del, \Xi, G}\qty\big[\opnorm{G \circ \Eps}] = \E_{\Del}\qty\Big{ \E_{\Xi, G}\qty\big[ \opnorm{G \circ \Eps} \mid \Del] }.\label{eq:opnorm-expectation}
\end{align}
Conditional on $\Del$, note that $G \circ \Eps$ are independent, mean zero random variables. Using \citet[Theorem~2]{latala2005some} (see \cref{prop:latala}) for the conditional expectation in \cref{eq:opnorm-expectation},
\begin{align}
    \E_{\Xi, G}\qty\Big[\opnorm{G \circ \Eps} \mid \Del] \lesssim \max_{i}\sqrt{\sum_{j}\E[g_{ij}^2\eps_{ij}^2 \mid \del_{ij}]} + \sqrt[4]{\sum_{i, j}\E[g_{ij}^4\eps_{ij}^4 \mid \del_{ij}]}.\label{eq:opnorm-conditional-expectation}
\end{align}
By noting that $g_{ij}^2=g_{ij}^4 = 1$, and from \ref{assumption:moments}, we have 
\begin{align}
    \E[g_{ij}^2\eps_{ij}^2 \mid \del_{ij}] = \E[\eps_{ij}^2 \mid \del_{ij}] \le \sigma^2 \qq{and} \E[g_{ij}^4\eps_{ij}^4 \mid \del_{ij}] = \E[\eps_{ij}^4 \mid \del_{ij}] \le \sigma^4.\label{eq:moment-bounds}
\end{align}
Since the above bounds hold uniformly over $\Del = (\del_{ij})$, there exists an absolute constant $C_1 > 0$ such that
\begin{align}
    \E\opnorm{\Eps - \E(\Eps)} \lesssim \E_{\Del}\qty\Big{\E_{\Xi, G}\qty\big[\opnorm{G \circ \Eps}\mid \Del] } \le C_1 \sigma\sqrt{n}.\label{eq:opnorm-deviation}
\end{align}
Plugging \cref{eq:opnorm-deviation} back into \cref{eq:delc-opnorm} and using the bound on $\Circled{1}$, we get
\begin{align}
    \E\opnorm{\Dc - \Delc} \lesssim \sigma\sqrt{n}.
\end{align}

For the tail bound in \cref{eq:D-opnorm-1}, using the symmetrization of tail probabilities from \cref{lem:opnorm-symmetrization} with $\Y = {\Eps - \E(\Eps)}$, $N = -\E(\Eps)$ and $t > 4\E\opnorm{\Eps - \E(\Eps)}$,
\begin{align}
    \pr_{\Eps}\qty( \opnorm{\Eps - \E(\Eps)} > t ) \le 4 \pr_{\Eps, G}\qty(\opnorm{G \circ \Eps} > \frac{t}{4}) = 4 \E_{\Del} \qty[ \pr_{\Xi, G}\qty(\opnorm{G \circ \Eps} > \frac{t}{4} \mid \Del) ],\label{eq:opnorm-tail-bound}
\end{align}
where the last equality follows from the law of total expectation. Since $G \circ \Eps$ has independent, mean zero entries conditional on $\Del$, using \cref{prop:latala-quantitative} for the conditional probability in the r.h.s. of \cref{eq:opnorm-tail-bound} it follows that, for $0 < r < (q-4)/2$, there exists an absolute constant $C_2 > 0$ such that
\begin{align}
    \pr_{\Xi, G}\qty\Bigg( \opnorm{G \circ \Eps} > C_2 \qty\Bigg(\underbrace{{\max_{i}\sqrt{\sum_{j}\E [g_{ij}^2\eps_{ij}^2 \mid \del_{ij}]}} + \max_{i,j}\qty\Big(\E{[|g_{ij}\eps_{ij}|^q  \mid \del_{ij}]})^{1/q} n^{\frac{r+2}{q}}\sqrt{\log{n}}}_{=:\Circled{3}}) \mid \Del ) \le \frac{1}{n^2} + \frac{1}{n^r}.\label{eq:opnorm-tail-bound-2}
\end{align}
Similar to \cref{eq:opnorm-conditional-expectation}, we have 
\begin{align}
    \Circled{3} \le \sigma \sqrt{n} + \sigma n^{\frac{r+2}{q}}\sqrt{\log{n}}.\label{eq:circled-3}
\end{align}
Since $(r+2)/q < \half$ for all $0 < r < (q-4)/2$, for an absolute constant $C_3 > 0$ there exists $\enn_1(r, q)$ such that for all $n > \enn_1(r, q)$,
\begin{align}
    \Circled{3} \le \sigma \sqrt{n} + \sigma n^{\frac{r+2}{q}}\sqrt{\log{n}} \le C_3\sigma\sqrt{n},
\end{align}
and,
\begin{align}
    \pr_{\Xi, G}\qty\Big( \opnorm{G \circ \Eps} > C_2C_3 \sigma\sqrt{n} ) \le \pr_{\Eps, G}\qty\Big( \opnorm{G \circ \Eps} > C_2 \cdot \Circled{3} \mid \Del ) \le \frac{1}{n^2} + \frac{1}{n^r}.\label{eq:opnorm-tail-bound-3}
\end{align}
Setting $t = C\sigma\sqrt{n}$ where $C > \max\qty{4C_1, 8C_2C_3}$, it follows that $t > 4\E\opnorm{\Eps - \E(\Eps)}$ for all $n > N_0$, and satisfies the requirement of \cref{lem:opnorm-symmetrization}. Note that the l.h.s. and r.h.s. of \cref{eq:opnorm-tail-bound-3} do not depend on $\Del$; therefore, integrating \cref{eq:opnorm-tail-bound-3} w.r.t. $\pr_{\Del}$ and plugging it back into \cref{eq:opnorm-tail-bound}, we get
\begin{align}
    \pr_{\Eps}\qty\Big( \opnorm{\Eps - \E(\Eps)} > C\sigma\sqrt{n} ) \le \frac{4}{n^2} + \frac{4}{n^r}.
\end{align}
Moreover, using the bound on $\Circled{1}$ in \cref{eq:delc-opnorm}, it follows that
\begin{align}
    \opnorm{\Dc - \Delc} \lesssim \sigma \sqrt{n}\label{eq:opnorm-tail-bound-4}
\end{align}
with probability greater than $1 - 4(n^{-2} + n^{-r})$ for all $n > \enn_1(r, q)$, which completes the proof.
\QED

\subsection{Proof of \cref{thm:X-opnorm}}
\label{proof:thm:X-opnorm}

Let $\bx = \frac{1}{n}\X\tr\onev$, and for every $Q \in \orth{p}$ let $g_Q(x_i) = Q\tr(x_i - \bx)$. Note that
$$
\norm{g_Q(x_i) - g_Q(x_j)} = \norm{Q\tr(x_i - \bx) - Q\tr(x_j - \bx)} = \norm{x_i - x_j} = \del_{ij},
$$
and $g_Q(\X) = (H\X)Q$ is a rigid transformation of the rows of $\X$. Therefore,
\begin{align}
    \min_{g \in \euc{p}}\opnorm\big{\hX - g(\X)} \le \min_{Q \in \orth{p}}\opnorm\big{\hX - g_Q(\X)} = \min_{Q \in \orth{p}}\opnorm\big{\hX - \Xc Q},\label{eq:gmin-opnorm}
\end{align}
where, for brevity, $\Xc = H\X$ is the centered configuration. Let $\Xc\pinv := (\Xc\tr\Xc)\inv\Xc\tr$ denote the Moore-Penrose pseudoinverse of $\Xc$ and $\tau^2 := \opnorm{\Dc - \Delc}$. Corollary~1 from \cite{arias2020perturbation} establishes that whenever $\smallnorm{\X\pinv}_2 \cdot \tau \le 1/\sqrt{2}$,
\begin{align}
    \min_{Q \in \orth{p}}\opnorm\big{\hX - \Xc Q} \le (1 + \sqrt{2}) \norm\big{\Xc\pinv} \tau^2.\label{eq:perturbation}
\end{align}
From assumption~\ref{assumption:compact}, we have
\begin{align}
    \opnorm\big{\Xc\pinv} 
    = \opnorm\big{ (\Xc\tr\Xc)\inv \Xc\tr } 
    = \sqrt{ \lambda_1\qty\big({(\Xc\tr\Xc)^{\inv}}) }
    = \qty(n\lambda_p\qty( \tfrac{\Xc\tr\Xc}{n} ))^{-1/2} \le \frac{\kappa}{\sqrt{n}}.
\end{align}
Moreover, from \cref{lem:D-opnorm}, there exists $\enn_1 \equiv \enn_1(r, q)$ and and absolute constant $C > 0$ such that for all $n > \enn_1$,
$$
\tau^2 \le C\sigma\sqrt{n}
$$ 
with probability greater than $1 - O(n^{-1} + n^{-r})$. For $n > 4C^2 {\sigma^2\kappa^4} =: n_2(\sigma, \kappa)$, we further have
\begin{align}
    \opnorm\big{\Xc\pinv} \cdot \tau \le \frac{\kappa\sqrt{C\sigma}}{n^{1/4}} \le \frac{1}{\sqrt{2}},
\end{align}
which satisfies the hypothesis for \citet[Corollary~1]{arias2020perturbation}. Plugging the bound from \cref{eq:perturbation} back into \cref{eq:gmin-opnorm}, it follows that for all $n > \enn_2(r, q, \sigma, \kappa) = \max\qty{\enn_1, n_2}$, 
\begin{align}
    \min_{g \in \euc{p}}\opnorm\big{\hX - g(\X)} \le (1+\sqrt{2})\opnorm\big{\Xc\pinv} \tau^2 \lesssim {\sigma\kappa}
\end{align}
with probability greater than $1 - O(n^{-2} + n^{-r})$.
\QED

\subsection{Proof of \cref{cor:frobenius-norm}}
\label{proof:cor:frobenius-norm}

For every $g \in \euc{p}$, we have $\smallnorm{\hX - g(\X)}_F \le \sqrt{p} \cdot \smallnorm{\hX - g(\X)}_2$. Therefore, from \cref{thm:X-opnorm}, it follows that
\begin{align}
    \loss\rmse(\hX, \X) 
    &= \qty(\min_{g \in \euc{p}}\frac{1}{n}\smallnorm{\hX - g(\X)}^2_F)^{1/2}\\ 
    &\le \qty(\frac{p}{n}\min_{g \in \euc{p}}\smallnorm{\hX - g(\X)}^2_2)^{1/2} \lesssim \sigma\kappa\sqrt{\frac{p}{n}},\label{eq:precise-frobenius-norm}
\end{align}
with probability greater than $1 - O(n^{-2} + n^{-r})$ for all $n > \enn_2(r, q, \sigma, \kappa)$.
\QED

\subsection{Proof of \cref{thm:X-ttinf}}
\label{proof:thm:X-ttinf}

For $\hX = \hU\hL^{1/2} = \cs{D, p}$, $H\X Q\tr = U\L^{1/2} = \cs{\Del, p}$ and for $\bx = \X\tr\onev/n$, let $g_*(v) = Q_*\tr Q(v - \bx)$ be the rigid transformation given in \cref{eq:g-star}. From Lemma~\ref{lem:decomposition}, we have the decomposition

\begin{align}
    \hX - g_*(\X) &= 
    \qty\big(\Dc - \Delc) \qty\big(\hU - UQ_*) \hL^{-1/2}         \tag{$=:R_1$}\\
    &\qq{}+U\L \qty\big(U\tr\hU - Q_*)\hL^{-1/2}                  \tag{$=:R_2$}\\
    &\qq{}+\Dc U \qty\big( Q_* \hL^{-1/2} - \L^{-1/2}Q_* )        \tag{$=:R_3$}\\
    &\qq{}+\qty\big(\Dc - \Delc) U \L^{-1/2}Q_*,                 \tag{$=:R_4$}
\end{align}
and, by the triangle inequality,
\begin{align}
    \loss\ttinft(\hX, \X) \le \ttinf\big{\hX - g_*(\X)} &\le \ttinf{R_1} + \ttinf{R_2} + \ttinf{R_3} + \ttinf{R_4}.
\end{align}
Let $A := \qty\Big{ \opnorm{\Dc - \Delc} \le C\sigma\sqrt{n} }$ be the event from \cref{lem:D-opnorm} for which $\pr(A^c) = O(n^{-2} + n^{-r})$ for all $n > \enn_1(r, q)$. The outline of the proof is as follows.
\begin{enumerate}
    \item On the event $A$, we bound the $\ell_{2\to\infty}$-operator norm of $R_1, R_2$ and $R_3$ using \cref{lem:gram} and \cref{lem:eigen} with $K_\eps = C \sigma $.
    \item We use \cref{prop:ttinf-concentration} to bound $\ttinf{R_4}$, and combine the results using a union bound.
\end{enumerate}

Throughout the proof we use the three properties of the $\ell_{2\to\infty}$-operator norm collected in \cref{prop:cape-ttinf}.

\noindent\textbf{Bound for $R_1, R_2, R_3$ on the event $A$.} Using the fact that $\ttinf{R_1} \le \opnorm{R_1}$, on the event $A$ we have
\begin{align}
    \ttinf{R_1} 
    &\le \opnorm\big{\qty\big(\Dc - \Delc) \qty\big(\hU - UQ_*) \hL^{-1/2}}\\
    &\le \opnorm\big{\Dc - \Delc} \cdot \opnorm\big{\hU - UQ_*} \cdot \opnorm\big{\hL^{-1/2}}\\
    &\lesssim  \sigma \sqrt{n} \cdot \frac{ \sigma\kappa^2 }{\sqrt{n}} \cdot \frac{\kappa}{\sqrt{n}}
    \quad\lesssim\quad \frac{ \sigma^2\kappa^3}{\sqrt{n}},\label{eq:ttinf-R1}
\end{align}
where the third line follows from \cref{lem:D-opnorm}, \cref{lem:eigen}\,\ref{lem:eigen-2} and \cref{lem:eigen}\,\ref{lem:eigen-1} with $K_\eps = C(\sigma + M)$. Similarly, using \cref{eq:ttinf-ABC} for ${R_2}$
\begin{align}
    \ttinf{R_2} 
    &\le \ttinf{U}\cdot \opnorm\big{\Lambda (U\tr\hU - Q_*)\hL^{-1/2}}\\ 
    &\le \ttinf{U}\cdot \opnorm{\Lambda} \cdot \opnorm\big{(U\tr\hU - Q_*)} \cdot \opnorm\big{\hL^{-1/2}}\\
    &\lesssim \frac{\Rx\kappa}{\sqrt{n}}\cdot \kappa^2n \cdot \frac{ \kappa^4\sigma ^2}{n} \cdot \frac{\kappa}{\sqrt{n}}
    \quad \lesssim\quad \frac{ \sigma ^2\Rx\kappa^8}{{n}},\label{eq:ttinf-R2}
\end{align}
using \cref{lem:gram}\,\ref{lem:gram-3}, \cref{lem:eigen}\,\ref{lem:eigen-2}, and \cref{lem:eigen}\,\ref{lem:eigen-1}, respectively. Again, using \cref{eq:ttinf-ABC} with $A = I$, we have
\begin{align}
    \ttinf{R_3} \le \ttinf{\Dc U} \cdot \opnorm\big{Q_* \hL^{-1/2} - \L^{-1/2}Q_*}.\label{eq:ttinf-R3-1}
\end{align}
From \cref{lem:eigen}\,\ref{lem:eigen-3}, we have
\begin{align}
    \opnorm\big{Q_* \hL^{-1/2} - \L^{-1/2}Q_*} \lesssim \frac{ \sigma\kappa^6 }{n},
\end{align}
and for the first term on the r.h.s. of \cref{eq:ttinf-R3-1}, using the fact that $\Delc = U \L U\tr$,
\begin{align}
    \ttinf{\Dc U} 
    &\le \ttinf{\Delc U} + \ttinf{\qty\big(\Dc - \Delc) U}\\
    &= \ttinf{U\L} + \opnorm{(\Dc - \Delc)U}\\
    &= \ttinf{U} \cdot \opnorm{\L} + \opnorm{\Dc - \Delc} \cdot \opnorm{U}\\
    &\lesssim \qty(\frac{\Rx\kappa}{\sqrt{n}} \cdot \kappa^2n) + \qty\Big( \sigma \sqrt{n} \cdot 1)
    \quad\lesssim\quad { \sigma \Rx\kappa^3}\sqrt{n}.
\end{align}
Combining the two bounds for \cref{eq:ttinf-R3}, we get
\begin{align}
    \ttinf{R_3} \lesssim \frac{ \sigma ^2\Rx\kappa^{9}}{\sqrt{n}}.\label{eq:ttinf-R3}
\end{align}
Since the bounds for $R_1, R_2$ and $R_3$ hold on the event $A$, combining \cref{eq:ttinf-R1}, \cref{eq:ttinf-R2} and \cref{eq:ttinf-R3}, it follows that for all $n > \enn_1$,
\begin{align}
    \pr\qty(\ttinf{R_1} + \ttinf{R_2} + \ttinf{R_3} \gtrsim \frac{ \sigma ^2\kappa^9\Rx}{\sqrt{n}}) = O(n^{-2} + n^{-r}).\label{eq:ttinf-R13-prob}
\end{align}

\noindent\textbf{Bound for $R_4$.} We begin by writing
\begin{align}
    R_4 
    = \qty\big(\Dc - \Delc) U \L^{-1/2}Q_* 
    = \qty\big(\Dc - \Delc) (U \L^{1/2})\L^{-1}Q_*
    = \qty(-\half H\Eps H) (H\X Q\tr) \L^{-1}Q_*,
\end{align}
where the final equality follows from \cref{eq:Dc} and $Q\in \orth{p}$ is given in \cref{eq:cs-exact-recovery}. Centering $\Eps$ on the r.h.s, we get
\begin{align}
    R_4 &= -\half \qty\Big(H \qty(\Eps - \E(\Eps)) H) (H\X) Q\tr\L\inv Q_* - \half \qty\Big(H\E(\Eps) H) (H\X) Q\tr\L\inv Q_*
\end{align}
and, using the triangle inequality along with \cref{prop:cape-ttinf},
\begin{align}
    \ttinf{R_4} &\lesssim \underbrace{\ttinf\Big{\qty\Big(H \qty({\Eps - \E(\Eps)}) H) (H\X) Q\tr\L\inv Q_*}}_{=:\Circled{1}} + \underbrace{\ttinf{\qty\Big(H\E(\Eps) H) (H\X) Q\tr\L\inv Q_*}}_{=:\Circled{2}}.
\end{align}
For $\Circled{2}$, using \cref{eq:ttinf-ABC} from \cref{prop:cape-ttinf} along with assumptions~\ref{assumption:compact}~\textit{\&}~\ref{assumption:expectation} and \cref{lem:eigen},
\begin{align}
    \Circled{2} &\le \norm{H\E(\Eps)H}_{\infty}\cdot \ttinf{H\X}\cdot \opnorm{Q\tr \L\inv Q_*} \lesssim \frac{\kappa^2\Rx}{\sqrt{n}}.
\end{align}
For $\Circled{1}$, using the fact that $\norm{H}_{\infty} = \max_i\sum_{j}\abs{h_{ij}}\le 2$ and $H^2 = H$, from \cref{eq:ttinf-ABC} it follows that
\begin{align}
    \Circled{1} 
    &\le \norm{H}_\infty \cdot \ttinf{(\Eps - \E(\Eps))H\X} \cdot \opnorm{Q\tr \L\inv Q_*}\\
    &\lesssim \ttinf{(\Eps - \E(\Eps))H\X} \cdot \opnorm\big{\L\inv}\\
    &\lesssim \underbrace{\ttinf{(\Eps - \E(\Eps))H\X}}_{=:\Circled{3}} \cdot \frac{\kappa^2}{n}.
\end{align}
For $\Circled{3}$, using the law of total expectation,
\begin{align}
    \pr\qty\Big( \ttinf{(\Eps - \E(\Eps))H\X} > t ) = \E_{\X}\qty\bigg[ \pr\qty\Big( \ttinf{(\Eps - \E(\Eps))H\X} > t \mid \X) ].
\end{align}
For fixed $\X$, let $\Z = H\X$ and $\Y = \Eps - \E(\Eps)$. From assumption~\ref{assumption:compact}, $\opnorm{\Z} \le \kappa\sqrt{n}$ and $\ttinf{\Z}\le \Rx$. Therefore, from \cref{prop:ttinf-concentration} it follows that there exists $n_1 \equiv n_1(r, q, \sigma, \kappa)$ such that conditional on $\X$ and for sufficiently large $n > n_1$,
\begin{align}
    \ttinf{(\Eps - \E(\Eps))H\X} \lesssim {(\kappa + \Rx)\sigma\sqrt{np\log{np} }},
\end{align}
with probability greater than $1 - O(n^{-2} + n^{-r})$. Since the above bound holds uniformly over $\X \in \bbX(\xpar)$, it follows that $\Circled{3} \lesssim {\sigma( \kappa + \Rx)\sqrt{np\log{np} }}$ with probability greater than $1 - O(n^{-2} + n^{-r})$. Plugging this back into the bound for $\Circled{1}$ and combining with $\Circled{2}$, it follows that with probability greater than $1 - O(n^{-2} + n^{-r})$,
\begin{align}
    \ttinf{R_4} \lesssim {\mathbf{\overline{c}}(\xpar)} \sigma\sqrt{\frac{p \log{np}}{n}} + \frac{\kappa^2\Rx}{\sqrt{n}}.
\end{align}
Further, there exists $n_2 \equiv n_2(\xpar, \sigma)$ such that 
$\kappa^2\Rx \lesssim \sigma(\kappa+\Rx)\sqrt{p \log{np}}$ for all $n > \max\qty{n_1, n_2}$, and we have
\begin{align}
    \pr\qty( \ttinf{R_4} \gtrsim \mathbf{\overline{c}}(\xpar){\sigma} \sqrt{\frac{p \log{np}}{n}} ) = O(n^{-2} + n^{-r}).\label{eq:ttinf-R4-prob}
\end{align}

\noindent\textbf{Combining the bounds.} Combining the bounds for $R_1, R_2, R_3$ and $R_4$ from \cref{eq:ttinf-R13-prob} and \cref{eq:ttinf-R4-prob} and using a union bound, it follows that
\begin{align}
    \pr\qty( \ttinf\big{\hX - g_*(\X)} \gtrsim \mathbf{\overline{c}}(\xpar) {\sigma} \sqrt{\frac{p \log{np}}{n}} + \frac{ \sigma ^2 \kappa^9\Rx}{\sqrt{n}} ) = O(n^{-2} + n^{-r});
\end{align}
finally exists $n_3 \equiv n_3(\xpar, \sigma)$ such that for all $n > n_3$,
\begin{align}
    \frac{\sigma\kappa^9\Rx}{\sqrt{n}}\lesssim {(\kappa + \Rx)} \sqrt{\frac{p \log{np}}{n}}.
\end{align}
Therefore, for sufficiently large $n > \enn_3(r, q, \xpar) := \max\qty{\enn_1, n_1, n_2, n_3}$ it follows that with probability greater than $1 - O(n^{-2} + n^{-r})$,
\begin{align}
    \ttinf\big{\hX - g_*(\X)} \lesssim {\mathbf{\overline{c}}(\xpar)}{\sigma} \sqrt{\frac{p \log{np}}{n}},\label{eq:precise-ttinf-bound}
\end{align}
which completes the proof.
\QED

\subsection{Proof of \cref{thm:minimax-frobenius}}
\label{proof:thm:minimax-frobenius}

Let $\Psi^*(\delta, \xi) := \xi$ and $\pi^* := \mathcal{N}(0, \sigma^2)$ be fixed, and let $\Theta^*(\xpar, \sigma) := \qty{ (\X, \Psi^*, \pi^*): \X \in \bbX }$ such that for each $\theta^* = (\X, \Psi^*, \pi^*) \in \Theta^*(\xpar, \sigma)$ the distribution $\prs_\X := \pr_{(\X, \Psi^*, \pi^*)}$ denotes the joint distribution of $D = \Delta(\X) + \Eps$ where $\eps_{ij} \simiid \mathcal{N}(0, \sigma^2)$ for all $i < j$, corresponding to the additive noise model in \cref{tab:noise-models}. Since $\Theta^*(\xpar, \sigma) \subset \Theta(\xpar, \sigma)$, it follows that
\begin{align}
    \inf_{\hX}\sup_{\Theta(\xpar, \sigma)} \pr_{(\X, \Psi, \pi)}\qty{ \loss\rmse\qty\big(\hX(D), \X) > s} \ge \inf_{\hX}\sup_{\X \in \bbX} \prs_{\X}\qty{ \loss\rmse\qty\big(\hX(D), \X) > s}.\label{eq:X-minimax}
\end{align}
Therefore, it suffices to establish a lower bound for the right hand side of \cref{eq:X-minimax}. As is standard in minimax analysis, we will establish the lower bound by reducing the problem to a hypothesis testing problem over a suitably large finite subset $\bbY \subset \bbX$ to get
\begin{align}
    \inf_{\hX}\sup_{\X \in \bbX} \prs_{\X}\qty{ \loss\rmse\qty\big(\hX(D), \X) > s} \ge \inf_{\hX}\sup_{\Y \in \bbY} \prs_{\Y}\qty{ \loss\rmse\qty\big(\hX(D), \Y) > s}.\label{eq:Y-minimax}
\end{align}
To establish the lower bound over $\bbY$, we will need the following general lower bound for testing multiple hypotheses using Fano's method.

\begin{lemma}[\citealp{tsybakov2008nonparametric},~Theorem~2.5]\label{lemma:fano}
    Let $\Theta$ be a parameter space equipped with a pseudometric $\rho(\cdot, \cdot)$, and suppose $\Theta$ contains $\{\theta_0, \theta_1, \dots, \theta_M\}$ elements for $M \ge 2$ such that (i) $d(\theta_j, \theta_k) \ge 2s$ for all $0 \le j < k \le M$, (ii) $\frac{1}{M}\sum_{i=1}^M \kl(\pr_{\theta_i} \| \pr_{\theta_0}) \le \alpha \log{M}$ with $0 < \alpha < 1/8$ where $\kl(\pr_{\theta_j} \| \pr_{\theta_0})$ is the Kullback-Leibler divergence. Then,
    \begin{align}
        \inf_{\hat\theta}\sup_{\theta\in\Theta} \pr_{\theta}\qty{ \rho(\hat\theta, \theta) > s } \ge \frac{\sqrt{M}}{1+\sqrt{M}}\qty( 1 - 2\alpha - \sqrt{\frac{2\alpha}{\log{M}}} ).
    \end{align}
\end{lemma}
For completeness, we establish that $\loss\rmse$ is a pseudometric on $\R^{n \times p}$ in \cref{lem:rho-rigid}. In view of \cref{lemma:fano}, the subset of $\bbY$ needs to be chosen to be sufficiently large such that the constituent configurations are well-separated in the $\loss\rmse$ pseudometric but whose distributions $\prs_{\Y}$ are close in $\kl(\cdot \| \cdot)$. 

To this end, set $\gamma := 2\kappa/(\kappa^2+1)$ and let $\X \in \bbX$ be chosen as follows. Let $U \in \R^{n \times p}$ be a matrix\footnote{In particular, $U$ is an orthogonal $p$-frame associated with the subspace of the Grassmann manifold $\text{Gr}(n, p)$ orthogonal $\text{span}{\qty{\onev}}$ \citep{edelman1998geometry} with  incoherence parameter $\Rx/\gamma\sqrt{p}$.} with orthonormal columns (i.e., $U\tr U = I_p$) such that $\ttinf{U}\le \Rx/2\sqrt{n}$ and its columns are orthogonal to $\onev$ (i.e., $\onev\tr U = \zerov$), and set $\X := \gamma\sqrt{n}U$. By construction, $\ttinf{\X} \le \Rx/2$,
\begin{align}
    \Xc := H\X = \qty(I - \frac{\onev\onev\tr}{n})\X = \X \qq{and} \frac{\X\tr\X}{n} = \gamma^2 I.\label{eq:X-properties-1}
\end{align}
Note that $\gamma$ is the harmonic mean of $\frac{1}{\kappa}$ and $\kappa$; therefore $\frac{1}{\kappa} \le s_p(H\X/\sqrt{n})= \gamma = s_1(H\X/\sqrt{n}) \le \kappa$ and it follows that $\X \in \bbX$. We construct $\bbY$ using $\X$ as follows. Set $m := \floor{n/2}$ and for $\tau \in \qty{0, 1}^m$ let $\omega(\theta) \in \qty{-1, 0, 1}^n$ be given by
\begin{align}
    \omega(\tau) = \begin{cases}
        (\tau, -\tau)\tr & \text{if $n$ is even},\\
        (\tau, 0, -\tau)\tr & \text{if $n$ is odd}.
    \end{cases}
    \label{eq:omega}
\end{align}
For $\eta \in (0, 1)$ and $v \in \Rp$  with $\norm{v}=1$, let $\bbY$ be a collection of $2^m$ configuration matrices~given~by
\begin{align}
    \bbY \equiv \bbY(\X, \eta, v) = \qty{ \Y(\tau) \in \Rnp : \Y(\tau) = \X + \eta \omega(\tau) v\tr,\; \tau \in \qty{0, 1}^m }.\label{eq:bbY}
\end{align}
Intuitively, each $\Y \in \bbY$ is a centered configuration such that $\loss\rmse(\Y, \X)$ is large but $\frobenius{\Del(\Y) - \Del(\X)}$ is sufficiently small. Moreover, when $\eta = o(1)$ is chosen to be sufficiently small, we have $\bbY \subset \bbX$. The following result establishes some key properties of $\bbY$ which will be crucial in establishing the lower bound. The proof is deferred to \cref{proof:lemma:Y-subset-X}.
\begin{lemma}\label{lemma:Y-subset-X}
    For $\gamma = 2\kappa / (\kappa^2 + 1)$, let $\X = \gamma\sqrt{n} U \in \Rnp$ be as chosen above and let $\bbY = \bbY(\X, \eta, v)$ be as defined in \cref{eq:bbY}. Then, for all $\Y \in \bbY$,
    \begin{enumerate}[label=\textup{(\roman*)}, ref=\cref{lemma:Y-subset-X}~(\roman*)]
        \item\label{eq:sY-bound} $\gamma - \eta \le s_p\qty(\frac{H\Y}{\sqrt{n}}) \le s_1\qty(\frac{H\Y}{\sqrt{n}}) \le \gamma + \eta$.
        \item\label{eq:DelY-bound} $\frobenius{\Del(\Y) - \Del(\X)}^2 \lesssim n^2( \gamma^2\eta^2 + \eta^4 )$.
        \item\label{eq:Y-separation} For all $\Y(\tau), \Y(\tau') \in \bbY$ and for all $\eta \le \gamma/8$, we have $$
        \loss\rmse\qty(\Y(\tau), \Y(\tau')) \gtrsim \frac{\gamma\eta}{\gamma + \eta} \cdot \sqrt{\frac{d_H(\tau, \tau')}{n}},$$ 
        where $d_H(\tau, \tau') = \half\norm{\tau - \tau'}_1$ is the Hamming distance between $\tau$ and $\tau'$.
    \end{enumerate}
\end{lemma}

With the above lemma in hand, we are now ready to establish the lower bound using \cref{lemma:fano}. 
By noting that $m = \lfloor n/2 \rfloor$ and using the Varshamov-Gilbert bound \citep[Lemma~2.9]{tsybakov2008nonparametric}, there exists $\qty{\tau_0, \tau_2, \dots, \tau_M} \subset \qty{0, 1}^m$ of size $\log{M} \gtrsim m \asymp n$ such that $H(\tau, \tau') \gtrsim m \asymp n$ and $\tau_0 = \zerov_m$. From \ref{eq:Y-separation}, if $\eta \le \gamma/8$ it follows that
\begin{align}
    \loss\rmse\qty\big(\Y(\tau_i), \Y(\tau_j)) \gtrsim \frac{\gamma\eta}{\gamma + \eta} \qq{for all} i \neq j.\label{eq:Y-separation-2}
\end{align}
Note that $\Y(\tau_0) = \X$ and that for each $i \in [M]$, $\pr_{i} := \pr_{Y(\tau_i)}$ is a matrix Normal distribution on the entries of $n \times n$ symmetric, hollow matrices with mean $\Delta(\Y(\tau_i))$ and variance $\sigma^2$. The Kullback-Leibler divergence is bounded by
\begin{align}
    \kl\qty\Big( \pr_{i} \| \pr_{0} ) = \frac{1}{2\sigma^2}\sum_{k < \ell}\norm{ \del_{k\ell}(\Y(\tau_i)) -  \del_{k\ell}(\X) }^2 
    &= \frac{1}{4\sigma^2} \norm{\Del(\Y(\tau_i)) - \Del(\X)}^2_F\\ 
    &\lesssim \frac{n^2}{\sigma^2}\cdot (\eta^2\gamma^2 + \eta^4)\\
    &\lesssim \log{M} \cdot \frac{n}{\sigma^2}(\eta^2\gamma^2 + \eta^4),\label{eq:kl-bound}
\end{align}
where the second line follows from \ref{eq:DelY-bound} and the third line follows by noting that $\log{M} \asymp n$. From \cref{eq:Y-separation-2} and \cref{eq:kl-bound}, and as per \cref{lemma:fano}, let
\begin{align}
    2s := c_1\frac{\gamma\eta}{\gamma + \eta} \qq{and} \alpha := c_2\frac{n}{\sigma^2}(\eta^2\gamma^2 + \eta^4),
\end{align}
where $c_1, c_2 > 0$ are absolute constants. For $C > 1$ universal sufficiently large\footnote{It suffices to take $C=\max\qty{8, 32c_2}$.}, let
\begin{align}
    \eta^2 = \frac{\sigma^2}{C \gamma^2n},
\end{align}
and let $\enn_3(\sigma, \kappa) = C\sigma^2\kappa^4$. For all $n \ge \enn_3(\sigma, \kappa)$, by noting that $\frac{1}{\kappa} \le \gamma$ we have
\begin{align}
    \eta^2 \le \frac{1}{C^2\gamma^2\kappa^4} \le \frac{\gamma^2}{C^2} \le \frac{\gamma^2}{64},
\end{align}
as required by \ref{eq:Y-separation}. It is also easy to verify that $\bbY \subset \bbX$ from \ref{eq:sY-bound}. Moreover,
\begin{align}
    &\alpha = c_2\frac{n}{\sigma^2}(\eta^2\gamma^2 + \eta^4) = c_2\qty(\frac{1}{C} + \frac{\sigma^2}{C^2n\gamma^4}) \le c_2\qty(\frac{1}{C} + \frac{\sigma^2\kappa^4}{C^2n}) \le \frac{2c_2}{C} \le \frac{1}{16}\\
    \implies & 1 - 2\alpha - \sqrt{\frac{2\alpha}{\log{M}}} > \frac{1}{2}.\label{eq:alpha-bound}
\end{align}
Plugging in this choice of $\eta$ into the expression for $s$, it follows that
\begin{align}
    2s 
    = c_1 \frac{\gamma\eta}{\gamma + \eta} 
    >\frac{c_1}{2} \eta
    = \frac{c_1\sigma}{2\gamma\sqrt{Cn}}
    = \frac{c_1\sigma}{4\sqrt{Cn}} \cdot \qty(\frac{\kappa^2 +1}{\kappa})
    \ge \frac{c_1\sigma\kappa}{4\sqrt{C}\sqrt{n}},\label{eq:s-bound}
\end{align}
where the first inequality follows from the fact that $\eta < \gamma$. By noting that $\log{M} \gtrsim n$, it follows that
$$
\frac{\sqrt{M}}{1+\sqrt{M}} = \frac{1}{1 + O(2^{-n})} = 1 - o(1),
$$ 
and using the bounds from \cref{eq:alpha-bound,eq:s-bound} in \cref{lemma:fano}, it follows that
\begin{align}
    \inf_{\hX}\sup_{\Y \in \bbY} \pr_{\Y}\qty{ \loss\rmse\qty\big(\hX(D), \X) > \frac{c_1}{8\sqrt{C}} \cdot \frac{\sigma\kappa}{\sqrt{n}} } \ge \frac{1}{2}\qty\Big(1 - o(1)),
\end{align}
when $n > \enn_3(\sigma, \kappa)$, and final result follows from \cref{eq:Y-minimax} and \cref{eq:X-minimax}.\qed

\subsection{Proof of \cref{thm:minimax-ttinf}}
\label{proof:thm:minimax-ttinf}

As in the proof of \cref{thm:minimax-frobenius}, let $\Psi^*(\delta, \xi) := \xi$ and $\pi^* := \mathcal{N}(0, \sigma^2)$ be fixed, and let $\prs_\X := \pr_{(\X, \Psi^*, \pi^*)}$ denote the joint distribution of $D = \Delta(\X) + \Eps$ where $\eps_{ij} \simiid \mathcal{N}(0, \sigma^2)$ for all $i < j$. We have
\begin{align}
    \inf_{\hX}\sup_{\Theta(\xpar, \sigma)} \pr_{(\X, \Psi, \pi)}\qty{ \loss\rmse\qty\big(\hX(D), \X) > s} 
    \ge 
    \inf_{\hX}\sup_{\X \in \bbX} \prs_{\X}\qty{ \loss\rmse\qty\big(\hX(D), \X) > s}.\label{eq:XY-minimax}
\end{align}
To establish the lower bound for the right hand side of \cref{eq:XY-minimax}, we will need the following result for testing multiple hypotheses using Le Cam's convex hull method.
\begin{lemma}[Adaptation of \citealp{yu1997assouad}, Lemma~1]\label{lem:le-cam}
    Let $\hat\theta$ be an estimator of $\theta(\pr)$ for $\pr \in \P$, taking values in a pseudometric space $(\Theta, \rho)$, and suppose there exist subsets $\Theta_0, \Theta_1 \subset \Theta$ which are $2s$-separated, i.e., $\rho(\theta_0, \theta_1) \ge 2s$ for all $\theta_i \in \Theta_i$, $i=0, 1$. Suppose also that $\P_0, \P_1$ are subsets of $\P$ for which $\theta(\pr_i) \in \Theta_i$ for $i=0,1$. Then,
    $$
    \inf_{\hat\theta}\sup_{\theta\in\Theta} \pr_{\theta}\qty{ \rho(\hat\theta, \theta) > s } \ge \half\qty( 1 - \inf_{{\bar\pr_i \in \textup{conv}(\P_i)}}\tv(\bar\pr_1, \bar\pr_2) ),
    $$
    where $\textup{conv}(\P_i) := \qty{ \sum_{k=1}^m{\lambda_k\pr_{ik}} : \pr_{ik} \in \P_i,\, \lambda_k > 0 \; \text{ and } \sum_{k=1}^m\lambda_k=1}$ is the convex hull of $\P_i$ and $\tv(\cdot, \cdot)$ is the total variation metric.
\end{lemma}

Note that by \cref{lem:rho-rigid}, $\loss\ttinft$ defines a pseudometric on $\bbX$. In view of \cref{lem:le-cam}, we now construct subsets $\bbX_0, \bbX_1 \subset \bbX$. To this end, and without loss of generality, suppose $n$ is even and set $m = n/2$. Choose $\U \in \R^{m \times p}$ similar to the construction in \cref{eq:X-properties-1} as follows. Let $V \in \R^{m \times p}$ with orthonormal columns satisfying $\ttinf{V}\le \Rx/2\sqrt{n}$, and let $U = [\half V\tr -\half V\tr]\tr \in \Rnp$ such that $U\tr U = V\tr V = I_p$. For $\gamma = 2\kappa/(\kappa^2+1)$, let $\X = \gamma\sqrt{n}U$. By construction, $X_k = -X_{k+m}$ for each $k \in [m]$, from which it  follows that $H\X = \X$ since $\onev\tr\X = \zerov$, and $\X \subset \bbX$ since
\begin{align}
    \ttinf{H\X} = \ttinf{\X} \le \frac{\gamma\Rx}{2} \le \frac{\Rx}{2} \qq{and} \frac{(H\X)\tr (H\X)}{n} = \frac{\X\tr\X}{n} = \gamma^2 I_p.\label{eq:X-properties-2}
\end{align}
Fix $\bbX_0 = \qty{\X}$ and let $\bbX_1$ be the collection of $n$ configurations given by
\begin{align}
    \bbX_1 \equiv \bbX_1(\X, \eta) := \qty\Big{ \X^k \in \Rnp: \X^k = \X + \eta e_k v(k)\tr,\; \text{where}\; v(k) = {X_k}/{\norm{X_k}} \qq{and} k\in [n]}.\label{eq:bbX-1}
\end{align}
Intuitively, each $\X^k \in \bbX_1$ is chosen such that $\loss\ttinft(\X^k, \X)$ is maximized while keeping $\smallnorm{\Del(\X^k) - \Del(\X)}_F$ sufficiently small. The following result precisely quantifies this. The proof is deferred to \cref{proof:lemma:X1-properties}.
\begin{lemma}\label{lemma:X1-properties}
    For $\gamma = 2\kappa/(\kappa^2+1)$, let $\X = \gamma\sqrt{n}U$ be as chosen above, and let $\bbX_1 = \bbX_1(\X, \eta)$ be as defined in~\cref{eq:bbX-1}. Then, for all $\X^k \in \bbX_1$,
    \begin{enumerate}[label=\textup{(\roman*)}, ref=\cref{lemma:X1-properties}~(\roman*)]
        \item\label{eq:X1-singular} $\smallnorm{H\X^k}\ttinft \le \eta + \gamma\Rx/2$ and $\gamma - \frac{1}{\sqrt{n}}\eta \le s_p\qty(\frac{H\X^k}{\sqrt{n}}) \le s_1\qty(\frac{H\X^k}{\sqrt{n}}) \le \gamma +  \frac{1}{\sqrt{n}}\eta$,
        \item\label{eq:X1-separation} $\loss\ttinft(\X^k, \X) \ge \eta/2$,
        \item\label{eq:X1-frobenius} $\smallnorm{\Del(\X^k) - \Del(\X)}^2_F \lesssim n\qty(\eta^4 + \eta^2(\gamma + \Rx)^2)$, and
        \item\label{eq:X1-frobenius-inner} For all $k \neq \ell$, $\big\langle{\Del(\X^k) - \Del(\X),\; \Del(\X^\ell) - \Del(\X)}\big\rangle_F \lesssim \eta^4 + \eta^2\Rx^2$.
    \end{enumerate}
\end{lemma}
We are now ready to establish the lower bound for \cref{eq:XY-minimax} using \cref{lem:le-cam}. Let $\prs_0$ be the joint distribution of $D = \Del(\X) + \Eps$ for $\X \in \bbX_0$, and define 
\begin{align}
    \bprs := \frac{1}{n}\sum_{k \in [n]}\prs_k \in \text{conv}\qty\big(\qty{\prs_1, \prs_2, \cdots, \prs_n}),\label{eq:bar-prs}
\end{align}
where $\prs_k$ denotes the joint distribution of $D = \Del(\X^k) + \Eps$ for $\X^k = \X + \eta e_kv\tr \in \bbX_1$ for each $k \in [n]$. The following lemma bounds the total variation distance between $\prs_0$ and $\bprs$. The proof is deferred to \cref{proof:lemma:tv-bound}.
\begin{lemma}\label{lemma:tv-bound}
    Let $\bbX_0$, $\prs_0$ and $\bbX_1 \equiv \bbX(\X, \eta)$ be as given above and let $\bprs$ be as defined in \cref{eq:bar-prs}. Then,
    \begin{align}
        \tv(\prs_0, \bprs)^2
        &\le \exp(\frac{1}{\sigma^2} \cdot {\max_{k}\smallnorm{\Del(\X^k)\!-\!\Del(\X)}^2_F} - \log{n})\\
        &\qquad\qquad + \exp(\frac{1}{\sigma^2} \cdot {\max_{k \neq \ell}{\big\langle{\Del(\X^k)\!-\!\Del(\X),\; \Del(\X^\ell)\!-\!\Del(\X)}\big\rangle_F}}) - 1.\label{eq:tv-bound}
    \end{align}
\end{lemma}
In view of \cref{lemma:tv-bound}, for $C > 1$ universal sufficiently large\footnote{It suffices to take $n = 4\max\qty{c_1, c_2}$ where $c_1, c_2$ are the absolute constants in \ref{eq:X1-frobenius} and \ref{eq:X1-frobenius-inner}, respectively.}, set
\begin{align}
    \eta^2 := \frac{\sigma^2}{C(\gamma+\Rx)^2}\frac{\log{n}}{n}.
\end{align}
There exists $n_1 \equiv n_1(\sigma, \xpar)$ such that for all $n> n_1$,
$$
\frac{1}{\kappa} \le \gamma - \frac{\sigma}{\sqrt{C}(\gamma + \Rx)}\cdot\sqrt{\frac{\log{n}}{n^2}} \le \gamma + \frac{\sigma}{\sqrt{C}(\gamma + \Rx)}\cdot\sqrt{\frac{\log{n}}{n^2}} \le \kappa;
$$
then from \ref{eq:X1-singular} it follows that $\bbX_1 \subset \bbX$ for all $n > n_1$. Let $n_2(\sigma, \xpar)$ be such that $n(\kappa+\Rx)^4 > \sigma^2\log{n}$ for all $n > n_2$. Plugging in this choice of $\eta$ into \ref{eq:X1-frobenius}, there exists an absolute constant $c_1 > 0$ such that for all $n > \enn_5(\sigma, \xpar) := \max\qty{n_1, n_2}$,
\begin{align}
    \frac{\smallnorm{\Del(\X^k) - \Del(\X)}^2_F}{\sigma^2} &\le \frac{c_1n\qty((\gamma+\Rx)^2\eta^2 + \eta^4)}{\sigma^2}\\ 
    &= c_1\qty(\frac{1}{C} + \frac{\sigma^2\log{n}}{C^2(\gamma+\Rx)^4n}) \log{n} \\
    &\le \frac{2c_1}{C}\log{n} \le \frac{1}{2}\log{n}\label{eq:tv-bound-1}
\end{align}
and from \ref{eq:X1-frobenius-inner} there exists $c_2 > 0$ such that,
\begin{align}
    \frac{\max_{k \neq \ell}\inner{\Del(\X^k) - \Del(\X),\; \Del(\X^\ell) - \Del(\X)}[F]}{\sigma^2} 
    &\le \frac{c_2((\gamma+\Rx)^2\eta^2 + \eta^4)}{\sigma^2}\\ 
    &= {c_2}\qty(\frac{\Rx^2}{C(\gamma+\Rx)^2} + \frac{\sigma^2\log{n}}{C^2(\gamma+\Rx)^4n})\frac{\log{n}}{n}\\[5pt]
    &\le \frac{2c_2}{C} \cdot \frac{\log{n}}{n} = o(1).\label{eq:tv-bound-2}
\end{align}
Substituting \cref{eq:tv-bound-1,eq:tv-bound-2} into \cref{lemma:tv-bound} and by noting that $\exp\qty\big(o(1)) \le 1 + o(1)$, we have
\begin{align}
    \tv(\prs_0, \bprs) \le \sqrt{ \exp( \frac{1}{2}\log{n} - \log{n} ) + \exp\qty\Big(o(1)) - 1 } = \sqrt{ \frac{1}{\sqrt{n}} + o(1) } = o(1).\label{eq:tv-bound-3}
\end{align}
Moreover, by noting that ${\gamma + \Rx} =  \frac{2\kappa}{\kappa^2+1} + \Rx \le 2(\frac{1}{\kappa} + \Rx)$, from \ref{eq:X1-separation} we also have
\begin{align}
    \loss\ttinft(\X^k, \X) \ge \frac{\eta}{2} = \frac{\sigma}{2\sqrt{C}(\gamma+\Rx)}\sqrt{\frac{\log{n}}{n}} \gtrsim \frac{\sigma\kappa}{1 + \kappa\Rx}\sqrt{\frac{\log{n}}{n}}.\label{eq:X1-separation-1}
\end{align}
Therefore, using \cref{eq:tv-bound-3} and \cref{eq:X1-separation-1} in \cref{lem:le-cam}, it follows that
\begin{align}
    \inf_{\hX}\sup_{\X \in \bbX} \prs_{\X}\qty{ \loss\ttinft\qty\big(\hX(D), \X) > \frac{\sigma\kappa}{1 + \kappa\Rx}\sqrt{\frac{\log{n}}{n}} } \ge \half\qty(1 - o(1)),
\end{align}
for all $n > \enn_5(\sigma, \xpar)$, and the final result follows from \cref{eq:XY-minimax}.\qed


\bibliographystyle{chicago}
\bibliography{refs}

\clearpage
\appendix


\begin{table}[h]
    \caption{Table of Notations}
    \label{tab:notation}\smallskip
    \begin{tabularx}{\textwidth}{lX}
        \toprule
        \textbf{Notation} & \textbf{Description} \\
        \midrule
        $\onev, \zerov, e_k$ & The vector of ones, the vector of zeros, $k$th basis vector \\
        $\norm{x}$ & $\ell_2$ norm of $x \in \R^p$ \\
        $I, J, H$ & The identity matrix, the matrix of all ones and the centering matrix $H = I - J/n$ \\
        $A\pinv$ & The Moore-Penrose pseudoinverse of a full rank matrix $A$ \\
        $\opnorm{A}, \ttinf{A}, \norm{A}_F$ & The $\ell_2 \to \ell_2$ operator norm, $\ell_2 \to \ell_\infty$ norm, and Frobenius norm of $A$\\
        $A \circ B$ & The Hadamard product of two matrices $A$ and $B$ \\
        $\X$ & The $n \times p$ configuration matrix with rows $X_1, \cdots, X_n \in \Rp$ \\
        $\Del$ & The pairwise squared Euclidean distance matrix $\del_{ij} = \norm{X_i - X_j}^2$ \\
        $\Xi$ & An $n \times n$ symmetric, hollow random matrix \\
        $\Eps$ & The $n \times n$ symmetric, hollow matrix of errors $\Eps = \Psi(\Del, \Xi)$ \\
        $D$ & The observed distance matrix $D = \Del + \Eps$ \\
        $\Dc, \Delc$ & The double centered dissimilarity matrices $\Dc = -\half HDH$ and $\Delc = -\half H\Del H$ \\
        $\sigma, \kappa, \Rx$ & Bounds on $\max_{ij}\E[|\eps_{ij}|^q \mid \del_{ij}]^{1/q}$, $s_1(H\X/\sqrt{n})$ and $\ttinf{H\X}$, respectively. \\
        \bottomrule
    \end{tabularx}
\end{table}



\section{Auxiliary Results and Deferred Proofs}\label{sec:auxiliary}

In this section, we present some auxiliary results that are used in the proofs of our main theorems along with the deferred proofs for some technical lemmas. \cref{sec:upper-bound-results} contains results that are used to establish upper bounds for \cref{sec:consistency}, and \cref{sec:lower-bound-results} contains the proofs for the lemmas used to establish lower bounds from \cref{sec:lower-bound}.

\subsection{Results for Upper Bounds}
\label{sec:upper-bound-results}

The following result is an adaptation of the symmetrization lemma for tail probabilities (\citealp[Lemma~2.3.7~of][]{van2023weak} and \citealp[Proposition~3.1.24~of][]{gine2016mathematical}) for the $\ell_{2}$-operator norm of a random matrix $\Y \in \R^{n \times n}$. The result is used to establish the tail probability bounds in \cref{lem:D-opnorm} and \cref{thm:X-ttinf}.

\begin{lemma}\label{lem:opnorm-symmetrization}
    Let $\Y = (Y_{ij}) \in \R^{n \times m}$ be a random matrix with mean zero entries. Then, for every fixed matrix $P \in \R^{n \times m}$ and for all $t> 0$,
    \begin{align}
        \beta(t) \pr\qty\Big( \opnorm{\Y} > t ) \le 2 \pr\qty\Big( \opnorm{G \circ (\Y - P)} > \frac{t}{4} ),\label{eq:opnorm-symmetrization}
    \end{align}
    where $G = (g_{ij}) \in \R^{n \times m}$ is a random matrix of independent Rademacher random variables with $G \indep \Y$, and $\beta(t) = \inf_{\norm{v}=1}{\pr( \smallnorm{\Y v} \le t/2 )}$. Moreover, for all $t > 4 \E\opnorm{\Y}$,
    \begin{align}
        \pr\qty\Big( \opnorm{\Y} > t ) \le 4 \pr\qty\Big( \opnorm{G \circ (\Y - P)} > \frac{t}{4} ).\label{eq:opnorm-symmetrization-markov}
    \end{align}
\end{lemma}

\begin{proof}
    The proof is similar to the proof of Proposition~3.1.24~of~\citealp{gine2016mathematical}. Let $\Y'$ be an \iid{} copy of $\Y$ with $\Y' \indep \Y$, and let $\pr_{\Y}$ and (resp. $\pr_{\Y'}$) be the probability measure associated with $\Y$ (resp. $\Y'$) on the sample space $\Omega$ (resp. $\Omega'$). For the unit $m$-sphere $\mathcal{V} := \qty\big{v \in \R^m: \norm{v}=1}$, let $f_n(v) := \smallnorm{\Y v}$ and $f_n'(v) := \smallnorm{\Y' v}$ for $v \in \cal{V}$. We can write
    $$\opnorm{\Y}
    = \sup_{v \in \cal{V}}\smallnorm{\Y v} 
    = \sup_{v \in \cal{V}}f_n(v)
    $$
    as the supremum of an empirical process indexed over $\cal{V}$. For $s < t$, set ${\beta(2s) = \inf_{v \in \cal{V}}\pr_{\Y'}( f_n'(v) \le s )}$ and let $A := \qty{ \omega: \sup_{v \in \cal{V}}f_n(v) > t}$ be the event on which $\opnorm{\Y}$ exceeds $t$. If $A$ holds, then for every $\omega \in A$ there exists a realization of the random matrix $\Y(\omega) \in \Rnn$ and $v_0 \equiv v_0(\omega) \in \cal{V}$ such that $f_n(v_0; \omega) \equiv \abs{v_0\tr \Y(\omega) v_0} > t$. Therefore, for every $\omega \in A$,
    \begin{align}
        \beta(2s)
        = \inf_{v \in \cal{V}}\pr_{\Y'}( f_n'(v) \le s )
        \le \pr_{\Y'}\qty\Big( f_n'(v_0) \le s )
        &\le \pr_{\Y'}\qty\Big( \abs{f_n(v_0; \omega) - f_n'(v_0)} > t-s )\\
        &\le \pr_{\Y'}\qty( \sup_{v \in \cal{V}}\abs{f_n(v; \omega) - f_n'(v)} > t-s ),\label{eq:beta-s}
    \end{align}
    where the second inequality follows from by noting that if $f_n'(v_0) < s$ and $\omega \in A$ then $$\abs{f_n(v_0; \omega)-f_n'(v_0)} > f_n(v_0; \omega) - f_n'(v_0) > t-s.$$ Moreover, by the reverse triangle inequality,
    $$
    \sup_{v \in \cal{V}} \abs{f_n(v; \omega) - f_n'(v)} \le \sup_{v \in \cal{V}}\smallnorm{(\Y(\omega) - \Y')v} = \opnorm\big{\Y(\omega)-\Y'},
    $$
    and since \cref{eq:beta-s} holds for every $\omega \in A$ where ${A = \qty\big{\omega: \opnorm{\Y(\omega)} > t}}$, we have
    \begin{align}
        \beta(2s) \mathbb{1}_A(\omega) \le \pr_{\Y'}\qty\Big( \opnorm{\Y(\omega)-\Y'} > t-s ) \mathbb{1}_A(\omega).
    \end{align}
    Integrating with respect to $\pr_\Y$, we get
    \begin{align}
        \beta(2s) \cdot \pr_\Y\qty\Big( \opnorm{\Y} > t ) \le \pr_{\Y, \Y'}\qty\Big( \opnorm{\Y-\Y'} > t-s  ).
    \end{align}
    Note that $(Y_{ij} - Y'_{ij}) \stackrel{d}{=} (Y'_{ij} - Y_{ij}) \stackrel{d}{=} g_{ij} (Y_{ij} - Y'_{ij})$ for every matrix $G = (g_{ij}) \in \qty{-1, +1}^{n \times n}$, and since $\Y - \Y' \stackrel{d}{=} (\Y - P) - (\Y' - P)$ for every fixed matrix $P \in \Rnn$, it follows that
    \begin{align}
        \beta(2s) \cdot \pr_\Y\qty\Big( \opnorm{\Y} > t ) 
        &\le \E_{R}\qty[\pr_{\Y, \Y'}\qty\Big( \opnorm{G \circ (\Y - \Y')} > t-s  )]\\
        &= \pr_{G,\Y, \Y'}\qty\Big( \opnorm{G \circ ((\Y - P)-(\Y' - P))} > t-s  )\\
        &\le 2\pr_{G,\Y}\qty\Big( \opnorm{G \circ (\Y - P)} > \frac{t-s}{2}  ),
    \end{align}
    where $G$ is a matrix of \iid{} Rademacher random variables with $G \indep \Y$ and $\E_{G}$ is the expectation with respect to $G$, and the final line results from an application of the triangle inequality followed by a union bound. By taking $s = t/2$, we get
    \begin{align}
        \beta(t) \cdot \pr_\Y\qty\Big( \opnorm{\Y} > t ) \le 2\pr_{G, \Y}\qty\Big( \opnorm{G \circ (\Y - P)} > \frac{t}{4}  ),
    \end{align}
    which establishes \cref{eq:opnorm-symmetrization}. Moreover, by Markov's inequality,
    \begin{align}
        \beta(t) = \inf_{v \in \cal{V}}\pr\qty( \norm{\Y v} \le \frac{t}{2} ) = 1 - \sup_{v \in \cal{V}}\pr\qty( \norm{\Y v} > \frac{t}{2} ) \ge 1 - \frac{2}{t}\sup_{v \in \cal{V}}\E(\norm{\Y v}) \ge 1 - \frac{2}{t}\E\opnorm{\Y}.
    \end{align}
    Therefore, for all $t > 4\E\opnorm{Y}$, it follows that $\beta(t) \ge \half$ and, plugging this into \cref{eq:opnorm-symmetrization} gives the desired result in \cref{eq:opnorm-symmetrization-markov}.
\end{proof}

The proof of \cref{lem:D-opnorm} uses the following bound due to \cite{latala2005some} which bounds the $\ell_2$-operator norm of a random matrix with independent entries and a finite $4$th moment. We restate the result here for symmetric $n \times n$ matrices. 

\begin{proposition}[Theorem~2 of \citealp{latala2005some}]\label{prop:latala}
    Suppose $\Y = (Y_{ij})$ is an $n \times n$ random matrix with independent zero mean random variables with $\norm{Y_{ij}}_{L^4} < \infty$. Then, there exists a universal constant $C > 0$ such that
    \begin{align}
        \E(\opnorm{\Y}) \le C \qty( \max_{i}\sqrt{\sum_{j}\E Y_{ij}^2} + \sqrt[4]{\sum_{i, j}\E Y_{ij}^4} )
    \end{align}
\end{proposition}

The following result is a quantitative version of \citeauthor{latala2005some}'s result. When the entries $Y_{ij}$ are \iid{}, similar upper bound is also obtained in \cite{mendelson2012generic} as part of a ``\textit{quantitative Bai-Yin Theorem}", and uses more advanced machinery from generic chaining and empirical processes. Even though the following result is not stronger than \cref{prop:latala}, it nevertheless provides a meaningful extension of \cite{mendelson2012generic} to non-identical entries, and proof is a simple application of the results in \cite{bandeira2016sharp}.

\begin{proposition}\label{prop:latala-quantitative}
    Suppose $\Y = (Y_{ij})$ is a symmetric $n \times n$ random matrix with independent, symmetric, zero mean entries with $\max_{i,j}\norm{Y_{ij}}_{L^q} \le \sigma$ for $q > 4$. Then, for any $\eps \in (0, \half]$ there exists a universal constant $c_\eps > 0$ such that for all $0 < r \le (q-4)/2$ and for all $t > 0$, with probability greater than $1 - e^{-t} - n^{-r}$,
    \begin{align}
        \opnorm{\Y} \le 2(1+\eps) \qty{\max_{i}\sqrt{\sum_{j}\E Y_{ij}^2}} + c_\eps \sigma n^{\frac{r+2}{q}} \sqrt{t + \log{n}}.\label{eq:latala-quantitative-1}
    \end{align}
    In particular, there exists an absolute constant $C > 0$ such that with probability greater than $1 - (n^{-2} + n^{-r})$,
    \begin{align}
        \opnorm{\Y} \le C \qty({\max_{i}\sqrt{\sum_{j}\E Y_{ij}^2}} + \sigma n^{\frac{r+2}{q}}\sqrt{\log{n}}).\label{eq:latala-quantitative-2}
    \end{align}
\end{proposition}

\begin{remark}
    \cref{prop:latala-quantitative} can be extended to rectangular $n\times m$ matrices using Theorem~3.1 from \cite{bandeira2016sharp} to yield the general quantitative analogue of Theorem~2 from \cite{latala2005some}. For the purposes of this paper, we only require the symmetric $n\times n$ matrix version.
\end{remark}

\begingroup
\renewcommand{\hy}{{\widehat{Y}}}
\renewcommand{\hY}{{\widehat{\Y}}}
\newcommand{\cy}{{\widecheck{Y}}}
\newcommand{\cY}{{\widecheck{\Y}}}
\begin{proof}
    For $K >0$ to be chosen later, let $\hY = (\hy_{ij})$ and $\cY = (\cy_{ij})$ be given by 
    \begin{align}
        \hy_{ij} := Y_{ij}\mathbb{1}\qty{\abs{Y_{ij}} \le K} \qq{and} \cy_{ij} := Y_{ij}\mathbb{1}\qty{\abs{Y_{ij}} > K}.
    \end{align}
    We have $Y_{ij} = \hy_{ij} + \cy_{ij}$ and $\Y = \hY + \cY$. Moreover, since each $Y_{ij}$ is symmetric, it follows that $\hy_{ij}$ are also independent, symmetric, mean zero random variables with $\smallnorm{\hy{}_{ij}}_{L^\infty} \le K$, and 
    $$
    \E(\hy_{ij}^2) = \E(Y_{ij}^2\mathbb{1}\qty{Y_{ij} \le K}) \le \E(Y_{ij}^2).
    $$
    Therefore, from \citet[Corollary~3.12 and Remark~3.13]{bandeira2016sharp}, it follows that for any $\eps \in (0, \half]$ there exists a universal constant $c_\eps$ such that for all $t \ge 0$, with probability greater than $1 - e^{-t}$,
    \begin{align}
        \opnorm\big{\hY} 
        &\le 2(1+\eps) \qty{\max_{i}\sqrt{\sum_{j}\E Y_{ij}^2}} + c_\eps K\sqrt{ t + \log{n}} =: \zeta_n(K),\label{eq:quantitative-1}
    \end{align}
    where we have used the fact that $\E(\hy_{ij}^2) \le \E (Y_{ij}^2)$. By the triangle inequality, we have
    \begin{align}
        \pr\qty\Big(\opnorm{\Y} > \zeta_n(K)) \le \pr\qty\Big( \smallnorm{\hY}_2 > \zeta_n(K)) + \pr\qty\Big(\smallnorm{\cY}_2 > 0) \le e^{-t^2} + \pr\qty\Big(\smallnorm{\cY}_2 > 0).\label{eq:quantitative-2}
    \end{align}
    Note that the event $\{ \smallnorm{\cY}_2 > 0 \} = \qty{ \exists (i, j) \;\text{s.t. }\; \cy_{ij} > 0 }$. By a union bound, it follows that
    \begin{align}
        \pr(  \smallnorm{\cY}_2 > 0  )  = n^2 \cdot \pr( \cy_{ij} > 0 ) &\le n^2 \cdot \pr( Y_{ij} > K )\\
        &\le n^2 \cdot \frac{\norm{Y_{ij}}^q_{L^q}}{K^q} \le \frac{n^2 \sigma^q}{K^q},
    \end{align}
    where the second line follows from Markov's inequality. For the probability in the above display to be stochastically bounded, we require $K = o(n^{2/q})$; since $q > 4$, this also guarantees that $\zeta_n(K) = O(\sqrt{n})$. Setting $K = \sigma n^{(r+2)/q}$ for some $0 < r < (q-4)/2$ it follows that
    \begin{align}
        \frac{n^2 \sigma^q}{K^q} = \frac{1}{n^r} \qq{and} K \ll n^{(q-2)/2q} \ll \sqrt{n}.\label{eq:K-choice}
    \end{align}
    Plugging in $K=\sigma n^{(r+2)/2}$ in $\zeta_n(K)$ and in \cref{eq:quantitative-2} gives
    \begin{align}
        \pr\qty( \opnorm{\Y} >  2(1+\eps) \qty\bigg{\max_{i}\sqrt{\sum_{j}\E Y_{ij}^2}} + c_\eps \sigma n^{\frac{(r+2)}{q}} \sqrt{ t + \log{n}}  ) \le e^{-t} + \frac{1}{n^r}.\label{eq:quantitative-3}
    \end{align}
    The claim in \cref{eq:latala-quantitative-2} follows by setting $t = {2\log{n}}$ and taking $C = \max\qty{2(1+\eps), c_\eps\sqrt{3}}$ with $\eps = 1/2$.
\end{proof}
\endgroup

The next result is an analogue of \cref{prop:latala-quantitative} for bounding $\ttinf{\Y \Z}$ where $\Y \in \R^{n \times n}$ is a random matrix and $\Z \in \R^{n \times p}$ is a fixed matrix. The result is used to establish the upper bound in \cref{thm:X-ttinf}.

\begin{proposition}\label{prop:ttinf-concentration}
Suppose $\Y = (Y_{ij})$ is an $n \times n$ random matrix such that each row $\Y_{i, *} \in \Rn$ consists independent, zero mean entries with $\max_{i,j}\norm{Y_{ij}}_{L^q} \le \sigma$ for $q > 4$. Let $\Z \in \R^{n \times p}$ be a fixed matrix with $\opnorm{\Z} \le \sqrt{n} \kappa$ and $\ttinf{\Z} \le \Rx$.

Then, for all $0 < r < (q-4)/2$ and for all $t > 0$ there exists an absolute constant $C > 0$ such that with probability greater than $1 - O(e^{-t} + n^{-r})$,
\begin{align}
    \ttinf{\Y\Z} \le C \qty({\sigma ( \kappa + \Rx) \sqrt{np (t + \log{np})} + n^{\frac{(r+2)}{q}}\sigma\Rx(t + \log{np})}).
    \label{eq:ttinf-concentration-1}
\end{align}
In particular, for sufficiently large $n > \enn(r, q, \xpar)$ it follows that
\begin{align}
    \ttinf\big{\Y\Z} \lesssim \sigma( \kappa + \Rx)\sqrt{np\log{np} },
    \label{eq:ttinf-concentration-2}
\end{align}
with probability greater than $1 - O(n^{-2} + n^{-r})$.
\end{proposition}

\begingroup
\begin{proof}
\newcommand{\cEps}{\widecheck{\Eps}}
\newcommand{\ceps}{\widecheck{\eps}}
\newcommand{\cEpsk}{\cEps^K}
\newcommand{\cepsk}{\ceps^K}
\newcommand{\hEps}{\widehat{\Eps}}
\newcommand{\heps}{\widehat{\eps}}
\newcommand{\hEpsk}{\hEps^K}
\newcommand{\hepsk}{\heps^K}
\renewcommand{\hy}{{\widehat{Y}}}
\renewcommand{\hY}{{\widehat{\Y}}}
\newcommand{\cy}{{\widecheck{Y}}}
\newcommand{\cY}{{\widecheck{\Y}}}

For $j \in [n]$, let $z_j \in \Rp$ denote the $j$th row of $\Z$, and for each $i \in [n]$ the $i$th row of $\Y\Z$, given by $(\Y\Z)_{i,*} = \sum_{j}Y_{ij}z_j \in \Rp$, is the sum of $n$ independent, zero mean random vectors.

\noindent\textbf{Step 1: Symmetrization.} For $\mathcal{V} = \qty{v \in \Rp: \norm{v}=1}$, we can write $\norm{(\Y\Z)_{i,*}}$ as the supremum of an empirical process indexed over $\mathcal{V}$, given by
$\smallnorm{(\Y\Z)_{i,*}} = \sup_{v \in \cal{V}}|{ \sum_{j \in [n]} Y_{ij}z_j\tr v }|.$ For $\gamma^2 := \sup_{v \in \cal{V}} \max_{j}\E[(Y_{ij}z_j\tr v)^2]$ and for all $u \ge \sqrt{2n\gamma^2}$, using the symmetrization lemma for tail probabilities \citep[Corollary~3.1.25]{gine2016mathematical},
\begin{align}
    \pr\qty\bigg( \norm\Big{\sum_{j}Y_{ij}z_j} > 2u + \sqrt{2n\gamma^2} ) \le 4\pr\qty\bigg( \norm\Big{\sum_{j}g_j(Y_{ij}z_j)} > u ),
\end{align}
where $g = (g_1, g_2, \dots, g_n) \in \qty{-1, +1}^n$ is a vector of independent Rademacher random variables. By the assumption on $\Y$ and $\Z$, 
\begin{align}
    \gamma^2 = \sup_{v \in \cal{V}} \max_{j}\E[(Y_{ij}z_j\tr v)^2] \le \sigma^2 \max_{j}\norm{z_j}^2 = \sigma^2 \ttinf{\Z}^2 \le \sigma^2 \Rx^2,
\end{align}
and writing $\tilde{Y}_{ij} := g_jY_{ij}$, it follows that for all $u > 2\sigma\Rx\sqrt{n}$,
\begin{align}
    \pr\qty\bigg( \norm\Big{\sum_{j}Y_{ij}z_j} > 4u ) \le 4\pr\qty\bigg( \norm\Big{\sum_{j}\tilde Y_{ij}z_j} > u ).\label{eq:ttinf-quantitative-3}
\end{align}

\noindent\textbf{Step 2: Upper bound for \cref{eq:ttinf-quantitative-3}.} Note that $(\tilde{Y}_{i1}, \dots, \tilde{Y}_{in})$ are independent, zero mean and symmetric for each $i \in [n]$. Similar to the proof of \cref{prop:latala-quantitative}, for $K > 0$ to be specified later, let $\hy_{ij}$ and $\cy_{ij}$ be given by $\hy_{ij} = \tilde Y_{ij} \mathbb{1}\qty\big{|{\tilde Y_{ij}}| \le K}$ and $\cy_{ij} = \tilde Y_{ij} \mathbb{1}\qty\big{|{\tilde Y_{ij}}| > K}$. By symmetry of each $\tilde Y_{ij}$, it follows that $\cy_{ij}$ and $\hy_{ij}$ are also symmetric with mean zero. For all $u > 0$, by the triangle inequality and using the union bound,
\begin{align}
    \pr\qty\Big( \norm\Big{\sum_{j}\tilde Y_{ij}z_j} > u) \le \underbrace{\pr\qty\Big( \norm\Big{\sum_{j}\hy_{ij}z_j} > u)}_{=:\Circled{1}} + \underbrace{\pr\qty\Big( \norm\Big{\sum_{j}\cy_{ij}z_j} > 0)}_{=:\Circled{2}}.\label{eq:ttinf-quantitative-4}
\end{align}
Similar to the proof of \cref{prop:latala-quantitative}, we bound the terms $\Circled{1}$ and $\Circled{2}$ separately. For each fixed $i \in [n]$, 
\begin{align}
    \norm\Big{ \sum_{j} \cy_{ij} z_j } \le \max_{j}\norm\big{z_j} \cdot \max_{j}\abs\big{\cy_{ij}} \le \ttinf{\Z} \cdot \max_{j}\abs\big{\tilde Y_{ij}}\mathbb{1}\qty{ |{\tilde Y_{ij}}| > K },
\end{align}
and it follows that the event $\qty\big{\smallnorm{ \sum_{j} \cy_{ij} z_j } > 0} \subseteq \qty\big{ \exists j : \max_{j} \abs\big{\tilde Y_{ij}} > K }$. Since $\abs{g_j}=1$, we also have $\smallnorm{\tilde Y_{ij}}_{L^q} = \norm{g_j Y_{ij}}_{L^q} \le \norm{Y_{ij}}_{L^q} \le \sigma$ for all $i, j \in [n]$. By applying a union bound followed by Markov's inequality,
\begin{align}
    \Circled{2} = \pr\qty\Big( \norm\Big{\sum_{j}\cy_{ij}z_j} > 0) \le \pr\qty({\max_{j}\abs\big{\tilde Y_{ij}} > K})
    &\le n \max_{j}\pr\qty\Big({\abs\big{\tilde Y_{ij}} > K}) \le n \max_{j}\frac{\norm{Y_{ij}}_{L^q}^q}{K^q} \le n \frac{\sigma^q}{K^q}.
\end{align}

For $\Circled{1}$, note that $\smallnorm{\hy_{ij}z_j}\le K\Rx$ a.s. for all $i, j$, and $\sum_{j \in [n]}\hy_{ij}z_j$ is the sum of $n$ independent zero mean random vectors. Using the matrix Bernstein inequality \citep[Theorem~5.4.1 and Exercise~5.4.15]{vershynin2018high}, there exists $C > 0$ such that for all $t > 0$,
\begin{align}
    \pr\qty\Bigg(\norm\Big{\sum_{j}\hy_{ij}z_j} \ge C \qty({\sqrt{ \varsigma^2_i (t + \log{p}) } + K\Rx(t + \log{p})})) \le 2e^{-t},\label{eq:ttinf-concentration-3}
\end{align}
where $\varsigma^2_i := \smallnorm{\sum_{j} \E(\hy_{ij}^2 \cdot z_jz_j\tr)}_2$. We further have, $\E(\hy_{ij}^2) \le \E(Y_{ij}^2) \le \sigma^2$ and $z_jz_j\tr \preccurlyeq \trace(z_jz_j\tr)I_{p}$, from which it follows that
\begin{align}
    \norm\Big{ \sum_{j}\E(\hy_{ij}^2 \cdot z_jz_j\tr) }_{^{^2}} \le \sigma^2 \trace\qty\Big(\sum_{j} z_jz_j\tr) = \sigma^2 \trace(\Z\tr\Z) = \sigma^2 p\opnorm{\Z}^2 \le \sigma^2p \cdot  \kappa^2n,
\end{align}
Plugging this into the bound from \cref{eq:ttinf-concentration-3}, for each $i \in [n]$ we have
\begin{align}
    \Circled{1} = \pr\qty\bigg(\norm\Big{\sum_{j}\hy_{ij}z_j} \gtrsim {\sigma \kappa\sqrt{np(t + \log{p})} + K\Rx(t + \log{p})} ) \le 2e^{-t}.\label{eq:ttinf-concentration-4}
\end{align}

\noindent\textbf{Step 3: Union bound.} As required by \cref{eq:ttinf-quantitative-3}, note that for all $t > 0$,
\begin{align}
    u := \sigma ( \kappa + \Rx) \sqrt{np (t + \log{p})} + K\Rx(t + \log{p}) \gtrsim \sigma\Rx\sqrt{n};
\end{align}
therefore, plugging in the bounds for $\Circled{1}$ and $\Circled{2}$ into \cref{eq:ttinf-quantitative-4} and \cref{eq:ttinf-quantitative-3}, we have that for each $i \in [n]$ and for all $K, t > 0$,
\begin{align}
    \pr\qty\bigg(  \norm\Big{\sum_{j}Y_{ij}z_j} \gtrsim {\sigma ( \kappa + \Rx) \sqrt{np (t + \log{p})} + K\Rx(t + \log{p})} ) \le 4\qty( 2e^{-t} + n{\sigma^q}/{K^q} ).
\end{align}
Setting $t \mapsto t + \log{n}$ and taking a union bound over all $i \in [n]$, we get
\begin{align}
    &\pr\qty\bigg( \max_{i \in [n]}\norm\Big{\sum_{j}Y_{ij}z_j}  \gtrsim {\sigma ( \kappa + \Rx) \sqrt{np (t + \log{n} + \log{p})} + K\Rx(t + \log{n} + \log{p})} ) \\
    &\le 
    n \cdot \max_{i \in [n]}\pr\qty\bigg(\norm\Big{\sum_{j}Y_{ij}z_j} \gtrsim {\sigma ( \kappa + \Rx) \sqrt{np (t + \log{n} + \log{p})} + K\Rx(t + \log{n} + \log{p})} )\\
    &\le 4n\qty( 2e^{-t -\log{n}} + n{\sigma^q}/{K^q} ) \le 8\qty(e^{-t} + n^2 \sigma^q/K^q).
\end{align}
Similar to \cref{eq:K-choice} in the proof of \cref{prop:latala-quantitative}, setting $K = \sigma n^{\frac{(r+2)}{q}}$ for any $0 < r < (q-4)/2$, it follows that $n^2K^q/\sigma^q = n^{-r}$, and, therefore, for all $t > 0$ with probability greater than $1 - 8(e^{-t} + n^{-r})$ we have
\begin{align}
     \ttinf{\Y\Z} \lesssim {\sigma ( \kappa + \Rx) \sqrt{np (t + \log{np})} + n^{\frac{(r+2)}{q}}\sigma\Rx(t + \log{np})},
\end{align}
which proves the claim in \cref{eq:ttinf-concentration-1}. Moreover, since $(r+2)/q < (q-2)/2q < 1/2$ when $q > 4$, it follows that
$
n^{{(r+2)}/{q}}\log{n} \ll \sqrt{n},\label{eq:K-bound}
$
and there exists $\enn \equiv \enn(r, q,  \kappa, \Rx)$ such that for all $n > \enn$,
\begin{align}
    n^{\frac{(r+2)}{q}}M\Rx(t + \log{np}) \lesssim {\sigma ( \kappa + \Rx) \sqrt{np (t + \log{np})}}.\label{eq:N-bound}
\end{align}
Setting $t = 2\log{n}$ in \cref{eq:ttinf-concentration-1} and noting that \cref{eq:N-bound} holds for $n > \enn$, we have
\begin{align}
    \ttinf{\Y\Z} \lesssim \qty({\sigma ( \kappa + \Rx) \sqrt{np (\log{n} + \log{np})} + n^{\frac{(r+2)}{q}}\sigma\Rx(\log{n} + \log{np})}) \lesssim \sigma( \kappa + \Rx)\sqrt{np\log{np} }
\end{align}
with probability greater than $1-8(n^{-2} + n^{-r})$, which proves the claim in \cref{eq:ttinf-concentration-2}.

\end{proof}
\endgroup

\begin{remark}
    The proof of \cref{prop:latala-quantitative} and \cref{prop:ttinf-concentration} is reminiscent of recent results on obtaining heavy-tailed concentration inequalities in statistical learning via truncation \citep[e.g.,][]{maurer2021concentration,kuchibhotla2022moving,bakhshizadeh2023sharp,klochkov2020uniform}, which build on the seminal work of \cite{nagaev1979large}. Both results are considerably easier to derive when the entries are symmetric random variables, which is facilitated by using the symmetrization lemma from \cref{lem:opnorm-symmetrization} and in Step~1 for the proof of \cref{prop:ttinf-concentration}.
\end{remark}

Lastly, the proof of \cref{thm:X-ttinf} in the $\loss\ttinft$ metric is based on \cref{lem:decomposition}, below, which provides a decomposition of the reconstruction error of $\hX$ into terms which are, relatively speaking, more amenable to analysis.


\begin{lemma}\label{lem:decomposition}
Under the setting of \cref{thm:X-ttinf}, let $\X, \hX$ be the true and estimated configurations associated with $\Del$ and $D$, respectively, and let $U \L U\tr$ and $\hU\hL\hU\tr$ be the rank-$p$ spectral decomposition of $\Delc$ and $\Dc$, respectively.
Furthermore, for $Q \in \orth{p}$ to be the matrix satisfying \cref{eq:cs-exact-recovery}, i.e., $H\X = U\L^{1/2}Q$, and for the Procrustes alignment matrix $Q_* = \argmin_{P \in \orth{p}}\smallnorm{U\tr\hU - P}_F$, let $g_* \in \euc{p}$ be the rigid transformation given by
$g_*(v) = Q_*\tr Q(v - \bx)$, where $\bx = \frac{1}{n}X\tr\onev$.
Then,
\begin{align}\label{eq:ttinf-expansion}
    \hX - g_*(\X) &= 
    \qty\big(\Dc - \Delc) \qty\big(\hU - UQ_*) \hL^{-1/2}         \tag{$=:R_1$}\\
    &\qq{}+U\L \qty\big(U\tr\hU - Q_*)\hL^{-1/2}                  \tag{$=:R_2$}\\
    &\qq{}+\Dc U \qty\big( Q_* \hL^{-1/2} - \L^{-1/2}Q_* )        \tag{$=:R_3$}\\
    &\qq{}+\qty\big(\Dc - \Delc) U \L^{-1/2}Q_*.                  \tag{$=:R_4$}
\end{align}
\end{lemma}

\begin{proof}
By definition of $g_*(v)$ in \cref{lem:decomposition}, $g_*(\X) = (H\X)Q\tr Q_*$ and
\begin{align}
    \hX - g_*(\X) = \hX - (H\X)Q\tr Q_* = \hU \hL^{1/2} - U\Lambda^{1/2}Q_* = \Dc \hU \hL^{-1/2} - \Delc U\Lambda^{-1/2}Q_*,
\end{align}
where the second equality follows from the fact that $(H\X)Q\tr = U\L^{1/2}$ from \cref{eq:cs-exact-recovery}, and the last equality uses the fact that $\Delc = U\L U\tr$ and $\Dc = \hU\hL\hU\tr + \hU_\perp\hL_\perp\hU_\perp\tr$ where $\hU_\perp\tr U = \O$. 

Expanding the expression, we get
\begin{align}
    \hX - g_*(\X)
    &= \Dc \hU \hL^{-1/2} - \Delc U\Lambda^{-1/2}Q_* {\phantom{\underbrace{1}_{1}}}\\
    &= \Dc \hU \hL^{-1/2} - \Dc U\Lambda^{-1/2}Q_* + \underbrace{ \qty(\Dc - \Delc) U\L^{-1/2}Q_* }_{=: R_4} \\
    &= \Dc \hU \hL^{-1/2} - \Dc UQ_*\Lambda^{-1/2} + \underbrace{\Dc U \qty\big( Q_*\hL^{-1/2} - \L^{-1/2}Q_* )}_{=:R_3} + {R_4} \\
    &= \Dc \qty\big(\hU - UQ_*) \hL^{-1/2} + R_3 + R_4 {\phantom{\underbrace{1}_{1}}}\\
    &= \underbrace{(\Dc - \Delc) \qty\big(\hU - UQ_*) \hL^{-1/2}}_{=:R_1} + \Delc \qty\big(\hU - UQ_*) \hL^{-1/2} + R_3 + R_4 \\
    &= R_1 + U\L U\tr\qty\big(\hU - UQ_*) \hL^{-1/2} + R_3 + R_4 {\phantom{\underbrace{1}_{1}}}\\
    &= R_1 + R_2 + R_3 + R_4,
\end{align}
where the final equality follows from $U\tr U = I$ and by the definition of $A_2$ in \cref{lem:decomposition}.
\end{proof}

From the expression in \cref{lem:decomposition} and in light of \cref{lem:D-opnorm}, it becomes clear that the controlling the first three terms are straightforward and use, for the most part, standard matrix concentration results in \cref{lem:eigen} and \cref{lem:gram} (similar to \citet{li2020central} and \citet{little2023analysis}). The near parametric rate of convergence in \cref{thm:X-ttinf} is achieved by controlling final term, which is achieved via symmetrization and using \cref{prop:ttinf-concentration}.

\subsection{Results for Lower Bounds}
\label{sec:lower-bound-results}

This subsection collects the deferred proofs which are used in the proofs for lower bounding the minimax risk in \cref{sec:proofs}. We begin with the following which establishes that any norm $\norm{\cdot}_{\dagger}$ induces a pseudometric $\min_{g \in \euc{p}}\norm{\Y-g(\X)}_{\dagger}$ on $\Rnp$.


\begin{lemma}\label{lem:rho-rigid}
    Let $\norm{\cdot}_\dagger$ be any norm on $\R^{n \times p}$ and let $\rho: \Rnp \times \Rnp \to \R_{\ge 0}$ be given by
    $$
    \rho(\X, \Y) = \min_{g \in \euc{p}}\norm{\X - g(\Y)}_{\dagger}.
    $$
    Then, $\rho$ is a pseudometric on $\Rnp$.
\end{lemma}

\begin{proof}
    By definition of a pseudometric, it suffices to show that: (i) $\rho(\X, \X) = 0$, (ii) $\rho(\X, \Y) = \rho(\Y, \X)$, and (iii) $\rho(\X, \Y) \le \rho(\X, \Z) + \rho(\Z, \Y)$ for any $\X, \Y, \Z \in \Rnp$. The first claim follows by noting that the identity map $\text{id}: \X \mapsto \X$ is an element of $\euc{p}$, and therefore $\min_{g}\norm{\X - g(\X)}_\dagger = \norm{\X - \X}_\dagger = 0$. Similarly, since $\euc{p}$ is a subgroup of isometric transformations on $\R^{n \times p}$, we have
    \begin{align}
        \rho(\X, \Y) = \min_{g \in \euc{p}}\norm{\X - g(\Y)}_{\dagger} = \min_{g \in \euc{p}}\norm{g(g\inv(\X)) - g(\Y)}_{\dagger} = \min_{g \in \euc{p}}\norm{g\inv(\X) - \Y}_{\dagger} = \rho(\Y, \X).
    \end{align}
    The triangle inequality follows from the triangle inequality for the norm $\norm{\cdot}_\dagger$, i.e., for any $\Z \in \Rnp$ and for all $h \in \euc{p}$,
    \begin{align}
        \norm{\X - g(\Y)}_{\dagger} \le \norm{\X - h(\Z)}_{\dagger} + \norm{h(\Z) - g(\Y)}_{\dagger} = \norm{\X - h(\Z)}_{\dagger} + \norm{\Z - (h\inv \circ g)(\Y)}_{\dagger}.
    \end{align}
    Since the above inequality holds for all $\X, \Y, \Z \in \Rnp$ and for all $g, h \in \euc{p}$, taking the minimum with respect to $h$ for the first term on the right hand side followed by the minimum with respect to $g$ on the left hand side gives
    \begin{align}
        \rho(\X, \Y) = \min_{g \in \euc{p}}\norm{\X - g(\Y)}_{\dagger} \le \min_{h \in \euc{p}}\norm{\X - h(\Z)}_{\dagger} + \norm{\Z - (h\inv \circ g)(\Y)}_{\dagger} = \rho(\X, \Z) + \norm{\Z - g'(\Y)}_{\dagger},
    \end{align}
    where $g' = h\inv \circ g \in \euc{p}$. Since this holds for all $g, h \in \euc{p}$ as well, taking the minimum over $g' \in \euc{p}$ on the right hand side gives the result.
\end{proof}

The following subsections contain the proofs for auxiliary lemmas used in the  proofs for the lower bounds in \cref{sec:proofs}.


\subsubsection{Proof of \cref{lemma:Y-subset-X}}
\label{proof:lemma:Y-subset-X}

By construction in \cref{eq:bbY}, since $\X$ is centered it follows that each $\Y \in \bbY$ is centered and, therefore, $H\Y= \Y$. This follows by noting that $\omega \in \qty{-1, 0, 1}^n$ is such that $\onev_n\tr\omega = 0$ and
\begin{align}
    \onev_n\tr\Y = \onev_n(\X + \eta\omega v\tr) = \onev_n\X + \eta(\onev_n\omega) v\tr = \zerov_p.\label{eq:Y-center}
\end{align}

\noindent\textit{Proof of \ref{eq:sY-bound}}.\quad From Weyl's inequality for singular values \citep[{Corollary 7.3.5}]{horn2012matrix}, for all $i \in [p]$ we have
\begin{align}
    \abs{ s_i(H\Y) - s_i(H\X) } \le \opnorm{H\Y - H\X} = \opnorm{\Y-\X} = \eta\smallnorm{\omega v\tr}_2 = \eta\norm{\omega}\cdot\norm{v} \le \eta\sqrt{n},
\end{align}
where the final inequality follows from the fact that $\norm{\omega} \le \sqrt{n}$ for all $\omega \in \qty{-1, 0, 1}^n$ and ${\norm{v}=1}$. Therefore, 
\begin{align}
    s_1\qty(\frac{H\Y}{\sqrt{n}}) \le s_1\qty(\frac{H\X}{\sqrt{n}}) + \eta \le \gamma + \eta \qq{and} s_p\qty(\frac{H\Y}{\sqrt{n}}) \ge s_p\qty(\frac{H\X}{\sqrt{n}}) - \eta \ge \gamma - \eta.\FINEQ
\end{align}

\noindent\textit{Proof of \ref{eq:DelY-bound}}.\quad
For $\Y = \X + \eta \omega v\tr$, let $u := \X v$. By noting that $v\tr v = 1$, we have
\begin{align}
    \Y\Y\tr 
    &= \X\X\tr + \eta\qty(\omega v\tr\X\tr + \X v \omega\tr) + \eta^2\omega v\tr v\omega\tr\\
    &= \X\X\tr + \eta\qty(\omega u\tr + u\omega\tr) + \eta^2\omega\omega\tr,\label{eq:YYt}
\end{align}
and
\begin{align}
    \diag(\Y\Y\tr) 
    &= \diag(\X\X\tr) + \eta \qty( \omega \circ u + u \circ \omega ) + \eta^2 (\omega \circ \omega)\\ 
    &= \diag(\X\X\tr) + 2\eta (\omega \circ u) + \eta^2 (\omega \circ \omega).\label{eq:diagYYt}
\end{align}
From \cref{eq:Del-matrix-form}, it follows that
\begin{align}
    \Del(\Y) - \Del(\X) = \diag(\Y\Y\tr - \X\X\tr)\onev\tr + \onev\diag(\Y\Y\tr - \X\X\tr)\tr - 2(\Y\Y\tr - \X\X\tr),
\end{align}
and
\begin{align}
    \frobenius{\Del(\Y) - \Del(\X)}
    &= \smallnorm{\diag(\Y\Y\tr - \X\X\tr)\onev\tr + \onev\diag(\Y\Y\tr - \X\X\tr)\tr - 2(\Y\Y\tr - \X\X\tr)}_F\\
    &\le 2\smallnorm{\diag(\Y\Y\tr - \X\X\tr)\onev\tr}_F + 2\smallnorm{\Y\Y\tr - \X\X\tr}_F\\
    &\lesssim \sqrt{n}\smallnorm{\diag(\Y\Y\tr - \X\X\tr)} + \smallnorm{\Y\Y\tr - \X\X\tr}_F,
\end{align}
where the final inequality follows by noting that ${\smallnorm{\diag(\Y\Y\tr - \X\X\tr)\onev\tr}_F = \smallnorm{\diag(\Y\Y\tr - \X\X\tr)} \smallnorm{\onev}}$. From \cref{eq:YYt}, we have
\begin{align}
    \smallnorm{\Y\Y\tr - \X\X\tr}_F 
    &= \smallnorm{\eta\qty\big(\omega u\tr + u\omega\tr) + \eta^2\omega\omega\tr}_F\\
    &\le 2\eta\smallnorm{\omega}\smallnorm{u} + \eta^2\smallnorm{\omega}^2\\
    &\le 2\eta\sqrt{n}\smallnorm{u} + \eta^2n,\label{eq:YYt-bound}
\end{align}
where the first inequality follows from the triangle inequality and the fact that $\smallnorm{\omega u\tr}_F = \norm{\omega}\norm{u}$. Similarly, from \cref{eq:diagYYt}, we have
\begin{align}
    \smallnorm{\diag(\Y\Y\tr - \X\X\tr)} 
    &= \smallnorm{2\eta (\omega \circ u) + \eta^2 (\omega \circ \omega)}\\
    &\le 2\eta\smallnorm{\omega \circ u} + \eta^2\smallnorm{\omega \circ \omega}\\
    &\le 2\eta\smallnorm{u} + \eta^2\sqrt{n},\label{eq:diagYYt-bound}
\end{align}
where the final inequality follows from the fact that for $\omega \in \qty{-1, 0, 1}^n$ we have $\smallnorm{\omega \circ u} \le \smallnorm{\onev \circ u} = \norm{u}$ and $\smallnorm{\omega \circ \omega} \le \smallnorm{\onev} = \sqrt{n}$. Since $\opnorm{H\X} = \gamma\sqrt{n}$, we further have
\begin{align}
    \norm{u} = \smallnorm{\X v} \le \opnorm{\X} = \gamma\sqrt{n}.
\end{align}
Plugging in the bounds from \cref{eq:YYt-bound,eq:diagYYt-bound} into the expression for $\frobenius{\Del(\Y) - \Del(\X)}$ gives
\begin{align}
    \frobenius{\Del(\Y) - \Del(\X)} \le 2\sqrt{n}\qty( \eta \sqrt{n}\norm{u} + \eta^2 \sqrt{n} ) + 2\qty(\eta\sqrt{n}\norm{u} + \eta^2n)\le 4n( \eta\gamma + \eta^2 ),
\end{align}
and by an application of \cref{lem:power-identities} we have
\begin{align}
    \frobenius{\Del(\Y) - \Del(\X)}^2 \le 32 \cdot n^2( \gamma^2\eta^2 + \eta^4 ).\FINEQ
\end{align}

\noindent\textit{Proof of \ref{eq:Y-separation}}.\quad
For $\tau, \tau' \in \qty{0, 1}^m$, in the interest of clarity, let $\omega, \omega' \in \qty{-1, 0, 1}^n$ denote the resulting binary vectors and let $\Y, \Y' \in \bbY$  be the resulting configuration matrices, respectively. We will establish the lower bound by showing that
\begin{align}
    \loss\rmse\qty(\Y, \Y') 
    \gtrsim \frac{\smallnorm{\Y'\Y'{}\tr - \Y\Y\tr}_F}{{n}(\gamma + \eta)}
    \gtrsim \frac{\gamma\eta\sqrt{n} \cdot \sqrt{d_H(\tau, \tau')}}{{n}(\gamma + \eta)}.\label{eq:claim-YYt}
\end{align}
Note that the right hand side of \cref{eq:claim-YYt} simplifies to the desired lower bound in \ref{eq:Y-separation}.

\textbf{Proof of the first inequality in \cref{eq:claim-YYt}.} By an application of the triangle inequality, for all $Q \in \orth{p}$ it follows that
\begin{align}
    \smallnorm{\Y'\Y'{}\tr - \Y\Y\tr}_F 
    &= \smallnorm{(\Y' - \Y Q)\Y'{}\tr + (\Y Q)(\Y' - \Y Q)\tr}_F\\
    &\le \smallnorm{(\Y' - \Y Q)\Y'{}\tr}_F + \smallnorm{(\Y Q)(\Y' - \Y Q)\tr}_F\\
    &\le \smallnorm{\Y' - \Y Q}_F \cdot \opnorm{\Y'} + \smallnorm{\Y' - \Y Q}_F\cdot \opnorm{\Y Q}\\
    &= \smallnorm{\Y' - \Y Q}_F \cdot \qty(\opnorm{\Y} + \opnorm{\Y'})\\
    &= \smallnorm{\Y' - \Y Q}_F \cdot 2 \sqrt{n}(\gamma + \eta),
\end{align}
where the final inequality follows from \ref{eq:sY-bound} and by noting that $H\Y = \Y$ and $H\Y' = \Y'$ from \cref{eq:Y-center}. Minimizing over all $Q \in \orth{p}$ gives
\begin{align}
    \frac{\smallnorm{\Y'\Y'{}\tr - \Y\Y\tr}_F}{2n(\gamma + \eta)}
    \le \min_{Q \in \orth{p}}\frac{\smallnorm{(H\Y') - (H\Y) Q}_F}{\sqrt{n}}
    =\min_{g \in \euc{p}}\frac{\smallnorm{\Y' - g(\Y)}_F}{\sqrt{n}} 
    = \loss\rmse(\Y', \Y),
\end{align}
where the first equality follows from the translation invariance of the orthogonal Procrustes problem. This yields the first expression in \cref{eq:claim-YYt}.

\textbf{Proof of the second inequality in \cref{eq:claim-YYt}.} From \cref{eq:YYt} and for $u := \X v$, we have
\begin{align}
    \Y'\Y'{}\tr - \Y\Y\tr 
    &= \eta(\omega' u\tr + u\omega'{}\tr ) + \eta^2\omega'\omega'{}\tr - \eta(\omega u\tr + u\omega\tr) - \eta^2\omega\omega\tr\\
    &= \eta\qty( (\omega'-\omega) u\tr + u(\omega'-\omega)\tr ) + \eta^2\qty( \omega'(\omega'{}-\omega)\tr + (\omega'{}-\omega)\omega\tr )\\
    &= \eta\qty( \zeta u\tr + u \zeta\tr ) + \eta^2\qty( \omega'\zeta\tr + \zeta\omega\tr ) =: \eta A + \eta^2 B,
\end{align}
where $\zeta := \omega'-\omega$, $A = \zeta u\tr + u \zeta\tr$ and $B :=  \omega'\zeta\tr + \zeta\omega\tr$. It follows that
\begin{align}
    \smallnorm{\Y'\Y'{}\tr - \Y\Y\tr}_F^2 = \smallnorm{\eta A + \eta^2 B}_F^2
    &= \eta^2 \smallnorm{A}^2_F + 2\eta^3 \trace(A\tr B) + \eta^4\smallnorm{B}^2_F \\
    &\ge \eta^2 \smallnorm{A}^2_F + 2\eta^3 \trace(A\tr B).\label{eq:YYt-AB}
\end{align}
For the first term in \cref{eq:YYt-AB}, we have
\begin{align}
    \smallnorm{A}^2_F = \smallnorm{ \zeta u\tr + u \zeta\tr }^2_F = 2 \smallnorm{ \zeta u\tr}^2_F + 2 \trace(u\zeta\tr u\zeta\tr) = 2 \smallnorm{ \zeta}^2\smallnorm{u}^2 + 2(\zeta\tr u)^2 \ge 2\smallnorm{ \zeta}^2\smallnorm{u}^2.
\end{align}
Furthermore, for $\omega, \omega' \in \qty{-1, 0, 1}^n$,
\begin{align}
    \smallnorm{\zeta}^2 = \smallnorm{\omega' - \omega}^2 = 2 d_H(\omega, \omega') = 8 d_H(\tau, \tau'),\label{eq:dh-zeta}
\end{align}
where the final inequality follows by noting that for $\omega, \omega'$ given by \cref{eq:omega}, we have $d_H(\omega, \omega') = 2 d_H(\tau, \tau')$. Therefore,
\begin{align}
    \smallnorm{A}^2_F \ge 8\smallnorm{u}^2 d_H(\tau, \tau').
\end{align}
For the second term in \cref{eq:YYt-AB}, by simplifying the trace of the resulting rank-$1$ matrices and using the Cauchy-Schwarz inequality, we have
\begin{align}
    \inner{A, B}_F = \trace(A\tr B) 
    &= \trace\qty( (\zeta u\tr + u \zeta\tr)\tr (\omega'\zeta\tr + \zeta\omega\tr) )\\ 
    &= (\zeta\tr u)\qty(\zeta\tr \omega + \zeta\tr \omega') + \norm{\zeta}^2 \qty( u\tr\omega + u\tr\omega' )\\
    &\ge -\norm{\zeta}^2 \cdot \norm{u} \cdot \smallnorm{\omega + \omega'}\\
    &\ge -16 d_H(\tau, \tau') \sqrt{n} \cdot \norm{u},
\end{align}
where the final inequality follows from \cref{eq:dh-zeta} and by noting that $\smallnorm{\omega + \omega'} \le 2\smallnorm{\onev} = 2\sqrt{n}$. Plugging these back into \cref{eq:YYt-AB} and by noting that $\norm{u} = \norm{\X v} = \gamma\sqrt{n}$, we get
\begin{align}
    \smallnorm{\Y'\Y'{}\tr - \Y\Y\tr}_F^2 \ge 8n\eta^2\gamma^2 \cdot d_H(\tau, \tau') - 32n\eta^3\gamma d_H(\tau, \tau') = 8n\eta^2\gamma^2 d_H(\tau, \tau') \cdot \qty( 1 - \frac{4\eta}{\gamma} ).
\end{align}
Since $\frac{1}{\kappa} \le \gamma$, when $\eta \le 1/8\kappa \le \gamma/8$, it follows that $1 - 4{\eta}/{\gamma} \ge \half$ and 
\begin{align}
    \smallnorm{\Y'\Y'{}\tr - \Y\Y\tr}_F^2 \ge 4n\eta^2\gamma^2 d_H(\tau, \tau').
\end{align}
Taking the square root of both sides gives the second inequality in \cref{eq:claim-YYt}.\QED


\subsubsection{Proof of \cref{lemma:X1-properties}}
\label{proof:lemma:X1-properties}

From \cref{eq:X-properties-2}, recall that that $\X$ is centered since $\onev\tr\X = \gamma \sqrt{n} \onev \tr [\half V\tr -\half V\tr]\tr = \zerov$.

\noindent\textit{Proof of \ref{eq:X1-singular}}.\quad
For $\X^k = \X + \eta  e_k  v(k)\tr$, since $H\X = \X$ and $\norm{v}=1$, using the triangle inequality we have
\begin{align}
    \smallnorm{H\X^k}\ttinft = \smallnorm{H\X + \eta  H e_k  v(k)\tr}\ttinft 
    &\le \ttinf{\X} + \eta\smallnorm{ He_k  v(k)\tr}\ttinft\\[5pt]
    &\le \frac{\gamma\Rx}{2} +  \eta \cdot \norm{He_k}_\infty\smallnorm{v(k)}\\ 
    &= \frac{\gamma\Rx}{2} + \eta,
\end{align}
where we used the fact that $\norm{He_k}_\infty = 1 - \frac{1}{n} \le 1$. For the bound on the singular values, for all $i \in [p]$ and from Weyl's inequality \cite[{Corollary 7.3.5}]{horn2012matrix} we have
\begin{align}
    \abs\big{ s_i(H\X^k) - s_i(H\X) } \le \smallnorm{H\X^k - H\X}_2 = \eta \smallnorm{ He_k v(k)\tr}_2 \le \eta \opnorm{H} \cdot \norm{e_k}\cdot\norm{v} \le \eta,
\end{align}
where the final inequality follows from the fact that $\opnorm{H} = 1$. Therefore, 
\begin{align}
    s_1\qty(\frac{H\X^k}{\sqrt{n}}) \le s_1\qty(\frac{H\X}{\sqrt{n}}) + \frac{\eta}{\sqrt{n}} \le \gamma + \frac{\eta}{\sqrt{n}} \qq{and} s_p\qty(\frac{H\X^k}{\sqrt{n}}) \ge s_p\qty(\frac{H\X}{\sqrt{n}}) - \frac{\eta}{\sqrt{n}} \ge \gamma - \frac{\eta}{\sqrt{n}}.\FINEQ
\end{align}

\noindent\textit{Proof of \ref{eq:X1-separation}}.\quad
Let $g(x) = Qx + w$ be an arbitrary rigid transformation. Then, using the fact that the quadratic mean is less than the maximum, we get
\begin{align}
    \smallnorm{\X^k - g(\X)}\ttinft 
    &= \max_{i \in [n]} \smallnorm{X^k_i - QX_i - w}\label{eq:max-qm}\\ 
    &\ge \max_{i = k, k+m} \smallnorm{X^k_i - QX_i - w} \ge \sqrt{\frac{\smallnorm{X^k_k - QX_k - w}^2 + \smallnorm{X^k_{k+m} - QX_{k+m} - w}^2}{2}}.
\end{align}
By construction of $\bbX_1$ in \cref{eq:bbX-1}, we have $X^k_k = X_k + \eta v(k)$ and $X^k_{k+m} = -X_k$. Using the parallelogram law, we get
\begin{align}
    \smallnorm{X_k + \eta v(k) - QX_k - w}^2 + \smallnorm{-X_k - Q(-X_{k}) - w}^2
    &=\smallnorm{X_k + \eta v(k) - QX_k - w}^2 + \smallnorm{X^k - QX_{k} + w}^2 \\[5pt]
    &= \half\qty({\norm{2X_k + \eta v(k) - 2QX_k}^2 + \norm{2w - \eta v}^2})\\[5pt]
    &\ge 2{\norm\Big{X_k + \half\eta v(k) - QX_k}^2}\\
    &= 2{\norm\Big{X_k \qty(1 + \frac{\eta}{\norm{X_k}}) - QX_k}^2}\label{eq:Xk-Q}\\
    &\ge 2\norm{X_k}^2 \cdot \frac{4\eta^2}{\norm{X_k}^2} = \frac{\eta^2}{2},
\end{align}
where the second equality follows from the fact that $v(k) = X_k/\norm{X_k}$ and the final inequality follows from the fact that $Q=I_p$ is the optimal alignment for \cref{eq:Xk-Q}. Since the inequality above holds for all rigid transformations, plugging this back into \cref{eq:max-qm} and minimizing over $\euc{p}$ gives
\begin{align}
    \loss\ttinft(\X^k, \X) = \min_{g \in \euc{p}}\smallnorm{X^k - g(X)}\ttinft \ge \frac{\eta}{2}.\FINEQ
\end{align}

\noindent\textit{Proof of \ref{eq:X1-frobenius}}.\quad
Since $X^k = \X + \eta e_k v(k)\tr$ modifies only the $k$th row of $\X$, for all $j \neq k$, we have
\begin{align}
    \Del(\X^k)_{kj} - \Del(\X)_{kj} 
    &= \norm{X_k + \eta v(k) - X_j}^2 - \norm{X_k - X_j}^2 
    = \eta^2 + 2\eta v(k)\tr(X_k - X_j),\label{eq:X1-frobenius-0}
\end{align}
and $\Del(\X^k)_{ij} = \Del(\X^k)_{ij}$ otherwise. Therefore,
\begin{align}
    \smallnorm{\Del(\X^k) - \Del(\X)}^2_F 
    &= 2\sum_{j \in [n]} \qty(\Del(\X^k)_{kj} - \Del(\X)_{kj})^2\\
    &= 2\sum_{j \in [n]} \qty(\eta^2 + 2\eta v(k)\tr(X_k - X_j))^2\\
    &\le 16\sum_{j \in [n]} \eta^4 + \eta^2 \qty(v(k)\tr X_k)^2 + \eta^2 \qty(v(k)\tr X_j)^2\label{eq:X1-frobenius-1-1}\\
    &= 16\sum_{j \in [n]} \eta^4 + \eta^2 \norm{X_k}^2 + \eta^2 v(k)\tr X_jX_j\tr v(k)\\
    &\le 16n \qty(\eta^4 + \eta^2\max_{i}\norm{X_i}^2 + \eta^2v(k)\tr \qty(\tfrac{\X\tr\X}{n}) v(k))\\
    &\le 16n \qty(\eta^4 + \eta^2\ttinf{H\X}^2 + \eta^2 \opnorm\Big{\tfrac{(H\X)\tr(H\X)}{n}} )\label{eq:X1-frobenius-1-2}\\[5pt]
    &\le 16n \qty(\eta^4 + \eta^2\Rx^2 + \eta^2\gamma^2)
    \le 16n \qty(\eta^4 + \eta^2(\gamma + \Rx)^2)\label{eq:X1-frobenius-1-3},
\end{align}
where first inequality uses \cref{lem:power-identities}, the penultimate inequality follows from the fact that $H\X=\X$, and the final inequality follows from the fact that $\ttinf{H\X}\le\Rx/2$ and $\opnorm{(H\X)\tr(H\X)} = \gamma^2n$.\FIN

\noindent\textit{Proof of \ref{eq:X1-frobenius-inner}}.\quad
Similar to \cref{eq:X1-frobenius-0}, for all $\ell \neq k \in [m]$ and for all $j \neq \ell$, we have
\begin{align}
    \qty(\Del(X^\ell) - \Del(X))_{\ell j} = \eta^2 + \eta v(\ell)\tr(X_\ell-X_j),
\end{align}
and $\qty(\Del(X^\ell) - \Del(X))_{ij} = 0$ otherwise. Therefore,
\begin{align}
    &\qty(\Del(\X^{k}) - \Del(\X))_{ij} \cdot \qty(\Del(\X^{\ell}) - \Del(\X))_{ij}\\ 
    &\qquad\qquad= 
    \begin{cases}
        \qty(\eta^2 + 2\eta v(k)\tr(X_k - X_\ell)) \cdot \qty(\eta^2 + 2\eta v(\ell)\tr(X_\ell - X_k)) & \text{if } i=k \text{ and } j=\ell\\
        \qty(\eta^2 + 2\eta v(\ell)\tr(X_\ell - X_k)) \cdot \qty(\eta^2 + 2\eta v(k)\tr(X_k - X_\ell)) & \text{if } i=\ell \text{ and } j=k\\
        0 & \text{otherwise}.
    \end{cases}
\end{align}
There are a total of $2$ non-zero terms in the above display, and since $\norm{v(k)} = \norm{v(\ell)}=1$, using the Cauchy-Schwarz inequality along with the fact that $\max_{k, \ell}\norm{X_k - X_\ell}\le 2\ttinf{H\X}$, we have
$$
\eta^2 + 2\eta v(k)\tr(X_k - X_\ell) \le \eta^2 + 2\eta \norm{X_k - X_\ell} \le \eta^2 + 4\eta \Rx,
$$
and, it follows that
\begin{align}
    \inner{ \Del(\Y^{(k)}) - \Del(\X),  \Del(\Y^{(j)}) - \Del(\X)}_F 
    &= 2\sum_{i < j} \qty(\Del(\Y^{(k)}) - \Del(\X))_{ij} \cdot \qty(\Del(\Y^{(\ell)}) - \Del(\X))_{ij}\\
    &= 2\qty(\eta^2 + 4\eta\Rx)^2 \le 4\eta^4 + 32\eta^2\Rx^2,
\end{align}
where the final inequality follows from \cref{lem:power-identities}.\qed


\subsubsection{Proof of \cref{lemma:tv-bound}}
\label{proof:lemma:tv-bound}

Let $\X^0 = \X \in \bbX_0$ and $\X^k \in \bbX_1$ from \cref{eq:bbX-1} for $k \in [n]$. For $\eps_{ij} \sim \mathcal{N}(0, \sigma^2)$, let $\phi_k(d_{ij})$ be the probability density function of $d_{ij} = \del(\X^k)_{ij} + \eps_{ij}$ for all $i < j$, i.e.,
\begin{align}
    \phi_k(d_{ij}) = \frac{1}{\sqrt{2\pi\sigma^2}}\exp\qty(-\frac{(d_{ij} - \del(\X^k)_{ij})^2}{2\sigma^2}),
\end{align}
and let $\bphi$ be probability density function of the Gaussian mixture given by
\begin{align}
    \bphi(d_{ij}) = \frac{1}{n}\sum_{k=1}^n \phi_k(d_{ij}).
\end{align}
By \citet[Lemma~2.7]{tsybakov2008nonparametric}, we have
\begin{align}
    \tv(\prs_0, \bprs)^2 \le \chi^2(\prs_0, \bprs) = \int \qty(\frac{d\bprs}{d\prs_0})^2 d\prs_0 - 1 = \underbrace{\qty(\int \prod_{i < j}\frac{\bphi(d_{ij})^2}{\phi_0(d_{ij})})}_{\Circled{1}} - 1,\label{eq:chi2-bound}
\end{align}
where $\chi^2(\cdot, \cdot)$ is the $\chi^2$-divergence. We bound the term $\Circled{1}$ in \cref{eq:chi2-bound} as follows. First, note that
\begin{align}
    \bphi(d_{ij})^2 = \qty(\frac{1}{n}\sum_{k=1}^n \phi_k(d_{ij}))^2 = \frac{1}{n^2}\sum_{k,\ell}\phi_k(d_{ij})\phi_\ell(d_{ij})
\end{align}
and hence,
\begin{align}
    \Circled{1} = \int \prod_{i < j}\frac{\bphi(d_{ij})^2}{\phi_0(d_{ij})} = \frac{1}{n^2}\sum_{k, \ell}\int \prod_{i < j}\frac{\phi_k(d_{ij})\phi_\ell(d_{ij})}{\phi_0(d_{ij})} = \frac{1}{n^2}\sum_{k, \ell} \prod_{i < j} \int\frac{\phi_k(d_{ij})\phi_\ell(d_{ij})}{\phi_0(d_{ij})}.\label{eq:circled-1-bound}
\end{align}
where the final equality follows from an application Fubini's theorem. By noting that $\X^0=\X$, we further have
\begin{align}
    \frac{\phi_k(d_{ij})\phi_\ell(d_{ij})}{\phi_0(d_{ij})} = \frac{1}{\sqrt{2\pi\sigma^2}}\exp\qty(-\frac{1}{2\sigma^2} \qty{ \qty\big(d_{ij} - \del(\X^k)_{ij})^2 + \qty\big(d_{ij} - \del(\X^\ell)_{ij})^2 - \qty\big(d_{ij} - \del(\X)_{ij})^2  } ).\label{eq:phi-ratio}
\end{align}
Using the identity $(z-a)^2 + (z-b)^2 - (z-c)^2 = (z - a - b + c)^2 - 2(a-c)(b-c)$, we get
\begin{align}
    & \qty\big(d_{ij} - \del(\X^k)_{ij})^2 + \qty\big(d_{ij} - \del(\X^\ell)_{ij})^2 - \qty\big(d_{ij} - \del(\X)_{ij})^2  \\ 
    &\hspace{3em} = \qty(d_{ij} - \del(\X^k)_{ij} - \del(\X^\ell)_{ij} + \del(\X)_{ij})^2 - 2\qty\big(\del(\X^k)_{ij} - \del(\X)_{ij})\cdot\qty\big(\del(\X^\ell)_{ij} - \del(\X)_{ij}).
\end{align}
Let $\mu := \del(\X^k)_{ij} + \del(\X^\ell)_{ij} - \del(\X)_{ij}$. Plugging this back into \cref{eq:phi-ratio} gives us
\begin{align}
    \frac{\phi_k(d_{ij})\phi_\ell(d_{ij})}{\phi_0(d_{ij})} = \exp\qty(\frac{\qty\big(\del(\X^k)_{ij} - \del(\X)_{ij})\cdot\qty\big(\del(\X^\ell)_{ij} - \del(\X)_{ij})}{\sigma^2}) \cdot \frac{1}{\sqrt{2\pi\sigma^2}}\exp(-\frac{(d_{ij} - \mu)^2}{2\sigma^2}).
\end{align}
Note that the second term on the right hand side is the density function of $\mathcal{N}(\mu, \sigma^2)$; therefore, integrating over $d_{ij}$, we get
\begin{align}
    \int \frac{\phi_k(d_{ij})\phi_\ell(d_{ij})}{\phi_0(d_{ij})} = \exp\qty(\frac{\qty\big(\del(\X^k)_{ij} - \del(\X)_{ij})\cdot\qty\big(\del(\X^\ell)_{ij} - \del(\X)_{ij})}{\sigma^2}),
\end{align}
and it follows that
\begin{align}
    \prod_{i < j}\int \frac{\phi_k(d_{ij})\phi_\ell(d_{ij})}{\phi_0(d_{ij})} 
    &= \exp\qty(\frac{1}{\sigma^2}\sum_{i < j}{\qty\big(\del(\X^k)_{ij} - \del(\X)_{ij})\cdot\qty\big(\del(\X^\ell)_{ij} - \del(\X)_{ij})})\\ 
    &= \exp\qty(\frac{1}{\sigma^2}{\inner{\Del(\X^k) - \Del(\X),\; \Del(\X^\ell) - \Del(\X)}[F]}) =: \Gamma(\X, \X^k, \X^\ell).
\end{align}
Plugging this back into \cref{eq:circled-1-bound}, we have
\begin{align}
    \Circled{1} 
    &= \frac{1}{n^2}\sum_{k, \ell}\Gamma(\X, \X^k, \X^\ell)\\
    &= \frac{1}{n^2}\sum_{k}\Gamma(\X, \X^k, \X^k) + \frac{1}{n^2}\sum_{k \neq \ell}\Gamma(\X, \X^k, \X^\ell)\\
    &\le \frac{n}{n^2}\max_{k}\Gamma(\X, \X^k, \X^k) + \frac{n(n-1)}{n^2}\max_{k \neq \ell}\Gamma(\X, \X^k, \X^\ell)\\
    &\le \frac{1}{n}\exp(\frac{1}{\sigma^2}\max_{k}\smallnorm{\Del(\X^k) - \Del(\X)}^2_F) + \exp(\frac{1}{\sigma^2}\max_{k \neq \ell} {\inner{\Del(\X^k) - \Del(\X),\; \Del(\X^\ell) - \Del(\X)}[F]})\\
    &= \exp(\frac{1}{\sigma^2}\max_{k}\smallnorm{\Del(\X^k) - \Del(\X)}^2_F - \log{n}) + \exp(\frac{1}{\sigma^2}\max_{k \neq \ell} {\inner{\Del(\X^k) - \Del(\X),\; \Del(\X^\ell) - \Del(\X)}[F]}).
\end{align}
Substituting the bound on $\Circled{1}$ back into \cref{eq:chi2-bound} gives us the final bound on $\tv(\prs_0, \bprs)^2$. \qed



\section{Toolkit}
\label{sec:toolkit}

The following section contains a collection of known results and auxiliary lemmas which are used in the proofs in \cref{sec:proofs}.

\begin{lemma}\label{lem:power-identities}
For any $x_1, x_2, \cdots, x_m \in \R$ and for all $p \ge 1$, $(x_1 + \cdots + x_m)^p \le m^{p-1} (x^p_1 + \cdots + x^p_m).$
\end{lemma}

\begin{proof}
Note that the function $f(t) = t^p$ is convex for all $p \ge 1$. By Jensen's inequality, we have
\begin{align}
    &f\qty(\frac{x_1 + x_2 + \cdots + x_m}{m}) \le \frac{1}{m}\qty\Big(f(x_1) + f(x_2) + \cdots + f(x_m))\\ 
    \implies
    &\qty(\frac{x_1 + x_2 + \cdots + x_m}{m})^p \le \frac{1}{m}\qty\Big(x_1^p + x_2^p + \cdots + x_m^p),
\end{align}
and the result follows by multiplying both sides by $m^p$.
\end{proof}

The proof of \cref{thm:X-ttinf} repeatedly uses the properties of the $\ell\ttinft$ norm from \cite[Section~6.1]{cape2019two}, which we collect here for reference. 

\begin{lemma}[Propositions~6.3~and~6.5~from~\citealp{cape2019two}]\label{prop:cape-ttinf}
    Let $A \in \R^{p_1 \times p_2}$. Then
    $$
    \ttinf{A} \le \opnorm{A},
    $$
    and for a matrix $U \in \R^{r \times p_2}$ with orthonormal columns, 
    $$
    \smallnorm{AU\tr}\ttinft = \ttinf{A}.
    $$
    Lastly, for two other compatible matrices $B \in \R^{p_2 \times p_3}$ and $C \in \R^{p_3 \times p_4}$,
    \begin{align}
        \ttinf{ABC} \le \norm{A}_\infty \ttinf{B} \norm{C}_2.\label{eq:ttinf-ABC}
    \end{align}
\end{lemma}

The following two lemmas collect some properties the eigenvalues and eigenvectors associated with the pairwise dissimilarity matrices $\Del$ and its noisy counterpart $\D = \Del + \Eps$, and are used in the proofs of the main theorems in \cref{sec:proofs}. Several of them are well-known results, and are provided here for completeness. The first lemma characterizes the eigenvalues and eigenvectors of the Gram matrix $\Delc=(H\X)(H\X)\tr$.

\begin{lemma}\label{lem:gram}
    Suppose $\X \in \bbX(\xpar)$ satisfying \ref{assumption:compact}. For the rank-$p$ spectral decomposition of the Gram matrix, $\Delc = (H\X)(H\X)\tr = U\L U\tr$, the following hold:
    \begin{enumerate}[label=\textup{(\roman*)}]
        \item\label{lem:gram-2} $\opnorm\big{\L} \le  \kappa^2n$ and $\opnorm\big{\L\inv} \le {\kappa^2}/{n}$. Moreover, $n/\kappa^2 \le \lambda_p \le \cdots \le \lambda_1 \le  \kappa^2n$.
        \item\label{lem:gram-3} $\ttinf{U} \le \kappa\Rx / \sqrt{n}$ and $\opnorm\big{U} = 1$.
    \end{enumerate}
\end{lemma}

\begin{proof} 
    By assumption, $H\X$ is a rank-$p$ matrix, and $\L \equiv \L((H\X)(H\X)\tr) = \L((H\X)\tr(H\X))$ is also rank-$p$, from assumption~\ref{assumption:compact} it follows that
    \begin{align}
        \frac{n}{\kappa^2} \le \lambda_p \le \cdots \le \lambda_1 = \opnorm{\L} \le  \kappa^2 n \qq{and, therefore,} \opnorm{\L\inv} = \frac{1}{\lambda_p} \le \frac{\kappa^2}{ n}.\FINEQ
    \end{align}
    From \cref{eq:cs-exact-recovery} we have $U\L^{1/2} = H\X Q\tr$ for a matrix $Q \in \orth{p}$. Note that $U \in \Rnp$ has orthonormal columns, therefore $\opnorm{U}=1$. For the first part, using \cref{eq:ttinf-ABC} in \cref{prop:cape-ttinf} and \ref{assumption:compact}, 
    \begin{align}
    \ttinf\big{U\L^{1/2}} = \ttinf\big{H\X Q\tr} = \ttinf\big{H\X} \opnorm\big{Q\tr} = \ttinf{H\X} \le \Rx,
    \end{align}
    and using \cref{prop:cape-ttinf} once again along with claim~\ref{lem:gram-2},
    \begin{align}
        \ttinf\big{U} = \ttinf\big{U\L^{1/2}\cdot\L^{-1/2}} \le \ttinf\big{U\L^{1/2}} \opnorm\big{\L^{-1/2}} \le \frac{\kappa\Rx}{\sqrt{n}},
    \end{align}
    which proves Claim~\ref{lem:gram-3}.
\end{proof}

For $\D = \Del + \Eps$, the next lemma provides bounds on the perturbation of the eigenvalues and the singular subspaces associated with $\Dc = -\half HDH$.
\newcommand{\keps}{K_\eps}
\begin{lemma}\label{lem:eigen}
    Suppose $\X \in \bbX(\xpar)$ satisfying \ref{assumption:compact} and $\D = \Del + \Eps$ in the noisy realizable setting of \cref{noisy-setting}. Let $\Delc=U\L U\tr$ and $\Dc = \hU\hL\hU\tr$ be the rank-$p$ spectral decompositions of $\Dc$ and $\Delc$, respectively, and let $Q_* = \argmin_{P \in \orth{p}}\smallnorm{U\tr\hU - P}_F$ be the Procrustes alignment matrix. Then, on the event $A := \qty{ \opnorm{\Dc - \Delc} \le \keps\sqrt{n} }$ for $\keps > 0$, the following hold simultaneously:
    \begin{enumerate}[label=\textup{(\roman*)}]
        \item\label{lem:eigen-1} $\opnorm\big{\hL} \lesssim  \kappa^2n$ and $\opnorm\big{\hL\inv} \lesssim \kappa^2/ n$
        \item\label{lem:eigen-2} $\opnorm\big{\hU - UQ_*} \lesssim \keps \kappa^2/ \sqrt{n}$ and $\opnorm\big{U\tr\hU - Q_*} \lesssim \keps^2\kappa^4 / n$
        \item\label{lem:eigen-3} $\opnorm\big{Q_*\hL^{-1/2} - \L^{-1/2}Q_*} \lesssim \keps \kappa^9/ n$.
    \end{enumerate}
\end{lemma}

\begin{proof}[Proof of \cref{lem:eigen}~\ref{lem:eigen-1}.]
    Using Weyl's inequality \citep[Theorem~VI.2.1]{bhatia2013matrix}, on the event $A$ we have that
    $
    \opnorm\big{\hL-\L} \le \opnorm{\Dc-\Delc} \le \keps \sqrt{n},\label{eq:weyl-hat-L}
    $
    and, using \cref{lem:gram}\,(ii), for sufficiently large $n$ it follows that
    \begin{align}
        \opnorm\big{\hL} \le \opnorm{\L} + \opnorm{\Dc-\Delc} \le  \kappa^2n + \keps\sqrt{n} \lesssim  \kappa^2n.
    \end{align}
    For $\opnorm\big{\hL}$, using \cref{lem:gram}\,\ref{lem:gram-2},
    \begin{align}
        \opnorm\big{\hL\inv} 
        = \frac{1}{\h\lambda_p}
        \le \frac{1}{\lambda_p \cdot (1 - \abs\big{\h\lambda_p - \lambda_p})} = \frac{1}{\lambda_p} \qty( 1 - {\frac{\abs\big{\h\lambda_p-\lambda_p}}{\lambda_p}} )\inv 
        &\le \frac{1}{\lambda_p} \qty( 1 - {\frac{\opnorm{\Dc-\Delc}}{\lambda_p}} )\inv\\
        &\le \frac{\kappa^2}{n} \qty( 1 - {\frac{\keps\kappa^2 \sqrt{n}}{n}} )\inv.
    \end{align}
    For sufficiently large $n$, $\keps\kappa^2/\sqrt{n} < 1$, and using the identity $(1-x)\inv = 1 + x + o(x)$ for $\abs{x} < 1$
    implies that
    \begin{align}
        \opnorm\big{\hL\inv} \le \frac{\kappa^2}{n}\qty( 1 + \frac{\keps\kappa^2}{ \sqrt{n}} + o\qty(\frac{1}{\sqrt{n}}) ) = \frac{\kappa^2}{n}\qty\big( 1 + o(1) ) \le \frac{2\kappa^2}{n}.\FINEQ
    \end{align}
    
    \noindent\textit{Proof of \cref{lem:eigen}~\ref{lem:eigen-2}.} Let $\Theta(\hU, U)$ be the $p\times p$ diagonal matrix where each diagonal entry is the principal angle between the corresponding columns of $\hU$ and $U$. Using the variant of the Davis-Kahan $\sin\Theta$ theorem by \cite{yu2015useful}, and as recast in \citet[Theorem~6.9]{cape2019two}:
    \begin{align}
        \opnorm\big{\sin\Theta(\hU, U)} \le \frac{2\opnorm{\Dc-\Delc}}{\lambda_{\text{gap}}} = \frac{2\opnorm{\Dc-\Delc}}{\lambda_p} \lesssim \frac{\keps\kappa^2}{\sqrt{n}}.\label{eq:davis-kahan}
    \end{align}
    The proof of the claim now uses \cref{eq:davis-kahan} with the singular subspace perturbation bounds from \citet[Section~6.2]{cape2019two}. In particular, \citet[Lemma~6.8]{cape2019two} yields
    \begin{align}
        \opnorm\big{\hU - UQ_*} 
        &\lesssim \opnorm\big{\sin\Theta(\hU, U)} + \opnorm\big{\sin\Theta(\hU, U)}^2 \lesssim \frac{\keps\kappa^2}{\sqrt{n}},
    \end{align}
    and using \citet[Lemma~6.7]{cape2019two},
    \begin{align}
        \opnorm\big{U\tr\hU - Q_*} 
        &\le \opnorm\big{\sin\Theta(\hU, U)}^2 \lesssim \frac{\keps^2\kappa^4}{{n}}.\FINEQ
    \end{align}
    
    \noindent\textit{Proof of \cref{lem:eigen}~\ref{lem:eigen-3}.} For $Q \in \orth{p}$ let $A := \hL$ and $B := Q\tr\L Q$. We can write
    \begin{align}
        Q\hL^{-\half} - \L^{-\half}Q = Q\hL^{-\half} - QQ\tr\L^{-\half}Q = Q\qty(\hL^{-\half} - Q\tr\L^{-\half}Q) = Q(A^{-\half} - B^{-\half}).\label{eq:Q-factoring}
    \end{align}
    Using the identity $A^{-\half} - B^{-\half} = A^{-\half}(B^{\half} - A^{\half})B^{-\half}$,
    \begin{align}
        \opnorm\big{Q\hL^{-\half} - \L^{-\half}Q} = \opnorm\big{A^{-\half} - B^{-\half}} \le \opnorm\big{A^{-\half}} \cdot \opnorm\big{B^{\half} - A^{\half}} \cdot \opnorm\big{B^{-\half}} \lesssim \frac{\kappa^2}{n}\opnorm\big{B^{\half} - A^{\half}},\label{eq:Q-2}
    \end{align}
    where, we used the fact that $\opnorm\big{A^{-\half}} \lesssim \kappa/\sqrt{n}$, from claim~\ref{lem:eigen-2} and $\opnorm\big{B^{-\half}} = \opnorm\big{\L^{-\half}} \le \kappa/\sqrt{n}$ from \cref{lem:gram}~\ref{lem:gram-2}. Moreover, by noting that 
    $
    \min\qty{\lambda_p({A}), \lambda_p({B})} = \min\qty\big{\lambda_p({\hL}), \lambda_p({\L})} \gtrsim \frac{n}{\kappa^2},
    $
    and using Theorem~X.3.8~and~Eq.~(X.46)~from~\citet{bhatia2013matrix}, we have
    \begin{align}
        \opnorm\big{B^{\half} - A^{\half}} \le \half \qty(\frac{n}{\kappa^2})^{-\half}\opnorm{B - A} .
    \end{align}
    Plugging this bound back into \cref{eq:Q-2} and using the fact that $\opnorm{B - A} = \opnorm\big{Q\hL - \L Q}$ as in \cref{eq:Q-factoring}, we get
    \begin{align}
        \opnorm\big{Q\hL^{-\half} - \L^{-\half}Q} \lesssim \frac{\kappa^3}{n^{3/2}}\opnorm\big{Q\hL - \L Q}.\label{eq:Q-3}
    \end{align}
    Finally, for the Procrustes alignment matrix $Q_* \in \orth{p}$, we can write
    \begin{align}
        Q_*\hL - \L Q_* 
        &= (Q_* - U\tr\hU + U\tr\hU)\hL - \L(Q_* - U\tr\hU + U\tr\hU)\\
        &= (Q_* - U\tr\hU)\hL - \L(Q_* - U\tr\hU) + U\tr\hU\hL - \L U\tr\hU\\
        &= (Q_* - U\tr\hU)\hL - \L(Q_* - U\tr\hU) + U\tr\Dc\hU - U\Delc\hU\\
        &= (Q_* - U\tr\hU)\hL - \L(Q_* - U\tr\hU) + U\tr(\Dc-\Delc)\hU,
    \end{align}
    where the second line follows by noting that $\Dc\hU = \hU\hL$ from the spectral decomposition of $\Dc$, and similarly for $\Delc$. By the triangle inequality, and using the bound from claim~\ref{lem:eigen-2},
    \begin{align}
        \opnorm\big{Q_*\hL - \L Q_* } &\le \opnorm\big{Q_* - U\tr\hU} \opnorm\big{\hL} + \opnorm\big{Q_* - U\tr\hU} \opnorm\big{\L} + \opnorm\big{U\tr(\Dc-\Delc)\hU}\\
        &\lesssim \qty( \frac{\keps\kappa^4}{ n} \cdot  \kappa^2n ) + \qty(  \kappa^2n \cdot \frac{\keps\kappa^4}{ n} ) + \keps\sqrt{n}\\ 
        &\lesssim {\keps\kappa^6}{\sqrt{n}}.
    \end{align}
    Plugging this back into \cref{eq:Q-3} gives $\opnorm\big{Q_*\hL - \L Q_*} \lesssim \keps \kappa^9 / n$ and completes the proof of the claim.
\end{proof}


\endgroup
\end{document}